\documentclass[a4paper,12pt,reqno]{amsart}

\usepackage[foot]{amsaddr}
\usepackage[flushmargin]{footmisc} 
\usepackage[margin=1in]{geometry}
\usepackage{hyperref}
\usepackage{amsmath,mathtools}
\usepackage{enumerate}
\usepackage{tikz}
\usepackage[all]{xy}
\usepackage{verbatim}
\usepackage{setspace}
\usepackage{enumitem}
\usepackage{mathdots}
\usepackage{longtable}
\usepackage{amssymb}
\usepackage{amsthm}
\usetikzlibrary{patterns}

\xyoption{rotate}

\DeclareMathOperator{\Hom}{Hom}
\DeclareMathOperator{\uHom}{\underline{Hom}}

\DeclareMathOperator{\End}{End}
\DeclareMathOperator{\uEnd}{\underline{End}}
\DeclareMathOperator{\Ext}{Ext}

\DeclareMathOperator{\im}{Im}
\DeclareMathOperator{\Ker}{Ker}

\DeclareMathOperator{\sgn}{sgn}
\DeclareMathOperator{\tp}{top}
\DeclareMathOperator{\rad}{rad}
\DeclareMathOperator{\soc}{soc}

\DeclareMathOperator{\Mod*}{mod}
\DeclareMathOperator{\MOD}{Mod}

\DeclareMathOperator{\Coker}{Coker}
\DeclareMathOperator{\val}{val}

\DeclareMathOperator{\id}{id}

\DeclareMathOperator{\Aut}{Aut}

\newcommand{\str}{\mathcal{S}}
\newcommand{\band}{\mathcal{B}}
\newcommand{\hstep}{\mathcal{H}_\chi}
\newcommand{\hstepone}{\hstep^1}
\newcommand{\hsteptwo}{\hstep^2}
\newcommand{\hstepthree}{\hstep^3}
\newcommand{\Mchi}{\mathcal{M}_\chi}

\newcommand{\Vchi}{\mathcal{V}_\chi}
\newcommand{\Gchi}{\mathcal{G}_\chi}
\newcommand{\Pchi}{\mathcal{P}_\chi}
\newcommand{\Wchi}{\mathcal{W}_\chi}
\newcommand{\Dchi}{\mathcal{D}_\chi}
\newcommand{\mult}{\mathfrak{m}}
\newcommand{\cycle}{\mathfrak{C}}
\newcommand{\sgerm}{\widehat{s}}
\newcommand{\tgerm}{\widehat{t}}
\newcommand{\Atilde}{\widetilde{\mathbb{A}}}
\newcommand{\Dtilde}{\widetilde{\mathbb{D}}}
\newcommand{\Etilde}{\widetilde{\mathbb{E}}}
\newcommand{\inv}{^{-1}}

\newcommand{\ualpha}{\alpha}
\newcommand{\ubeta}{\beta}

\newcommand{\ueta}{\eta}
\newcommand{\ub}{b}
\newcommand{\uw}{w}
\newcommand{\bw}{\mathbf{w}}
\newcommand{\ap}{\mathcal{A}}
\newcommand{\sap}{\underline{\mathcal{A}}}
\newcommand{\es}{\mathcal{E}}

\theoremstyle{plain}
\newtheorem{thm}{Theorem}[section]
\newtheorem*{thm*}{Theorem}
\newtheorem{lem}[thm]{Lemma}
\newtheorem{cor}[thm]{Corollary}
\newtheorem{prop}[thm]{Proposition}
\newtheorem*{prop*}{Proposition}

\theoremstyle{definition}
\newtheorem{defn}[thm]{Definition}

\newtheorem{exam}[thm]{Example}

\newtheorem*{exam*}{Example}

\theoremstyle{remark}
\newtheorem{rem}[thm]{Remark}

\numberwithin{equation}{section}
\allowdisplaybreaks

\begin{document}
	{ \small
	\title[Tubes in Symmetric Special Multiserial Algebras]{Tubes Containing String Modules in Symmetric Special Multiserial Algebras}
	\author{Drew Damien Duffield}
	\maketitle
	\begin{abstract}
Symmetric special multiserial algebras are algebras that correspond to decorated hypergraphs with orientation, called Brauer configurations. In this paper, we use the combinatorics of Brauer configurations to understand the module category of symmetric special multiserial algebras via their Auslander-Reiten quiver. In particular, we provide methods for determining the existence and ranks of tubes in the stable Auslander-Reiten quiver of symmetric special multiserial algebras using only the information from the underlying Brauer configuration. Firstly, we define a combinatorial walk around the Brauer configuration, called a Green `hyperwalk', which generalises the existing notion of a Green walk around a Brauer graph. Periodic Green hyperwalks are then shown to correspond to periodic projective resolutions of certain classes of string modules over the corresponding symmetric special multiserial algebra. Periodic Green hyperwalks thus determine certain classes of tubes in the stable Auslander-Reiten quiver, with the ranks of the tubes determined by the periods of the walks. Finally, we provide a description of additional rank two tubes in symmetric special multiserial algebras that do not arise from Green hyperwalks, but which nevertheless contain string modules at the mouth. This includes an explicit description of the space of extensions between string modules at the mouth of tubes of rank two.
	\end{abstract}

	~
	
	\setcounter{tocdepth}{1}
	\tableofcontents

	\section*{Acknowledgements}
	The author would like to thank the Leverhulme Trust, as this paper was completed whilst under the funding of the Leverhulme Trust grant RPG-2019-153. The author would also like to thank the reviewers for conducting a thorough review of the paper.
	
	\section{Introduction}
	Since their inception special biserial algebras and their various subclasses such as gentle algebras have been extremely influential. Historically, they appear in various classification theorems, such as the classification of blocks of tame representation type \cite{ErdmannString}. More recently, gentle algebras have received renewed interest due to their appearance in a wide array of modern areas of mathematical research, such as in dimer models (\cite{Broomhead}) and in cluster theory, where they appear as Jacobian algebras of associated to surfaces with boundaries and marked points (\cite{GentleTriangulation, Labardini}).
	
	There have been numerous generalisations of special biserial algebras and gentle algebras. On the one hand, there are the clannish algebras (\cite{Clans, ClanMaps}) and skewed-gentle algebras (\cite{Pena}), which are also fundamental to representation theory through their recent applications to cluster theory (see for example \cite{Qiu}). On the other hand, there are the special multiserial algebras (\cite{BCA, Multiserial,VonHohne}) and almost gentle algebras (\cite{AlmostGentle}), which themselves contain certain subclasses of clannish algebras, skewed-gentle algebras, radical cube zero algerbas (\cite{BensonRad}), and iterated tilted algebras of type $\mathbb{E}$ and $\Etilde$ and their trivial extensions (\cite{AlmostGentle, SkowDomestic}).
	
	The focus of this paper is on the representation theory of this latter class of algebras --- namely, special multiserial algebras. The representation theory of these algebras is extremely complicated, and is currently only well-understood in certain cases. For example, the subclass of special biserial algebras is tame (\cite{TameBiserial,WaldWasch}) and a classification of the indecomposable modules and morphisms between them is known in this case. This is the celebrated string and band classification appearing in various papers (for example, \cite{butlerRingel,CBMaps,TreeMaps,WaldWasch}), which is largely due to the functorial filtration method utilised by Gel'fand and Ponomarev in \cite{FuncFilt}. The class of clannish and skewed-gentle algebras are also tame with a similar classification of their indecomposable modules (\cite{Clans,Deng}) and morphisms (\cite{ClanMaps}). However in most cases, special multiserial algebras are wild in the sense of \cite{Drozd} (see \cite{DuffieldTameWild} for an account when the algebra is symmetric). Thus, a classification of the indecomposable representations is largely out of reach for many general special multiserial algebras. Thus, the question then turns to what information one can discover about special multiserial algebras in spite of their generally wild behaviour.
	
	With this goal in mind, it is useful to look to the class of special multiserial algebras that are symmetric. Not only is every special multiserial algebra a quotient of a symmetric special multiserial algebra (\cite{SMQuotient}), but one can also associate a decorated form of hypergraph with orientation to the algebra, called a Brauer configuration (\cite{BCA}). Brauer configurations can be viewed as a certain hypergraph generalisation of ribbon graphs, and in particular, as a generalisation of Brauer graphs, which appear in numerous papers that utilise symmetric special biserial algebras (such as \cite{ErdSkowSurface,BGACover,BGAExt,MarshSchroll,Roggenkamp,TrivialGentle}). Brauer configurations are incredibly useful devices in understanding the representation theory of symmetric special multiserial algebras, as the representation theory of the special multiserial algebra is encoded in the combinatorial data of the corresponding Brauer configuration. In addition, Brauer configurations provide a useful way of classifying subfamilies of symmetric special multiserial algebras that have common representation theoretic properties, as demonstrated in \cite{DuffieldTameWild}.
	
	The Auslander-Reiten quiver of an algebra is a translation quiver that serves as a means of presenting categories associated to the algebra, such as the module category. The vertices of the Auslander-Reiten quiver consist of the indecomposable objects, and the arrows are the irreducible morphisms between indecomposables, with the translate given by the Auslander-Reiten translate $\tau$. If an algebra is of infinite representation type, then the Auslander-Reiten quiver of the algebra's module category contains a number of connected components of various different shapes with respect to the Auslander-Reiten translate.
	
	Understanding the Auslander-Reiten theory of an algebra is an important goal in understanding its representation theory. The structure of the Auslander-Reiten quiver of the module categories of special biserial algebras and clannish algebras is already well-understood (\cite{butlerRingel,ClanMaps,Pena}). For gentle and skewed-gentle algebras, our understanding also extends to the derived category (\cite{GentleDerMorph,BekkertSkewGentle,BekkertGentle}). 
	
	In the case where the algebra is symmetric, there has been numerous work on the shapes of the Auslander-Reiten components in the module category of special biserial algebras (\cite{Chinburg,DuffieldBGA,ErdmannAR,GroupsWOGroups,GreenWalk}). Since symmetric special multiserial algebras generalise symmetric special biserial algebras, it makes sense to see what results from the Auslander-Reiten theory in the biserial setting carry over to the more general and complicated multiserial setting.  This is precisely the aim in the first few sections of this paper.
	
	In \cite{DuffieldBGA}, it is shown that the tubes containing string modules in the stable Auslander-Reiten quiver $_s\Gamma_A$ of the module category of a symmetric special biserial algebra $A$ are in bijective correspondence with combinatorial walks around the Brauer graph corresponding to $A$, called double-stepped Green walks (named after the walks appearing in \cite{GreenWalk}). In Section~\ref{sec:Hyperwalks}, we generalise the notion of a Green walk to Brauer configurations, where we call them Green hyperwalks. After examining the combinatorial properties of these newly defined walks, we prove a result in Section~\ref{sec:ProjResolutions} on Green hyperwalks that is analogous to the results in \cite{Roggenkamp} on the periodic projective resolutions of modules associated to Green walks. Namely, we define a distinguished set of modules $\Mchi$ in the symmetric special multiserial algebra that bijectively correspond to the steps along Green hyperwalks and show the following.
	\begin{thm*}
		Let $A$ be a symmetric special multiserial algebra corresponding to a Brauer configuration $\chi$. Consider the minimal projective resolution
		\begin{equation*}
			\mathbf{P}_\bullet\colon \cdots \rightarrow P_1 \rightarrow P_0 \rightarrow M \rightarrow 0
		\end{equation*}
		of a module $M \in \mathcal{M}_\chi$ corresponding to the zeroth step of a periodic Green hyperwalk of $\chi$. Then $\mathbf{P}_\bullet$ is periodic. Furthermore, the $i$-th syzygy $\Omega^i(M)$ is in $\Mchi$ and corresponds to the $i$-th step along the Green hyperwalk.
	\end{thm*}
	
	In Section~\ref{sec:HyperwalkTubes}, we show that the modules in $\Mchi$ associated to periodic Green hyperwalks sit at the mouth of a tube in the stable Auslander-Reiten quiver $_s\Gamma_A$ of the module category of a symmetric special multiserial algebra $A$. This allows us to generalise the main result of \cite{DuffieldBGA} on tubes to the following.
	\begin{thm*}
		Let $A$ be a connected symmetric special multiserial algebra of infinite-representation type corresponding to a Brauer configuration $\chi$.
		\begin{enumerate}[label=(\alph*)]
			\item Let $X$ be a periodic double-stepped Green hyperwalk of $\chi$ of period $n$. Then there exists a corresponding tube $\mathcal{T}_X$ of the form $\mathbb{Z}\mathbb{A}_\infty / \langle \tau^n \rangle$ in $_s\Gamma_A$.
			\item Given a tube $\mathcal{T}_X$ in $_s\Gamma_A$ corresponding to a periodic double-stepped Green hyperwalk $X$, the modules at the mouth of $\mathcal{T}_X$ are precisely the modules of $\Mchi$ corresponding to the steps of $X$.
			\item Let $X$ and $X'$ be distinct periodic double-stepped Green hyperwalks of $\chi$. Then $\mathcal{T}_X$ and $\mathcal{T}_{X'}$ are distinct tubes in $_s\Gamma_A$.
		\end{enumerate}
	\end{thm*}
	
	Unlike the symmetric special biserial setting, Green hyperwalks do not describe all of the tubes in $_s\Gamma_A$ containing string modules. One has additional tubes of rank 2 with string modules at the mouth. This occurs in many algebras that generalise the hereditary algebras of type $\Dtilde$, such as clannish algebras and skewed-gentle algebras. Section~\ref{sec:rank2} is devoted to describing these tubes in the symmetric special multiserial algebra. In particular, we define a set $\Wchi$ of strings between certain polygons in the Brauer configuration, along with involution operations $\mu_1$ and $\mu_2$ on $\Wchi$ and show the following.
	\begin{thm*}
		Let $A$ be a symmetric special multiserial algebra corresponding to a Brauer configuration $\chi$.
		\begin{enumerate}[label=(\alph*)]
			\item Let $w \in \Wchi$. Then the string module corresponding to $w$ is a mouth module of a tube of rank 2 in $_s\Gamma_A$.
			\item For any $w,w' \in \Wchi$, define an equivalence relation $w \sim w'$ if and only if $w = (\mu_1\mu_2)^i(w')$ for some $i \in\{0,1\}$. Then there exists a distinct tube $\mathcal{T}_{[w]}$ of rank 2 for each equivalence class $[w] \in {\Wchi}/{\sim}$.
		\end{enumerate}
	\end{thm*}
	
	Section~\ref{sec:rank2} briefly makes use of clannish combinatorics. Thus, a summary on the relevant aspects of clannish combinatorics required to prove the results in Section~\ref{sec:rank2} is provided in Appendix~\ref{sec:AppClans}. In particular, we explain how symmetric clannish band modules, which are typically defined over a clannish algebra (given by a quotient of a path algebra by an \emph{inadmissible ideal}), can be obtained in the module category of a path algebra modulo an \emph{admissible} ideal that is isomorphic to a clannish algebra. Certain subclasses of these clannish modules give rise to what we call \emph{hyperstrings} in special multiserial algebras (c.f. Section~\ref{sec:Hyperstrings}).
	
	\subsection{Notation}
	Throughout the paper, we let $Q=(Q_0,Q_1, s,t)$ be a quiver with vertex set $Q_0$, arrow set $Q_1$, and source and target functions $s$ and $t$. We let $K$ be an algebraically closed field and  let $I$ be an admissible ideal of the path algebra $KQ$ such that $KQ/I$ is a basic finite-dimensional algebra. We denote the stationary path in $KQ/I$ at a vertex $x$ by $\varepsilon_x$. We let $\Mod*A$ denote the category of finitely generated right modules and denote the stable Auslander-Reiten quiver of $A$ by $_s\Gamma_A$. Additional notation will be defined when needed. However for the benefit of the reader, a glossary of notation is provided in Appendix~\ref{ap:Glossary}.
	
	\section{Special Multiserial Algebras}
	As the central topic of this paper, we will begin with a review of the relevant definitions of special multiserial algebras and Brauer configurations. Our definitions are primarily sourced from \cite{Multiserial, BCA}, as well as previous works by the named author (for example, \cite{DuffieldTameWild}).
	
	\begin{defn}
		We say a (left or right) module $M$ is \emph{uniserial} if $\rad^i M / \rad^{i+1}M$ is simple or zero for all $i$, where $\rad M$ is the (left or right) Jacobson radical of $M$.
	\end{defn}
	\begin{defn}
		We say a (left or right) module $M$ is \emph{multiserial} (or more specifically, \emph{$n$-serial}) if the (left or right) Jacobson radical of the module is a sum of uniserial modules $U_1, \ldots, U_n$ such that $U_i \cap U_j$ is simple or zero for each $i \neq j$. We say an algebra $A$ is \emph{multiserial} if it is multiserial as a left and right $A$-module.
	\end{defn}
	
	The terms \emph{biserial}, \emph{triserial} and \emph{quadserial} are sometimes used for when $n=2,3,4$ in the above definition, respectively. An interesting subclass of multiserial algebras are the special multiserial algebras.
	
	\begin{defn}
		We say an algebra $A$ is \emph{special multiserial} if it is Morita equivalent to an algebra $KQ/I$ for some algebraically closed field $K$, quiver $Q$, and admissible ideal $I$ such that the following property holds.
		\begin{enumerate}[label=(SM\arabic*)]
			\item For any arrow $\beta \in Q_1$, there exists at most one arrow $\gamma \in Q_1$ such that $\beta \gamma \not \in I$ and at most one arrow $\alpha \in Q_1$ such that $\alpha\beta \not \in I$.
		\end{enumerate}
		We say $A$ is \emph{special $n$-serial} if in addition, the following is satisfied.
		\begin{enumerate}[resume, label=(SM\arabic*)]
			\item For each vertex $x\in Q_0$, there are at most $n$ arrows in $Q_1$ of source $x$ and at most $n$ arrows $Q_1$ of target $x$.
		\end{enumerate}
	\end{defn}
	
	\subsection{Brauer Configurations}
	Brauer configurations generalise the notion of Brauer graphs, a form of decorated ribbon graph. We will outline the definition in this subsection. To begin with we will define the notion of a configuration, which is equivalent to the definition of a hypergraph.
	
	\begin{defn}
		A \emph{configuration} is a tuple $\chi=(\mathcal{V}_\chi, \mathcal{G}_\chi, \mathcal{P}_\chi, \kappa)$ satisfying the following.
		\begin{enumerate}[label=(\roman*)]
			\item $\mathcal{V}_\chi$ is a finite set, whose elements are called \emph{vertices}.
			\item $\mathcal{G}_\chi$ is a finite set, whose elements are called \emph{germs of polygons} (or shortly, \emph{germs}).
			\item $\mathcal{P}_\chi$ is a finite set of subsets of $\mathcal{G}_\chi$. Each $x \in \mathcal{P}_\chi$ is called a \emph{polygon}, and more specifically an \emph{$n$-gon} if $|x|=n$. We require that $|x|\geq 2$ for all $x \in \mathcal{P}_\chi$, and that $x \cap y = \emptyset$ for any distinct $x,y \in \mathcal{P}_\chi$. Moreover, we require that $\bigcup_{x \in \mathcal{P}_\chi} x = \mathcal{G}_\chi$.
			\item $\kappa\colon \mathcal{G}_\chi \rightarrow \mathcal{V}_\chi$ is a function, which can be thought of as a map taking each germ of a polygon to an incident vertex.
		\end{enumerate}
	\end{defn}
	Essentially, we have an equivalence relation $\sim$ on $\mathcal{G}_\chi$ defined such that $g_1 \sim g_2$ if and only if $g_1,g_2 \in x$ for some $x \in \mathcal{P}_\chi$. Thus, polygons can be viewed as (defined) equivalence classes of germs. We will thus define a map $\pi\colon\mathcal{G}_\chi \rightarrow \mathcal{P}_\chi$, which maps each germ of a polygon to its respective polygon.
	
	A configuration can be thought of as a generalisation of a graph in the sense that we have vertices and connected polygons rather than vertices and connected edges. If all polygons in $\chi$ are 2-gons, then $\chi$ is indeed a graph, in which case, one could instead equivalently replace the set $\mathcal{P}_\chi$ with an involutive operation without fixed points (as is often done in the literature when defining a ribbon graph).
	
	We call a polygon $x \in \mathcal{P}_\chi$ \emph{self-folded} if there exist germs $g, g' \in x$ such that $\kappa(g)=\kappa(g')$. We define the \emph{valency} of a vertex $v \in \mathcal{V}_\chi$ to be the value
	\begin{equation*}
		\val(v) = |\{ g \in \mathcal{G}_\chi : \kappa(g)=v\}|,
	\end{equation*}
	which can be thought of as the number of germs of polygons incident to $v$.
	
	In specific examples, we will occasionally use the following notation regarding germs of polygons. If for some $x \in \mathcal{P}_\chi$ there exists a unique germ $g\in x$ such that $\kappa(g)=v$ (for some $v \in \mathcal{V}_\chi$), then we will write $g$ as $x^v$. Otherwise, if there exist precisely $m$ germs $g_1, \ldots, g_m \in x$ such that $\kappa(g_i)=v$ for all $i$ (and some $v\in \mathcal{V}_\chi$), then we will write each $g_i$ as $x^{v,i}$.
	
	We call a configuration $\chi'=(\mathcal{V}_\chi', \mathcal{G}_\chi', \mathcal{P}_\chi', \kappa')$ a \emph{subconfiguration} of a configuration $\chi=(\mathcal{V}_\chi, \mathcal{G}_\chi, \mathcal{P}_\chi, \kappa)$, written as $\chi' \subseteq \chi$, if $\mathcal{V}_\chi' \subseteq \mathcal{V}_\chi$, $\mathcal{P}_\chi' \subseteq \mathcal{P}_\chi$ and $\kappa'(g')=\kappa(g')$ for all $g' \in \mathcal{G}_\chi'$. (The fact that $\mathcal{G}_\chi' \subseteq \mathcal{G}_\chi$ follows from $\mathcal{P}_\chi' \subseteq \mathcal{P}_\chi$.)
	
	\begin{defn}
		A \emph{Brauer configuration} is a tuple $\chi=(\mathcal{V}_\chi, \mathcal{G}_\chi, \mathcal{P}_\chi, \kappa, \mathfrak{o}, \mathfrak{m})$ which satisfies the following properties.
		\begin{enumerate}[label=(\roman*)]
			\item $(\mathcal{V}_\chi, \mathcal{G}_\chi, \mathcal{P}_\chi, \kappa)$ is a configuration.
			\item $\mathfrak{o}$ is a cyclic ordering on the germs of polygons incident to each vertex, where for any distinct vertices $u,v \in \mathcal{V}_\chi$, no germ incident to $u$ is in the same cycle as a germ incident to $v$.
			\item $\mathfrak{m}\colon\mathcal{V}_\chi \rightarrow \mathbb{Z}_{>0}$ is a function called the multiplicity function.
			\item For any $x \in \mathcal{P}_\chi$ such that $|x|>2$, there exists no germ $g \in x$ such that $\val(\kappa(g))=1$ and $\mathfrak{m}(\kappa(g))=1$.
		\end{enumerate}
		A Brauer configuration that is a graph is called a \emph{Brauer graph}.
	\end{defn}
	
	Given germs $g, g' \in \mathcal{G}_\chi$ such that $\kappa(g)=\kappa(g')$, we say $g'$ is the \emph{successor} to $g$ if $g'$ directly follows $g$ in the cyclic ordering at $\kappa(g)$. In this case, we also say that $g$ is the \emph{predecessor} to $g'$. We define a map $\sigma\colon\mathcal{G}_\chi \rightarrow \mathcal{G}_\chi$ which maps each germ of a polygon to its successor under $\mathfrak{o}$. Similarly, we define $\sigma^{-1}\colon\mathcal{G}_\chi \rightarrow \mathcal{G}_\chi$ to be the map taking each germ of a polygon to its predecessor under $\mathfrak{o}$.
	
	We say a vertex $v \in \mathcal{V}_\chi$ is \emph{truncated} if $\val(v)=1$ and $\mathfrak{m}(v)=1$. We say a polygon $x \in \mathcal{P}_\chi$ is \emph{truncated} if it is incident to a truncated vertex. By property (iv) of the definition of a Brauer configuration, every truncated polygon of $\chi$ is an edge.
	
	\subsection{Brauer Configuration Algebras}
	From a Brauer configuration $\chi$, we obtain a finite-dimensional algebra $A$. In the exceptional case where $\chi$ is a graph consisting of a single non-loop edge with $\mathfrak{m} \equiv 1$, we define $A=K[x]/(x^2)$. Otherwise, $A = KQ/I$ is a quotient algebra of a path algebra determined as follows. Firstly, the vertices of $Q$ are in bijective correspondence with the polygons of $\chi$. Secondly, for each germ $g \in \mathcal{G}_\chi$ such that $\kappa(g)$ is non-truncated, there exists an arrow $\pi(g) \rightarrow \pi(\sigma(g))$ in $Q$. Thus, each non-truncated vertex in $\chi$ determines a cycle of arrows in $Q$, and no two such distinct cycles share a common arrow. For each non-truncated $v \in \mathcal{V}_\chi$, denote by $\mathfrak{C}_v$ the cycle of arrows in $Q$ (up to cyclic permutation/rotation) generated by the vertex $v$. We denote by $\mathfrak{C}_{v, \alpha}$ the specific rotation of $\mathfrak{C}_v$ such that the first arrow of the cycle is $\alpha$.
	
	We define a set of relations $\rho$ from $\chi$ in the following way. For each truncated edge $e \in \mathcal{P}_\chi$ connected to a non-truncated vertex $v \in \mathcal{V}_\chi$, we say there exists a relation $\mathfrak{C}_{v, \alpha}^{\mathfrak{m}(v)} \alpha \in \rho$, where $\alpha$ is the (unique) arrow such that $s(\alpha) = e$. For any other (non-truncated) polygon $x \in \mathcal{P}_\chi$ that is connected to some (not necessarily distinct) vertices $u,v \in \mathcal{V}_\chi$, we say there exists a relation $\mathfrak{C}_{u, \beta}^{\mathfrak{m}(u)}-\mathfrak{C}_{v, \gamma}^{\mathfrak{m}(v)} \in \rho$ for any distinct arrows $\beta$ and $\gamma$ such that $s(\beta)=x=s(\gamma)$. The remaining relations of $\rho$ consist of all length two paths $\delta\zeta$ in $Q$ such that $\delta\zeta$ is not a subpath of $\mathfrak{C}_v$ for any non-truncated vertex $v \in \mathcal{V}_\chi$.
	
	\begin{defn}
		Let $\chi$ be a Brauer configuration. Let $Q$ be the quiver and $\rho$ be the set of relations corresponding to $\chi$ given by the construction above. Let $I$ be the ideal of the path algebra $KQ$ generated by the relations $\rho$. We call the algebra $A=KQ/I$ the \emph{Brauer configuration algebra} corresponding to the Brauer configuration $\chi$.
	\end{defn}	
		
	By the work of Green and Schroll (c.f. \cite{BCA} and \cite{Multiserial}), the class of Brauer configuration algebras coincides with the class of symmetric special multserial algebras, and thus these terms may be used interchangeably.
	
	One can see from the definition of a Brauer configuration algebra that the polygons of the corresponding Brauer configuration $\chi$ are in bijection with the simple modules of the algebra. The same is true for the indecomposable projective modules. Motivated by this fact, we respectively denote by $S(x)$ and $P(x)$ the simple modules and indecomposable projective modules corresponding to a polygon $x \in \mathcal{P}_\chi$.
	
	To understand the representation theory of Brauer configuration algebras, it is sufficient to consider those that arise from connected Brauer configurations. This is because a union of disconnected Brauer configurations gives rise to a product of corresponding Brauer configuration algebras (\cite[Theorem A]{BCA}). Consequently, we will assume all Brauer configurations in this paper are connected.
	
	\section{Strings and Bands} \label{sec:StringsBands}
	String and band modules were first defined in the context of special biserial (2-serial) algebras, as they provide a complete classification of the indecomposable modules over these algebras (c.f. \cite{butlerRingel}, \cite{FuncFilt} and \cite{WaldWasch}). In the multiserial case, not all indecomposable modules are string and band modules, but these are nevertheless an incredibly useful class of indecomposable modules to work with. Since we use both strings, bands, and their respective modules later in the paper, we provide a brief outline of their definition here. Throughout, we let $A=KQ/I$.
	
	For each arrow $\alpha \in Q_1$, we define a symbol $\alpha^{-1}$ called the \emph{formal inverse} of $\alpha$, which is such that $s(\alpha^{-1})=t(\alpha)$ and $t(\alpha^{-1})=s(\alpha)$. We denote the set of all formal inverses of $Q_1$ by $Q_1^{-1}$. In the context where $A$ is a Brauer configuration algebra corresponding to a Brauer configuration $\chi$, one can view $s(\alpha)$ and $t(\alpha)$ as polygons in $\chi$ for any $\alpha \in Q_1 \cup Q^{-1}_1$. In addition, since each germ $g \in \mathcal{G}_\chi$ with $\kappa(g)$ non-truncated corresponds to precisely one arrow $\alpha_g:\pi(g) \rightarrow \pi(\sigma(g))$ in $Q$, one can define surjective maps
	\begin{align*}
		\widehat{s}\colon &Q_1 \cup Q^{-1}_1 \rightarrow \mathcal{G}_\chi \setminus \{g : \kappa(g) \text{ truncated} \} \\
		\widehat{t}\colon &Q_1 \cup Q^{-1}_1 \rightarrow \mathcal{G}_\chi \setminus \{g : \kappa(g) \text{ truncated} \}
	\end{align*}
	by $\widehat{s}(\alpha_g)=g=\widehat{t}(\alpha^{-1}_g)$ and $\widehat{t}(\alpha_g)=\sigma(g)=\widehat{s}(\alpha^{-1}_g)$. These can be viewed as the germs of polygons at the source and target of the arrow $\alpha_g$. As a remark, it is easy to see that the restrictions $\widehat{s}|_{Q_1}$, $\widehat{s}|_{Q^{-1}_1}$, $\widehat{t}|_{Q_1}$ and $\widehat{t}|_{Q^{-1}_1}$ are bijective, and thus offer more precise information regarding the source or target of an arrow or formal inverse than the functions $s$ and $t$.
	
	A \emph{word} of length $n$ is a concatenation of $n$ letters $\alpha_1\ldots\alpha_n$ such that each $\alpha_i \in Q_1 \cup Q_1^{-1}$. A \emph{string} of length $n$ is a word $w=\alpha_1\ldots\alpha_n$ such that $t(\alpha_i)=s(\alpha_{i+1})$ and $\alpha_{i} \neq \alpha_{i+1}\inv$ for each $i$ and such that $w$ avoids the relations generating $I$. We define $|w|=n$. Typically, one also defines the existence of strings of length zero (that is, strings such that $|w|=0$). We refer the reader to \cite{butlerRingel} for their precise definition, however we remark that any stationary path in a quiver may be considered as a zero string, and that non-zero strings may be concatenated with appropriate zero strings in the natural way. We denote by $\str_A$ the set of all strings associated to $A$.
	
	A subword of a string $w$ is called a \emph{substring} of $w$. Given a string $w=\alpha_1\ldots\alpha_n$, we define $w^{-1}=\alpha_n^{-1}\ldots\alpha_1^{-1}$. We naturally extend the definition of the source and target functions $s$ and $t$ to strings by setting $s(w)=s(\alpha_1)$ and $t(w)=t(\alpha_n)$. Similarly, we extend the germ source and target functions $\widehat{s}$ and $\widehat{t}$ to non-zero strings by setting $\widehat{s}(w)=\widehat{s}(\alpha_1)$ and $\widehat{t}(w)=\widehat{t}(\alpha_n)$.
	
	A string $w=\alpha_1\ldots\alpha_n$ is said to be \emph{direct} if every $\alpha_i \in Q_1$ and is said to be \emph{inverse} if every $\alpha_i \in Q_1^{-1}$. We say a direct (resp. inverse) string $w$ is \emph{maximal} if there exists no (non-zero) direct (resp. inverse) string $w'$ such that $w w'$ is a string.
	
	A \emph{band} $b=\beta_1\ldots\beta_n$ is a cyclic string such that neither $b$ (nor any cyclic permutation/rotation of $b$) is a proper power of some string $w$. We denote by $\band_A$ the set of all bands associated to $A$.
	
	For each string $w=\alpha_1\ldots\alpha_n$, there exists a corresponding indecomposable $A$-module $M(w)$ called a \emph{string module}. The vector space $M(w)$ is defined to have a canonical basis $\{b_0, \ldots, b_n\}$. This is then given an $A$-module structure by defining the action of a stationary path or arrow $\beta\in A$ on $M(w)$ by
	\begin{equation*}
		b_i \beta =
		\begin{cases}
			b_{i-1}	&	\text{if } \beta = \alpha_i\inv					\\
			b_i			&	\text{if } \beta = \varepsilon_{t(\alpha_i)} \text{ or } \beta = \varepsilon_{s(\alpha_{i+1})}	\\
			b_{i+1}	&	\text{if } \beta = \alpha_{i+1}					\\
			0			&	\text{otherwise},
		\end{cases}
	\end{equation*}
	where $\varepsilon_{t(\alpha_i)}$ and $\varepsilon_{s(\alpha_{i+1})}$ denote the stationary path in $KQ$ at the vertex $t(\alpha_i)=s(\alpha_{i+1})$. In addition, for each band $b$, one obtains an infinite family of indecomposable $A$-modules $M(b,m,\phi)$ called \emph{band modules}, where $m \in \mathbb{Z}_{>0}$ and $\phi \in \Aut(K^m)$. We refer the reader to \cite{butlerRingel} for the precise construction of band modules, however we remark that $M(b,m,\phi) \cong M(b',m,\phi)$ for any rotation $b'$ of a band $b$.
	
	Maps between string and modules were completely described by Crawley-Boevey and Krause in \cite{CBMaps} and \cite{TreeMaps}. We will summarise the results for string modules here. Consider strings $w=\alpha_1\ldots\alpha_m$ and $w'=\beta_1\ldots\beta_n$. A substring $\alpha_i\alpha_{i+1}\ldots \alpha_j$ of $w$ is called a \emph{factor substring} of $w$ if
	\begin{enumerate}[label=(\roman*)]
		\item either $i=1$ or $\alpha_{i-1} \in Q\inv_1$; and
		\item either $j=m$ or $\alpha_{j+1} \in Q_1$.
	\end{enumerate}
	A substring $\beta_k\beta_{k+1}\ldots\beta_l$ of $w'$ is called an \emph{image substring} of $w'$ if
	\begin{enumerate}[label=(\roman*)]
		\item either $k=1$ or $\beta_{k-1} \in Q_1$; and
		\item either $l=n$ or $\beta_{l+1} \in Q\inv_1$.
	\end{enumerate}
	Given a pair of strings $(w,w')$ with substrings $w_+$ and $w_-$ of $w$ and $w'$ respectively, the pair of substrings $(w_+,w_-)$ is said to be \emph{admissible with respect to} $(w,w')$ if
	\begin{enumerate}[label=(\roman*)]
		\item $w_+$ is a factor substring of $w$,
		\item $w_-$ is an image substring of $w'$, and
		\item $w_+=w_-$ or $w_+ = w_-\inv$.
	\end{enumerate}
	A basis of $\Hom_A(M(w),M(w'))$ is indexed by the set
	\begin{equation*}
		\{(w_+,w_-) : (w_+,w_-) \text{ admissible with respect to } (w,w')\}.
	\end{equation*}
	Given any strings $w,w' \in \str_A$, we will denote the set of admissible pairs of $(w,w')$ by $\ap(w,w')$. If $f \in \Hom_A(M(w),M(w'))$ corresponds to $(w_+,w_-) \in \ap(w,w')$, then $\im f \cong M(w_-)$ and $M(w) / \Ker f \cong M(w_+)$. In particular, if $w=\alpha_1\ldots\alpha_m$ and $w'=\beta_1\ldots\beta_n$, and $M(w)$ and $M(w')$ have canonical bases $\{b_0, \ldots, b_m\}$ and $\{b'_0, \ldots, b'_n\}$ respectively, then the map $f \in \Hom_A(M(w),M(w'))$ corresponding to
	\begin{equation*}
		(\alpha_i\ldots \alpha_{i+j},\beta_k\ldots\beta_{k+j}) \in \ap(w,w')
	\end{equation*}
	is such that 
	\begin{equation*}
		f(b_r)=
		\begin{cases}
			b'_{r-i+k}	&	\text{if } i-1 \leq r \leq i+j, \\
			0			&	\text{otherwise.}
		\end{cases}
	\end{equation*}
	
	\section{Green Hyperwalks in Brauer Configurations} \label{sec:Hyperwalks}
	In this section, we will define a combinatorial device called a Green hyperwalk. This is a hyper-ribbon graph generalisation of the incredibly powerful Green walk (named after J. A. Green), defined in \cite{GreenWalk} for Brauer trees and used in \cite{DuffieldBGA} and \cite{Roggenkamp} for Brauer graphs. We will see in the next section that this generalisation has similar applications to symmetric special multiserial algebras as the Green walks do to symmetric special biserial algebras. Since this is a rather combinatorial construction, we will provide numerous examples throughout to help provide the reader with a good intuition of the definitions and subsequent results. As for notation, $\chi=(\mathcal{V}_\chi, \mathcal{G}_\chi, \mathcal{P}_\chi, \kappa, \mathfrak{o}, \mathfrak{m})$ is taken to be a connected Brauer configuration for the entirety of this section.
	\begin{defn} \label{def:HyperwalkSteps}
		We say a nonempty proper subset $X\subset\mathcal{G}_\chi$ is a \emph{Green hyperwalk step} of $\chi$ if $X$ satisfies precisely one of the following properties.
		\begin{enumerate}[label=(G\arabic*)]
			\item $|X|=1$
			\item $|X|=2$, $X \subset x$ is a proper subset for some $x \in \mathcal{P}_\chi$, and for both $g \in X$, we have $\mathfrak{m}(\kappa(g))=1$, $\val(\kappa(g))=2$ and that $\pi(\sigma(g))$ is a truncated edge.
			\item $X= \{g, g'\}$, where $\{\sigma(g),\sigma(g')\}$ satisfies (G2).
		\end{enumerate}
	\end{defn}
	
	We denote by $\hstep$ the set of all Green hyperwalk steps of $\chi$. Trivially, if an element of $\hstep$ is of type (G1), then $X \subset x$ for some $x \in \mathcal{P}_\chi$. We also remark that if $X=\{g,g'\}$ satisfies (G3), then necessarily $\pi(g)$ and $\pi(g')$ are distinct truncated edges that are both connected to a polygon $x$ via (distinct) vertices $\kappa(g)$ and $\kappa(g')$ respectively. It is easy to deduce from this that the elements of $\hstep$ of type (G2) are in bijective correspondence with the elements of type (G3). It is also easy to see that (G1), (G2) and (G3) are mutually exclusive properties. Thus, we will denote by $\hstep^1$, $\hstep^2$ and $\hstep^3$ the disjoint subsets of $\hstep$ consisting of elements of type (G1), (G2) and (G3) respectively.
	
	\begin{exam}
		Consider the Brauer configuration $\chi$ given by
		
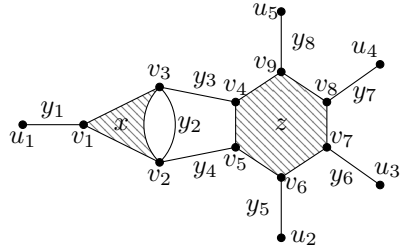
\begin{center}
			\begin{tikzpicture}
			\draw [fill=black] (-1.3,0) ellipse (0.05 and 0.05);
			\draw [fill=black] (0.5,0.5) ellipse (0.05 and 0.05);
			\draw [fill=black] (0.5,-0.5) ellipse (0.05 and 0.05);
			\draw [fill=black] (-0.5,0) ellipse (0.05 and 0.05);
			\draw [pattern=north west lines,pattern color=gray](0.5,0.5) -- (-0.5,0) -- (0.5,-0.5) to [bend left=45] (0.5,0.5);
			\draw (0.5,0.5) to [bend left=45] (0.5,-0.5);
			\draw (-1.3,0) -- (-0.5,0);
			\draw (0.5,0.5) -- (1.5,0.3);
			\draw (0.5,-0.5) -- (1.5,-0.3);
			\draw [pattern=north west lines,pattern color=gray](1.5,0.3) -- (1.5,-0.3) -- (2.1,-0.7) -- (2.7,-0.3) -- (2.7,0.3) -- (2.1,0.7)-- (1.5,0.3);
			\draw [fill=black] (1.5,0.3) ellipse (0.05 and 0.05);
			\draw [fill=black] (2.7,-0.3) ellipse (0.05 and 0.05);
			\draw [fill=black] (2.1,-0.7) ellipse (0.05 and 0.05);
			\draw [fill=black] (1.5,-0.3) ellipse (0.05 and 0.05);
			\draw (3.4,-0.8) -- (2.7,-0.3);
			\draw (2.1,-0.7) -- (2.1,-1.5);
			\draw [fill=black] (3.4,-0.8) ellipse (0.05 and 0.05);
			\draw [fill=black] (2.1,-1.5) ellipse (0.05 and 0.05);
			\draw [fill=black] (2.7,0.3) ellipse (0.05 and 0.05);
			\draw [fill=black] (3.4,0.8) ellipse (0.05 and 0.05);
			\draw [fill=black] (2.1,0.7) ellipse (0.05 and 0.05);
			\draw [fill=black] (2.1,1.5) ellipse (0.05 and 0.05);
			\draw (2.1,1.5) -- (2.1,0.7);
			\draw (2.7,0.3) -- (3.4,0.8);
			
			\draw (0,0) node {\footnotesize $x$};
			\draw (2.1,0) node {\footnotesize $z$};
			\draw (-0.9,0.2) node {\footnotesize $y_1$};
			\draw (0.9,0) node {\footnotesize $y_2$};
			\draw (1.1,0.6) node {\footnotesize $y_3$};
			\draw (1.1,-0.6) node {\footnotesize $y_4$};
			\draw (1.8,-1.1) node {\footnotesize $y_5$};
			\draw (2.9,-0.7) node {\footnotesize $y_6$};
			\draw (3.2,0.4) node {\footnotesize $y_7$};
			\draw (2.4,1.1) node {\footnotesize $y_8$};
			\draw (-0.5,-0.2) node {\footnotesize $v_1$};
			\draw (0.5,-0.7) node {\footnotesize $v_2$};
			\draw (0.5,0.7) node {\footnotesize $v_3$};
			\draw (1.5,0.5) node {\footnotesize $v_4$};
			\draw (1.5,-0.5) node {\footnotesize $v_5$};
			\draw (2.3,-0.8) node {\footnotesize $v_6$};
			\draw (2.9,-0.2) node {\footnotesize $v_7$};
			\draw (2.7,0.5) node {\footnotesize $v_8$};
			\draw (1.9,0.8) node {\footnotesize $v_9$};
			\draw (-1.3,-0.2) node {\footnotesize $u_1$};
			\draw (2.4,-1.5) node {\footnotesize $u_2$};
			\draw (3.5,-0.6) node {\footnotesize $u_3$};
			\draw (3.2,1) node {\footnotesize $u_4$};
			\draw (1.9,1.5) node {\footnotesize $u_5$};
			\end{tikzpicture}
		\end{center}
		with $\mathfrak{m}(u_3),\mathfrak{m}(v_6)>1$ and $\mathfrak{m}(t)=1$ for any vertex $t \neq u_3,v_6$. Thus, the edges $y_1$, $y_7$ and $y_8$ are truncated and all other polygons are non-truncated. The subsets $\{x^{v_1}\}$, $\{y_2^{v_3}\}$ and $\{z^{v_9}\}$ of $\mathcal{G}_\chi$ are examples of elements in $\hstep^1$. There is only one element in $\hstep^2$, which is $\{z^{v_8}, z^{v_9}\}$. Similarly, there is only one element in $\hstep^3$, which corresponds to the unique $\hstep^2$-step by moving to the next germs in the successor sequence around the vertices $v_8$ and $v_9$. Namely, this is the Green hyperwalk step $\{y_7^{v_8}, y_8^{v_9}\}$.
		
		As for non-examples, it is trivially the case that any subset $X \subset \mathcal{G}_\chi$ such that $|X|>1$ is not a Green hyperwalk step of type (G1), and it is not a Green hyperwalk step at all if $|X|>2$. The following subsets of $\mathcal{G}_\chi$ are not of type (G2) or (G3) for different reasons.
		\begin{enumerate}[label=(\roman*)]
			\item $\{x^{v_1}, x^{v_2}\}$, since $\val(v_2) \neq 2$.
			\item $\{z^{v_6}, z^{v_8}\}$, since $\mathfrak{m}(v_6)>1$.
			\item $\{z^{v_7}, z^{v_8}\}$, since $\pi(\sigma(z^{v_7}))=y_6$ is not a truncated edge (so $\{z^{v_7}, z^{v_8}\} \not\in \hstep^2$) and neither is $\pi(\sigma^2(z^{v_7}))=\pi(\sigma^2(z^{v_8}))=z$, which implies that $\{\sigma(z^{v_7}),\sigma(z^{v_8})\} \not\in \hstep^2$  (so $\{z^{v_7}, z^{v_8}\} \not\in \hstep^3$).
			\item $\{y_1^{u_1}, y_1^{v_1}\}$, since (amongst other reasons) this is not a proper subset of a polygon (so not in $\hstep^2$), and neither is $\{\sigma(y_1^{u_1}), \sigma(y_1^{v_1})\}$ (so not in $\hstep^3$).
		\end{enumerate}
	\end{exam}

	\begin{defn} \label{HyperwalkDefn}
		Let $X_0 \in \hstep$. A \emph{Green hyperwalk} of $\chi$ from $X_0$ is a sequence $(X_i)_{i \in \mathbb{Z}_{\geq 0}}$ of subsets of $\mathcal{G}_\chi$ defined as follows. If $X_i=\emptyset$ then define $X_{i+1}=\emptyset$. Otherwise, we proceed with the following construction.
		\begin{enumerate}[label=(\roman*)]
			\item If $X_i \in \hstepone \cup \hsteptwo$ with $X_i \subset x \in \mathcal{P}_\chi$, then define a set $Y_i=\{\sigma(g): g \in x \setminus X_i\}$.
			\item If $X_i\in \hstepthree$, then define a set $Y_i=\{\sigma(g) : g \in X_i\}$.
		\end{enumerate}
		In either case, if $Y_i \in \hstep$, then define $X_{i+1}=Y_i$. Otherwise, define $X_{i+1} = \emptyset$.
		
		A Green hyperwalk $(X_i)_{i \in \mathbb{Z}_{\geq 0}}$ is said to \emph{terminate} if there exists an integer $k>0$ such that $X_k= \emptyset$. A Green hyperwalk is said to be \emph{periodic} if there exists an integer $k$ such that $X_i = X_{i+k}$ for all $i \geq 0$. The \emph{period} of a periodic Green hyperwalk is the minimal value of $k$ such that $X_i = X_{i+k}$ for all $i \geq 0$.
	\end{defn}
	
	\begin{rem} \label{GWPeriodic}
		A Green hyperwalk step $X \in \hstep^2 \cup \hstep^3$ necessarily implies the existence of an $n$-gon with $n>2$. Thus if $\chi$ is a Brauer graph, then any Green hyperwalk step $X$ is such that $X \in \hstep^1$. In this case, every Green hyperwalk is in fact a Green walk (in the sense of \cite{GreenWalk}) and hence periodic (c.f. \cite{Roggenkamp}).
	\end{rem}
	
	\begin{rem}	\label{GWG3}
		Given a Green hyperwalk $(X_i)$, if there exists an integer $j$ such that $X_j \in \hstep^3$, then we necessarily have $X_{j+1} \neq \emptyset$. This is a trivial consequence of the definitions.
	\end{rem}
	
	\begin{exam} \label{ex:GWs}
		Let $\chi$ be the Brauer configuration given by
		
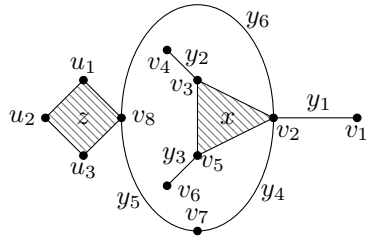
\begin{center}
			\begin{tikzpicture}
			\draw [pattern=north west lines, pattern color=gray](-2,-0.5) -- (-2,0.5) -- (-1,0) -- (-2,-0.5);
			\draw (-2,0.5) -- (-2.4,0.9);
			\draw (-2,-0.5) -- (-2.4,-0.9);
			\draw  (-2,0) ellipse (1 and 1.5);
			\draw [pattern=north west lines,pattern color=gray](-3,0) -- (-3.5,0.5) -- (-4,0) -- (-3.5,-0.5) -- (-3,0);
			\draw (-1,0) -- (0.1,0);
			\draw [fill=black] (0.1,0) ellipse (0.05 and 0.05);
			\draw [fill=black] (-1,0) ellipse (0.05 and 0.05);
			\draw [fill=black] (-2,0.5) ellipse (0.05 and 0.05);
			\draw [fill=black] (-2.4,0.9) ellipse (0.05 and 0.05);
			\draw [fill=black] (-2,-0.5) ellipse (0.05 and 0.05);
			\draw [fill=black] (-2.4,-0.9) ellipse (0.05 and 0.05);
			\draw [fill=black] (-3,0) ellipse (0.05 and 0.05);
			\draw [fill=black] (-3.5,0.5) ellipse (0.05 and 0.05);
			\draw [fill=black] (-3.5,-0.5) ellipse (0.05 and 0.05);
			\draw [fill=black] (-4,0) ellipse (0.05 and 0.05);
			\draw [fill=black] (-2,-1.5) ellipse (0.05 and 0.05);
			\draw (-1.6,0) node {\footnotesize $x$};
			\draw (-0.4,0.2) node {\footnotesize $y_1$};
			\draw (-2,0.8) node {\footnotesize $y_2$};
			\draw (-2.3,-0.5) node {\footnotesize $y_3$};
			\draw (-1.2,1.3) node {\footnotesize $y_6$};
			\draw (-1,-1) node {\footnotesize $y_4$};
			\draw (-2.9,-1.1) node {\footnotesize $y_5$};
			\draw (-3.5,0) node {\footnotesize $z$};
			\draw (-3.5,0.7) node {\footnotesize $u_1$};
			\draw (-4.3,0) node {\footnotesize $u_2$};
			\draw (-3.5,-0.7) node {\footnotesize $u_3$};
			\draw (0.1,-0.2) node {\footnotesize $v_1$};
			\draw (-0.8,-0.2) node {\footnotesize $v_2$};
			\draw (-2.2,0.4) node {\footnotesize $v_3$};
			\draw (-2.5,0.7) node {\footnotesize $v_4$};
			\draw (-1.8,-0.6) node {\footnotesize $v_5$};
			\draw (-2.1,-1) node {\footnotesize $v_6$};
			\draw (-2,-1.3) node {\footnotesize $v_7$};
			\draw (-2.7,0) node {\footnotesize $v_8$};
			\end{tikzpicture}
		\end{center}
		with $\mult(u_i)>1$ for all $i$ and $\mult(v_i)=1$ for $3 \leq i \leq 6$. The Green hyperwalk starting from $X_0=\{y_4^{v_2}\}$ is periodic. Specifically, the sequence is given by
		\begin{equation*}
			(\{y_4^{v_2}\}, \{y_5^{v_7}\}, \{y_6^{v_8}\}, \{x^{v_2}\}, \{y_2^{v_3},y_3^{v_5}\}, \{x^{v_3},x^{v_5}\}, \{y_4^{v_2}\},\{y_5^{v_7}\}, \{y_6^{v_8}\},\ldots).
		\end{equation*}
		This sequence may be presented diagramatically, as in Figure~\ref{GreenHyperwalkEg}. An example of a terminating Green hyperwalk is given by $X_0=\{y_1^{v_2}\}$. The sequence is then
		\begin{equation*}
			(\{y_1^{v_2}\}, \{y_1^{v_1}\}, \{y_6^{v_2}\}, \{z^{v_8}\}, \emptyset,\ldots).
		\end{equation*}
		Following the notation of Definition~\ref{HyperwalkDefn}, the reason why the 5th step is the empty set is because $Y_5=\{z^{u_1},z^{u_2},z^{u_3}\} \not\in \hstep$.
		\begin{figure}[h] 
			\centering
			\usetikzlibrary{patterns}
			\begin{tikzpicture}
				\draw [pattern=north west lines, pattern color=gray](-2,-0.5) -- (-2,0.5) -- (-1,0) -- (-2,-0.5);
				\draw (-2,0.5) -- (-2.4,0.9);
				\draw (-2,-0.5) -- (-2.4,-0.9);
				\draw  (-2,0) ellipse (1 and 1.5);
				\draw [pattern=north west lines, pattern color=gray](-3,0) -- (-3.5,0.5) -- (-4,0) -- (-3.5,-0.5) -- (-3,0);
				\draw (-1,0) -- (0.1,0);
				\draw [fill=black] (0.1,0) ellipse (0.05 and 0.05);
				\draw [fill=black] (-1,0) ellipse (0.05 and 0.05);
				\draw [fill=black] (-2,0.5) ellipse (0.05 and 0.05);
				\draw [fill=black] (-2.4,0.9) ellipse (0.05 and 0.05);
				\draw [fill=black] (-2,-0.5) ellipse (0.05 and 0.05);
				\draw [fill=black] (-2.4,-0.9) ellipse (0.05 and 0.05);
				\draw [fill=black] (-3,0) ellipse (0.05 and 0.05);
				\draw [fill=black] (-3.5,0.5) ellipse (0.05 and 0.05);
				\draw [fill=black] (-3.5,-0.5) ellipse (0.05 and 0.05);
				\draw [fill=black] (-4,0) ellipse (0.05 and 0.05);
				\draw [fill=black] (-2,-1.5) ellipse (0.05 and 0.05);
				\draw [->,teal](-1.1064,0.2805) arc (110.7729:166.1538:0.3);
				\draw [->,blue](-1.8753,0.3436) arc (-51.4341:105:0.2);
				\draw [->,blue](-1.9269,-0.3138) arc (68.5656:195:0.2);
				\draw [->,red](-2.2,0.55) arc (165.9647:274.2857:0.2061);
				\draw [->,red](-2.05,-0.7) arc (-104.0366:34.2857:0.2062);
				\draw [->,teal](-1.2913,-0.0718) arc (-166.1537:-110.7692:0.3);
				\draw [->,teal](-1.8229,-1.4071) arc (27.6798:152.3077:0.2);
				\draw [->,teal](-2.8936,-0.2805) arc (-69.2271:69.2308:0.3);
				\draw [teal] (-1.02,-0.3) ellipse (0.05 and 0.05);
				\draw [teal] (-2.21,-1.47) ellipse (0.05 and 0.05);
				\draw [teal] (-2.98,0.3) ellipse (0.05 and 0.05);
				\draw [teal] (-1.3,0.01) ellipse (0.05 and 0.05);
				\draw [blue] (-2.15,0.65) ellipse (0.05 and 0.05);
				\draw [blue] (-2.15,-0.65) ellipse (0.05 and 0.05);
				\draw [red] (-1.85,-0.3) ellipse (0.05 and 0.05);
				\draw [red] (-1.9,0.25) ellipse (0.05 and 0.05);
				\draw [teal] (-0.85,-0.45) node {\footnotesize $0$};
				\draw [teal] (-2.25,-1.2) node {\footnotesize $1$};
				\draw [teal] (-2.8,0.45) node {\footnotesize $2$};
				\draw [teal] (-1.3,0.3) node {\footnotesize $3$};
				\draw [blue] (-2.1,0.9) node {\footnotesize $4$};
				\draw [blue] (-2.35,-0.6) node {\footnotesize $4$};
				\draw [red] (-2,0.1) node {\footnotesize $5$};
				\draw [red] (-1.65,-0.3) node {\footnotesize $5$};
				\draw [teal] (-1.2,-0.45) node {\footnotesize $6$};
				
				\draw [teal, anchor = west] (1.2,1.5) node {\footnotesize $X_0= \{y_4^{v_2}\}$};
				\draw [teal, anchor = west] (1.2,1) node {\footnotesize $X_1= \{y_5^{v_7}\}$};
				\draw [teal, anchor = west] (1.2,0.5) node {\footnotesize $X_2= \{y_6^{v_8}\}$};
				\draw [teal, anchor = west] (1.2,0) node {\footnotesize $X_3= \{x^{v_2}\}$};
				\draw [blue, anchor = west] (1.2,-0.5) node {\footnotesize $X_4= \{y_2^{v_3},y_3^{v_5}\}$};
				\draw [red, anchor = west] (1.2,-1) node {\footnotesize $X_5=  \{x^{v_3},x^{v_5}\}$};
				\draw [teal, anchor = west] (1.2,-1.5) node {\footnotesize $X_6= \{y_4^{v_2}\}$};
			\end{tikzpicture}
			\caption{Left: An example of a periodic Green hyperwalk with the germs at each step circled. Each number corresponds to the step number. Each arrow is to be interpreted as `walking onto' the next germ. Steps are colour-coded such that Green hyperwalk steps in $\hstepone$, $\hsteptwo$ and $\hstepthree$ are respectively coloured teal, red and blue. Right: The Green hyperwalk steps given in the notation of Definition~\ref{def:HyperwalkSteps}, where $X_i$ corresponds to the step labelled $i$ in the left diagram. See Example~\ref{ex:GWs} for the corresponding labelling of the vertices and polygons.} \label{GreenHyperwalkEg}
		\end{figure}
	\end{exam}
	
	Since the zeroth step of any Green hyperwalk is nonempty, no periodic Green hyperwalk terminates. In fact, the next few results show that the converse is true. Namely, that every non-periodic Green hyperwalk is one which terminates.
	
	
	\begin{lem} \label{QuadserialWalk}
		Let $(X_j)_{j \in \mathbb{Z}_{\geq 0}}$ be a non-terminating Green hyperwalk in a Brauer configuration $\chi$. Then $|\pi (g)| \leq 4$ for any $g \in X_i$ and any $i\geq 0$. In particular, if there exists a germ $g \in X_i$ for some $i \geq 0$ such that $|\pi (g)| = 4$, then $\chi$ is the Brauer configuration
		
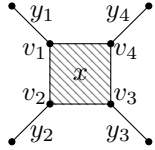
\begin{center}
			\begin{tikzpicture}
				\draw[pattern = north west lines, pattern color=gray] (0.6,-0.4) -- (0.6,0.4) -- (1.4,0.4) -- (1.4,-0.4) -- (0.6,-0.4);
				\draw (0.1,0.9) -- (0.6,0.4);
				\draw (0.1,-0.9) -- (0.6,-0.4);
				\draw (1.4,0.4) -- (1.9,0.9);
				\draw (1.4,-0.4) -- (1.9,-0.9);
				
				\draw (1,0) node {\footnotesize$x$};
				\draw (0.5,0.8) node {\footnotesize$y_1$};
				\draw (0.5,-0.8) node {\footnotesize$y_2$};
				\draw (1.5,-0.8) node {\footnotesize$y_3$};
				\draw (1.5,0.8) node {\footnotesize$y_4$};
				\draw (0.4,0.3) node {\footnotesize$v_1$};
				\draw (0.4,-0.3) node {\footnotesize$v_2$};
				\draw (1.6,-0.3) node {\footnotesize$v_3$};
				\draw (1.6,0.3) node {\footnotesize$v_4$};
				\draw [fill=black] (0.1,0.9) ellipse (0.04 and 0.04);
				\draw [fill=black] (0.6,0.4) ellipse (0.04 and 0.04);
				\draw [fill=black] (0.1,-0.9) ellipse (0.04 and 0.04);
				\draw [fill=black] (0.6,-0.4) ellipse (0.04 and 0.04);
				\draw [fill=black] (1.9,-0.9) ellipse (0.04 and 0.04);
				\draw [fill=black] (1.4,-0.4) ellipse (0.04 and 0.04);
				\draw [fill=black] (1.9,0.9) ellipse (0.04 and 0.04);
				\draw [fill=black] (1.4,0.4) ellipse (0.04 and 0.04);
			\end{tikzpicture}
		\end{center}
		with $\mathfrak{m} \equiv 1$ and $(X_j)$ is of one of the following periodic forms.
		\begin{enumerate}
			\item $(\{x^{v_j}, x^{v_k}\}, \{y_l^{v_l}, y_m^{v_m}\}, \{x^{v_l}, x^{v_m}\}, \{y_j^{v_j}, y_k^{v_k}\},\{x^{v_j}, x^{v_k}\},\{y_l^{v_l}, y_m^{v_m}\},\ldots)$ for some pairwise distinct integers $j,k,l,m \in \{1,2,3,4\}$.
			\item $(\{y_j^{v_j}, y_k^{v_k}\},\{x^{v_j}, x^{v_k}\}, \{y_l^{v_l}, y_m^{v_m}\}, \{x^{v_l}, x^{v_m}\},\{y_j^{v_j}, y_k^{v_k}\},\{x^{v_j}, x^{v_k}\}, \ldots)$ for some pairwise distinct integers $j,k,l,m \in \{1,2,3,4\}$.
		\end{enumerate}
	\end{lem}
	\begin{proof}
		We will first show that if $(X_j)$ is non-terminating, then $|\pi (g)| \leq 4$ for any $g \in X_i$ and any $i\geq 0$. So suppose for a contradiction that there exists $g \in X_i$ for some $i \geq 0$ such that $|\pi(g)|>4$. Then $X_i \in \hstepone \cup \hsteptwo$. Following the procedure of Definition~\ref{HyperwalkDefn}, we know that either $X_{i+1} = \emptyset$ or $|X_{i+1}|=|\pi(g) \setminus X_i|$. But $|\pi(g) \setminus X_i|>2$ and thus $X_{i+1} \not\in \hstep$. So $X_{i+1} = \emptyset$, which is a contradiction to the assumption that $(X_j)$ is non-terminating. Hence if $(X_j)$ is a non-terminating Green hyperwalk, then $|\pi (g)| \leq 4$ for any $g \in X_i$ and any $i\geq 0$, as required.
		
		Now suppose that $(X_j)$ is a non-terminating Green hyperwalk such that there exists a germ $g \in X_i$ for some $i \geq 0$ with $|\pi (g)| = 4$. We will show that $\chi$ must be of the form given in the lemma statement. We first note that $X_i\not\in\hstepone$, since then $|\pi(g) \setminus X_i|=3$, which would imply that $X_{i+1} = \emptyset$, contradicting the assumption that $(X_j)$ is non-terminating. Trivially, we also note that $X_i \not \in \hstepthree$, since $\pi(g_i)$ is not a truncated edge. So $X_i \in \hsteptwo$.
		
		Write $X_i = \{g_i, g'_i\}$ and $x=\pi(g_i)$. Since $X_i \in \hsteptwo$, it follows that $\val(\kappa(g_i))=\val(\kappa(g'_i))=2$ and $\pi(\sigma(g_i))$ and $\pi(\sigma(g'_i))$ are truncated edges. This implies that $x$ cannot be self-folded at either of the vertices $\kappa(g_i)$ and $\kappa(g'_i)$. In particular, there is only one polygon distinct from $x$ connected to $\kappa(g_i)$, which is a truncated edge, and likewise for the vertex $\kappa(g'_i)$. We also have $\mathfrak{m}(\kappa(g_i))=\mathfrak{m}(\kappa(g'_i))=1$.
		
		We will now investigate $X_{i+1}$ and the implications this has on the structure of $\chi$. Since $(X_j)$ non-terminating and $|x \setminus X_i|=2$, it follows that $X_{i+1}=\{g_{i+1},g'_{i+1}\} \in \hsteptwo \cup \hstepthree$. Necessarily, this means that $\val(\kappa(g_{i+1}))= \val(\kappa(g'_{i+1}))=2$ and $\mathfrak{m}(\kappa(g_{i+1}))=\mathfrak{m}(\kappa(g'_{i+1}))=1$. Since $\pi(\sigma(g_{i+1})) = \pi(\sigma(g_{i+1}')) = x$ is not truncated, we have $X_{i+1} \in \mathcal{H}_\chi^3$. This also implies that $\pi(g_{i+1})$ and $\pi(g'_{i+1})$ must be distinct truncated edges.
		
		A consequence of the above is that $\kappa(g_i)$, $\kappa(g'_i)$, $\kappa(g_{i+1})$ and $\kappa(g'_{i+1})$ are pairwise distinct vertices of $x$, since $x$ cannot be self-folded at any of these vertices. In addition, each of these vertices has multiplicity 1, valency 2, and is attached to precisely one truncated edge. Thus, $\chi$ is precisely the Brauer configuration given in the lemma statement.
		
		To show that $(X_j)$ must be either of the form (1) or (2) in the lemma statement, we need only calculate a few terms of the Green hyperwalk. If we write $X_i=\{x^{v_j}, x^{v_k}\}$ for some distinct $j,k \in \{1,2,3,4\}$, then one can easily verify that
		\begin{align*}
			X_{i+1} &= \{y_l^{v_l}, y_m^{v_m}\}	&	X_{i+2} &= \{x^{v_l}, x^{v_m}\}	\\
			X_{i+3} &= \{y_j^{v_j}, y_k^{v_k}\}		&	X_{i+4} &= X_i.
		\end{align*}
		where $l,m \in \{1,2,3,4\} \setminus \{j,k\}$. Setting either $X_0= X_i$ or $X_0 = X_{i+3}$ gives the required result.
	\end{proof}
	
	\begin{lem} \label{TriserialWalk}
		Let $(X_j)_{j \in \mathbb{Z}_{\geq 0}}$ be a non-terminating Green hyperwalk in a Brauer configuration $\chi$. Suppose there exists a germ $g \in X_i$ for some $i \geq 0$ such that $|\pi (g)| = 3$. Let $x=\pi(g)$. Then $\chi$ is of the form
		
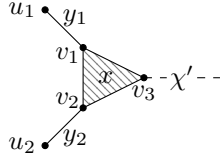
\begin{center}
			\begin{tikzpicture}
				\draw[pattern = north west lines, pattern color=gray] (0.6,-0.4) -- (0.6,0.4) -- (1.4,0) -- (1.4,0) -- (0.6,-0.4);
				\draw (0.1,0.9) -- (0.6,0.4);
				\draw (0.1,-0.9) -- (0.6,-0.4);
				\draw [dashed](1.4,0) -- (1.7,0);
				\draw [dashed](2.1,0) -- (2.4,0);
				
				\draw (0.9,0) node {\footnotesize$x$};
				\draw (0.5,0.8) node {\footnotesize$y_1$};
				\draw (0.5,-0.8) node {\footnotesize$y_2$};
				\draw (1.9,0) node {\footnotesize$\chi'$};
				\draw (-0.2,0.9) node {\footnotesize$u_1$};
				\draw (0.4,0.3) node {\footnotesize$v_1$};
				\draw (-0.2,-0.9) node {\footnotesize$u_2$};
				\draw (0.4,-0.3) node {\footnotesize$v_2$};
				\draw (1.4,-0.2) node {\footnotesize$v_3$};
				\draw [fill=black] (0.1,0.9) ellipse (0.04 and 0.04);
				\draw [fill=black] (0.6,0.4) ellipse (0.04 and 0.04);
				\draw [fill=black] (0.1,-0.9) ellipse (0.04 and 0.04);
				\draw [fill=black] (0.6,-0.4) ellipse (0.04 and 0.04);
				\draw [fill=black] (1.4,0) ellipse (0.04 and 0.04);
			\end{tikzpicture}
		\end{center}
		where $\mathfrak{m}(u_i)=\mathfrak{m}(v_i)=1$ for $i\in\{1,2\}$ and $\chi'\subseteq \chi$ is defined such that $\mathcal{V}_{\chi'} = \mathcal{V}_\chi \setminus \{u_1, u_2, v_1, v_2\}$ and $\mathcal{P}_{\chi'} = \mathcal{P}_\chi \setminus \{x, y_1, y_2\}$. In addition, one of the following holds.
		\begin{enumerate}
			\item $X_i = \{x^{v_3}\}$, $X_{i+1} = \{y_1^{v_1}, y_2^{v_2}\}$, $X_{i+2} = \{x^{v_1}, x^{v_2}\}$ and $X_{i+3} = \{\sigma(x^{v_3})\}$.
			\item $X_{i} = \{x^{v_1}, x^{v_2}\}$ and $X_{i+1} = \{\sigma(x^{v_3})\}$.
		\end{enumerate}
	\end{lem}
	\begin{proof}
		We note that since $|x|=3$, it is not truncated. Thus, $X_i\in \hstepone \cup \hsteptwo$ and hence there are two cases to consider. For the first case, suppose $X_i \in \hstepone$ and write $X_i=\{g_i\}$. Then $|x \setminus X_i|=2$ and so $|X_{i+1}|=2$. So $X_{i+1}=\{g_{i+1}, g'_{i+1}\} \in\hsteptwo \cup \hstepthree$, as $(X_j)$ is non-terminating by assumption. This implies that $\val(\kappa(g_{i+1})),\val(\kappa(g_{i+1})) = 2$ and therefore that $\sigma(g_{i+1}), \sigma(g'_{i+1}) \in x$. Since $x$ is non-truncated, we must have $X_{i+1} \in \hstepthree$.
		
		Similar to the arguments used in the proof of the previous lemma, this necessarily implies that $x$ is not self-folded at the vertices $v_1=\kappa(g_{i+1})$ and $v_2=\kappa(g'_{i+1})$. In addition, there is only one polygon distinct from $x$ connected to $v_1$, which is a truncated edge $y_1=\pi(g_{i+1})$, and likewise, there is only one polygon distinct from $x$ connected to $v_2$, which is a truncated edge $y_2=\pi(g'_{i+1})$. We also have $\mathfrak{m}(v_1)=\mathfrak{m}(v_2)=1$. It follows that $x = \{g_i, \sigma(g_{i+1}), \sigma(g'_{i+1})\}$ and that the vertices $v_1$, $v_2$ and $v_3=\kappa(g_i)$ are pairwise distinct. Thus, $\chi$ must be of the form given in the lemma statement. By writing $g_i=x^{v_3}$, it is easy to verify that (1) holds from the lemma statement.
		
		For the final case, suppose instead that $X_i \in \hsteptwo$ and write $X_i=\{g_i, g'_i\}$. Since $|x \setminus X_i|=1$, it follows that $X_{i+1} \in \hstepone$. Write $X_{i+1}=\{g_{i+1}\}$. It is then easy to verify again that $\chi$ is the Brauer configuration given in the lemma statement, where $v_1=\kappa(g_i)$, $v_2 = \kappa(g'_i)$, $v_3=\kappa(g_{i+1})$, $x=\{g_i, g'_i, \sigma^{-1}(g_{i+1})\}$, $y_1 = \pi(\sigma(g_i))$ and $y_2 = \pi(\sigma(g'_i))$. By writing $g_i=x^{v_1}$ and $g'_i=x^{v_2}$, it is also easy to verify that (2) holds from the lemma statement.
	\end{proof}
	
	\begin{prop}
		Let $(X_j)_{j \in \mathbb{Z}_{\geq 0}}$ be a Green hyperwalk in $\chi$ that does not terminate. Then $(X_j)$ is periodic.
	\end{prop}
	\begin{proof}
		By Lemma~\ref{QuadserialWalk}, we know that $|\pi(g)| \leq 4$ for any $g \in X_i$ and any $i \geq 0$. Since $|\pi(g)| \neq 1$, there are three cases to consider.
		
		Case 1: Suppose that $|\pi (g)| = 2$ for any $g \in X_i$ and all $i \geq 0$. In this case, $X_i \in \hstepone$ for all $i \geq 0$. To see this, first note that $X_i=\{g, g'\} \not\in \hsteptwo$ for any $i \geq 0$, since then $X_i$ would otherwise not be a proper subset of a polygon. We then note that $X_i=\{g, g'\} \not\in \hstepthree$ for any $i \geq 0$, since otherwise $X_{i+1}=\{\sigma(g), \sigma(g')\} \in \hsteptwo$. But if $X_i \in \hstepone$ for all $i \geq 0$, then $(X_j)$ follows a Green walk around some subconfiguration of $\chi$ that is a Brauer graph. Since Green walks in Brauer graphs are periodic (Remark~\ref{GWPeriodic}), it follows that $(X_j)$ is periodic.
		
		Case 2: Now suppose there exists a germ $g \in X_i$ for some $i \geq 0$ such that $|\pi (g)| = 3$. By Lemma~\ref{QuadserialWalk}, this automatically implies that $|\pi (g)| \leq 3$ for any $g \in X_i$ and any $i\geq 0$. There are two subcases to consider. Namely, we have the subcase where there exists no $k\geq 0$ such that $X_k=\{g\}$ with $|\pi(g)|=2$, and the subcase where there exists some $k \geq 0$ such that $X_k=\{g\}$ with $|\pi(g)|=2$.
		
		Case 2a: Suppose that there exists no $k\geq 0$ such that $X_k=\{g\}$ with $|\pi(g)|=2$. If $X_0 \in \hstepone$, then $X_0=\{g_0\}$ with $|\pi(g_0)|=3$. By Lemma~\ref{TriserialWalk}, $X_3 = \{\sigma(g_0)\}$, which must be such that $|\pi(\sigma(g_0))|=3$ by our supposition. Proceeding inductively, one can see that $X_{3i} = \{\sigma^i(g_0)\}$ for all $i \geq 0$. But there exists some $k >0$ such that $\sigma^k(g_0)=g_0$. Thus, $(X_j)$ must be periodic. If we instead have $X_0 \in \hsteptwo$, then $X_1=\{g_1\}$ by Lemma~\ref{TriserialWalk}, which due to our supposition is such that $|\pi(g_1)|=3$. Using the previous argument, there exists a $k$ such that $X_k=\{\sigma^{-1}(g_1)\}$. One can then see from Lemma~\ref{TriserialWalk} that $X_{k+1}\in \hstepthree$ and $X_{k+2}=X_0$, so $(X_j)$ is periodic. Finally, if $X_0 \in \hstepthree$, then $(X_j)$ is again periodic by a similar argument.
		
		Case 2b: Suppose that there exists some $i\geq 0$ such that $X_i=\{g\}$ with $|\pi(g)|=2$. Define sets
		\begin{align*}
			\mathcal{Q}&=\{\pi(g) : g \in X_i \text{ for some } i \text{ such that } X_i \in \hstepone \text{ and } |\pi(g)|=2\} \\
			\mathcal{Q}'&=\{\pi(g) : g \in X_i \text{ for some } i \text{ such that } X_i \in \hsteptwo \cup \hstepthree\} \\
			\mathcal{U}&=\{\kappa(g) : g \in x \text{ for some } x \in \mathcal{Q}\} \\
			\mathcal{U}'&=\{\kappa(g) : g \in x \text{ for some } x \in \mathcal{Q}'\}
		\end{align*}
		Note that the sets $\mathcal{Q}'$ and $\mathcal{U}'$ capture precisely the vertices and polygons in all possible subconfigurations of the form $\chi \setminus \chi'$ of Lemma~\ref{TriserialWalk}. In the context of (1) in Lemma~\ref{TriserialWalk}, it is also easy to then see that the pair $(\mathcal{U}, \mathcal{Q})$ induces a connected subgraph of $\chi$ (that is, a subconfiguration consisting entirely of edges). Green walks around Brauer graphs are periodic, and so the subsequence of $(X_j)$ that skips steps $X_i \in \hsteptwo \cup \hstepthree$ must also be periodic. The periodicity of the whole of $(X_j)$ then follows from (1) and (2) of Lemma~\ref{TriserialWalk}.
	
		Case 3: If there exists a germ $g \in X_i$ for some $i \geq 0$ such that $|\pi (g)| = 4$, then $(X_j)$ is periodic by Lemma~\ref{QuadserialWalk}.
	\end{proof}
	
	\section{Periodic Projective Resolutions from Green Hyperwalks} \label{sec:ProjResolutions}
	In this section, we show that Green hyperwalks in the Brauer configuration are closely related to projective resolutions in the corresponding symmetric special multiserial algebra. For this, we need to define some distinguished sets of modules. For the purpose of avoiding some slight technical differences, we will assume throughout the rest of the paper that $\chi$ is not the exceptional Brauer configuration consisting precisely of one truncated edge.
	
	\begin{defn}
		Let $A=KQ/I$ be a symmetric special multiserial algebra corresponding to a Brauer configuration $\chi$. Define disjoint sets $\Mchi^0, \ldots, \Mchi^3$ of $A$-modules as follows.
		\begin{enumerate}[label=(M\arabic*)]
			\setcounter{enumi}{-1}
			\item $M \in \Mchi^0$ if and only if $M = S(x)$ for some truncated edge $x \in \mathcal{P}_\chi$.
			\item $M \in \Mchi^1$ if and only if $M$ is a maximal non-projective uniserial epimorphic image of $P(x)$ for some $x \in \mathcal{P}_\chi$.
			\item $M \in \Mchi^2$ if and only if $M$ is a string module $M(\alpha^{-1}\beta)$ such that $\alpha,\beta \in Q_1$ and $\{\widehat{t}(\alpha^{-1}),\widehat{s}(\beta)\} \in \hsteptwo$.
			\item $M \in \Mchi^3$ if and only if $M$ is a string module $M(\alpha\beta^{-1})$ such that $\alpha,\beta \in Q_1$ and $\{\widehat{s}(\alpha),\widehat{t}(\beta^{-1})\} \in \hstepthree$.
		\end{enumerate}
		Finally, we define a set $\Mchi = \Mchi^0 \cup \Mchi^1 \cup \Mchi^2\cup \Mchi^3$.
	\end{defn}
	
	Those familiar with Green walks in Brauer graphs and the representation theory of special biserial algebras may recognise definitions (M0) and (M1). Namely, the set $\Mchi^0 \cup \Mchi^1$ appears as the set $\mathcal{M}$ defined in \cite{DuffieldBGA} and \cite{Roggenkamp} and as the set $\overline{\mathcal{M}}$ defined in \cite{TrivialGentle}.
	
	\begin{rem} \label{M1Strings}
		If $M \in \Mchi^1$, then $M=M(w)$ for some direct string $w$. If $s(w)$ is a non-truncated polygon, then $w$ is maximal. If $s(w)$ is instead a truncated edge, then $w\alpha$ is a maximal direct string for some (unique) arrow $\alpha$. We also remark that in the truncated case, $M(w\alpha)=P(s(w))$.
	\end{rem}
	
	The modules in $\mathcal{M}_\chi$ and the Green hyperwalk steps in $\hstep$ are closely related. To show the relationship between the two, we will define a map $\omega\colon \mathcal{M}_\chi \rightarrow \hstep$ by
	\begin{equation*}
		\omega(M) =
		\begin{cases}
			\{g \in x : \kappa(g) \text{ truncated}\}, &	\text{if } M=S(x)\in\Mchi^0, \\
			\{\widehat{s}(w)\},		&	\text{if } M=M(w) \in \Mchi^1 \text{ with } w \text{ direct,} \\
			\{\widehat{t}(\alpha^{-1}), \widehat{s}(\beta) \},	&	\text{if } M=M(\alpha^{-1}\beta)\in\Mchi^2, \\
			\{\widehat{s}(\alpha), \widehat{t}(\beta^{-1}) \},	&	\text{if } M=M(\alpha\beta^{-1})\in\Mchi^3.
		\end{cases}
	\end{equation*}
	
	\begin{lem}
		The map $\omega$ is bijective. In particular, $\omega$ induces bijections
		\begin{align*}
			\omega_1 &\colon \Mchi^0 \cup \Mchi^1 \rightarrow \hstep^1 \\
			\omega_2 &\colon\Mchi^2 \rightarrow \hstep^2 \\
			\omega_3 &\colon\Mchi^3 \rightarrow \hstep^3.
		\end{align*}
	\end{lem}
	\begin{proof}
		Showing the map is surjective is straightforward. If $X = \{g\} \in \hstepone$, then either $\kappa(g)$ is non-truncated, in which case there exists a (unique) arrow $\alpha$ such that $\widehat{s}(\alpha)=g$, or $\kappa(g)$ truncated. In the non-truncated case, there exists a direct string $w$ in accordance with Remark~\ref{M1Strings} such that $M(w) \in \Mchi^1$ and $\omega(M(w))=X$. In the truncated case, we have $M=S(\pi(g)) \in \Mchi^0$ with $\omega(M)=X$. If $X = \{g, g'\} \in \hstep^2$, then let $\alpha$ and $\beta$ be the (unique) arrows in $Q_1$ such that $\widehat{s}(\alpha)=g$ and $\widehat{s}(\beta)=g'$. Then $M=M(\alpha^{-1}\beta) \in \Mchi^2$ and $\omega(M)=X$, as required. Finally, if $X = \{g, g'\} \in \hstep^3$, then again let $\alpha$ and $\beta$ be the (unique) arrows in $Q_1$ such that $\widehat{s}(\alpha)=g$ and $\widehat{s}(\beta)=g'$. It follows that $M=M(\alpha\beta^{-1}) \in \Mchi^3$ and $\omega(M)=X$, as required. All possible cases have been considered. Thus for any $X \in \hstep$, there exists a module $M \in \mathcal{M}_\chi$ such that $\omega(M)=X$. So $\omega$ is surjective. In particular, the restriction of $\omega$ to $\Mchi^0 \cup \Mchi^1$, $\Mchi^2$ and $\Mchi^3$ induce respective surjections $\omega_1$, $\omega_2$ and $\omega_3$.
		
		It remains to show that $\omega$ is injective. So suppose $\omega(M_1)=\omega(M_2)=X$ for some $M_1,M_2 \in \mathcal{M}_\chi$. If $X=\{g\} \in \hstep^1$,  then necessarily $M_1,M_2 \in \Mchi^0 \cup \Mchi^1$ . There are two subcases to consider. The first subcase is where $\kappa(g)$ is truncated. Since there exists no arrow $\alpha \in Q_1$ such that $\widehat{s}(\alpha)=g$, neither $M_1$ nor $M_2$ can be a string module of the form $M(w)$ for some direct string $w$, and thus $M_1,M_2 \in \Mchi^0$. The definition of $\omega$ along with Shur's Lemma then implies $M_1=M_2=S(\pi(g))$. The second subcase is where $\kappa(g)$ is non-truncated. This can only be true if both $M_1,M_2 \in \Mchi^1$. Moreover, there exists only one arrow $\alpha\in Q_1$ such that $\widehat{s}(\alpha)=g$. So there exists precisely one direct string $w$ such that $\omega(M(w))=X$. Hence, $M_1=M_2$ by Remark~\ref{M1Strings}. If we instead have $X=\{g, g'\} \in \hstep^2 \cup \hstep^3$, then the uniqueness of arrows $\alpha,\beta \in Q_1$ satisfying $\widehat{s}(\alpha)=g$ and $\widehat{s}(\beta)=g'$ shows that we necessarily have $M_1=M_2$. All possible cases have been considered. Thus, $\omega$ is injective. Since $\omega$ is also surjective, it is bijective (as are the restriction maps $\omega_1$, $\omega_2$ and $\omega_3$).
	\end{proof}
	
	\begin{thm} \label{ProjRes}
		Let $A=KQ/I$ be a symmetric special multiserial algebra corresponding to a Brauer configuration $\chi$. Consider the minimal projective resolution
		\begin{equation*}
			\mathbf{P}_\bullet\colon \cdots \rightarrow P_1 \rightarrow P_0 \rightarrow M \rightarrow 0
		\end{equation*}
		of a module $M \in \mathcal{M}_\chi$. Let $(X_i)_{i \in \mathbb{Z}_{\geq 0}}$ be the Green hyperwalk of $\chi$ from $\omega(M)$ and let $n>0$ be an integer such that $X_i \neq \emptyset$ for all $i<n$. Then
		\begin{enumerate}[label=(\alph*)]
			\item For any $i<n$, we have $\Omega^i(M)=\omega^{-1}(X_i) \in \mathcal{M}_\chi$.
			\item For any $i<n$,
			\begin{equation*}
				P_i=
				\begin{cases}
					P(x_i)								&	\text{ if } X_i \in \hstep^1 \cup \hstep^2 \text{ with } X_i \subset x_i \in \mathcal{P}_\chi, \\
					P(\pi(g_i)) \oplus P(\pi(g'_i))	&	\text{ if } X_i = \{g_i, g'_i\} \in \hstep^3.
				\end{cases}
			\end{equation*}
			\item If $(X_i)$ is periodic then so is $\mathbf{P}_\bullet$.
			\item If $(X_i)$ terminates and $n$ is maximal, then $\Omega^n(M) \not\in \Mchi$. Moreover, the module $\soc \Omega^n(M)$ is simple and we have
			\begin{equation*}
				\Omega^n(M) = \sum_{j=1}^m U_j
			\end{equation*}
			for some maximal uniserial submodules $U_1,\ldots,U_m \subset P_{n-1}$.
		\end{enumerate}
	\end{thm}
	\begin{proof}
		(a) Let $i \geq 0$ and suppose $X_{i+1} \neq \emptyset$. As a remark, this implies that $|\pi(g)|\leq 4$ for all $g \in X_i$. We will show that $\Omega(\omega^{-1}(X_i))=\omega^{-1}(X_{i+1})$ in all possible cases.
		
		Case 1i: Suppose $X_i=\{g_i\}$ and $|\pi(g_i)|=2$. That $\Omega(\omega^{-1}(X_i))=\omega^{-1}(X_{i+1})$ follows from \cite[Proposition 2.4, Lemma 2.5]{Roggenkamp}.
		
		Case 1ii: Suppose $X_i=\{g_i\}$ and $|\pi(g_i)|=3$. In this case, we note that $\omega^{-1}(X_i) = M(w)$, where $w$ is the maximal direct string such that $\widehat{s}(w)=g_i$. It follows from Lemma~\ref{TriserialWalk} that $X_{i+1}=\{g_{i+1}, g'_{i+1}\} \in \hstep^3$. In accordance with the notation of Lemma~\ref{TriserialWalk}, let $x=\pi(g_i)=\{g_i, \sigma(g_{i+1}), \sigma(g'_{i+1})\}$, let $y=\pi(g_{i+1})$ and let $y'=\pi(g'_{i+1})$. Note that $y$ and $y'$ are distinct truncated edges with no common vertex. Moreover,
		\begin{equation*}
			\rad P(x) / \soc P(x) = S(y) \oplus S(y') \oplus \rad M(w).
		\end{equation*}
		Thus, there exists an exact sequence 
		\begin{equation*}
			0 \rightarrow \Omega(M(w)) \rightarrow P(x) \rightarrow M(w) \rightarrow 0.
		\end{equation*}
		It follows that $\Omega(M(w))$ is the submodule of $P(x)$ such that $\tp \Omega(M(w)) = S(y) \oplus S(y')$ and $\rad \Omega(M(w)) = S(x)$, which is precisely the module $M(\alpha\beta^{-1})$ with $\alpha,\beta \in Q_1$ such that $\widehat{s}(\alpha)=g_{i+1}$ and $\widehat{s}(\beta)=g'_{i+1}$. It is easy to verify that $\omega^{-1}(X_{i+1})=M(\alpha\beta^{-1})$, so $\Omega(\omega^{-1}(X_i))=\omega^{-1}(X_{i+1})$, as required.
		
		Case 1iii: Suppose $X_i=\{g_i\}$ and $|\pi(g_i)|>3$. This is not possible, since we then have $|\pi(g_i) \setminus X_i|>2$, which implies $X_{i+1}=\emptyset$.
		
		Case 2i: Suppose $X_i=\{g_i, g'_i\} \in \hstep^2$ and $|\pi(g_i)|=3$. By Lemma~\ref{TriserialWalk}, $X_{i+1}=\{g_{i+1}\}$ and $x=\pi(g_i)=\{g_i, g'_i, \sigma^{-1}(g_{i+1})\}$. In this case, $y=\pi(\sigma(g_i))$ and $y'=\pi(\sigma(g'_i))$ are distinct truncated edges with no common vertex, and $\omega^{-1}(X_{i+1})=M(w)$ for some direct string $w$ such that $\widehat{s}(w)=g_{i+1}$. In addition, $\omega^{-1}(X_i)=M(\alpha^{-1}\beta)$ with $\alpha,\beta \in Q_1$ such that $\widehat{s}(\alpha)= g_i$ and $\widehat{s}(\beta)= g'_i$. Since $\tp M(\alpha^{-1}\beta)=S(x)$, there exists an exact sequence
		\begin{equation*}
			0 \rightarrow \Omega(M(\alpha^{-1}\beta)) \rightarrow P(x) \rightarrow M(\alpha^{-1}\beta) \rightarrow 0.
		\end{equation*}
		But
		\begin{equation*}
			\rad P(x) / \soc P(x) = S(y) \oplus S(y') \oplus M(w) / \soc M(w).
		\end{equation*}
		So $\Omega(\omega^{-1}(X_i))=\Omega(M(\alpha^{-1}\beta)) = M(w)=\omega^{-1}(X_{i+1})$, as required.
		
		Case 2ii: Suppose $X_i=\{g_i, g'_i\} \in \hstep^2$ and $|\pi(g_i)|=4$. By definition, $X_{i+1}=\{g_{i+1}, g'_{i+1}\} \in \hstep^3$. In particular, $x=\pi(g_i)=\{g_i, g'_i, \sigma(g_{i+1}), \sigma(g'_{i+1})\}$, and the edges $y=\pi(\sigma(g_i))$, $y'=\pi(\sigma(g'_i))$, $z=\pi(g_{i+1})$ and $z'=\pi(g'_{i+1})$ are truncated, pairwise distinct, and have no common vertices. Thus, $\chi$ is precisely the Brauer configuration in Lemma~\ref{QuadserialWalk}.  Furthermore, we have $\omega^{-1}(X_i)=M(\alpha^{-1}\beta)$, where $\alpha,\beta \in Q_1$ are such that $\widehat{s}(\alpha)=g_i$ and $\widehat{s}(\beta)=g'_i$, and we have $\omega^{-1}(X_{i+1})=M(\gamma\delta^{-1})$, where $\gamma,\delta \in Q_1$ are such that $\widehat{s}(\gamma)=g_{i+1}$ and $\widehat{s}(\delta)=g'_{i+1}$. Since 
		\begin{equation*}
			\rad P(x) / \soc P(x) = S(y) \oplus S(y') \oplus S(z) \oplus S(z')
		\end{equation*}
		and $\soc M(\alpha^{-1}\beta) =  S(y) \oplus S(y')$ and $\tp M(\gamma\delta^{-1}) = S(z) \oplus S(z')$, it is easy to see that $\Omega(M(\alpha^{-1}\beta))=M(\gamma\delta^{-1})$. Thus, $\Omega(\omega^{-1}(X_i))=\omega^{-1}(X_{i+1})$, as required.
		
		Case 3: $X_i=\{g_i, g'_i\} \in \hstep^3$ . In this case, $X_{i+1}=\{\sigma(g_i), \sigma(g'_i)\} \in \hstep^2$. Let $y= \pi(g_i)$, $y'=\pi(g'_i)$ and $x=\pi(\sigma(g_i))=\pi(\sigma(g'_i))$. Note that $\omega^{-1}(X_i)=M(\alpha\beta^{-1})$, where $\alpha,\beta \in Q_1$ are such that $\widehat{s}(\alpha)=g_i$ and $\widehat{s}(\beta)=g'_i$. Further note that $\omega^{-1}(X_{i+1})=M(\gamma^{-1}\delta)$, where $\gamma,\delta \in Q_1$ are such that $\widehat{s}(\gamma)=\sigma(g_i)$ and $\widehat{s}(\delta)=\sigma(g'_i)$. We thus have an exact sequence
		\begin{equation*}
			0 \rightarrow \Omega(M(\alpha\beta^{-1})) \rightarrow P(y) \oplus P(y') \rightarrow M(\alpha\beta^{-1}) \rightarrow 0,
		\end{equation*}
		and we have
		\begin{equation*}
			\rad P(y) / \soc P(y)=\rad P(y') / \soc P(y') = S(x).
		\end{equation*}
		It then follows that 
		\begin{equation*}
			\Omega(\omega^{-1}(X_i))=\Omega(M(\alpha\beta^{-1}))=M(\gamma^{-1}\delta)=\omega^{-1}(X_{i+1}),
		\end{equation*}
		as required.
		
		All possible cases have been considered. We therefore conclude that for any $i \geq 0$ such that $X_{i+1}\neq \emptyset$, we have $\Omega(\omega^{-1}(X_i))=\omega^{-1}(X_{i+1})$. Thus for any $i<n$, it follows that $\Omega^i(M)=\omega^{-1}(X_i)$, as required.
		
		(b) The statement follows from the exact sequences given in the proof of (a).
		
		(c) This is a trivial consequence of (a) and (b).
		
		(d) By the maximality of $n$, we have $X_n = \emptyset$. This has implications for the projective module $P_{n-1}$. Firstly, $P_{n-1}$ must be indecomposable. To see this, we note that $X_{n-1} \not\in\hstep^3$ by Remark~\ref{GWG3}, which implies that $X_{n-1} \in \hstep^1 \cup \hstep^2$, and hence by (b), that $P_{n-1} = P(x_{n-1})$ for some $x_{n-1} \in \Pchi$. Secondly, we can deduce that $P(x_{n-1})$ is not uniserial, since this would otherwise imply that $x_{n-1}$ is a truncated edge, which necessarily implies that $X_n \in \hstep^1$, contradicting the assumption that $n$ is maximal.
		
		The above reasoning also has implications for the module $\Omega^{n-1}(M)$. Since $X_{n-1} \in \hstep^1 \cup \hstep^2$ and $P_{n-1}$ is uniserial, it follows from (a) that $\Omega^{n-1}(M) \in \Mchi^1 \cup \Mchi^2$. Thus, $\Omega^{n-1}(M)$ is a maximal uniserial (if in $\Mchi^1$) or biserial (if in $\Mchi^2$) epimorphic image of $P(x_{n-1})$. The latter of these two cases follows from (a), as $\Omega^{n-1}(M) \in \Mchi^2$ implies that $X_{n-1} \in \hstep^2$, which implies that $x_{n-1}$ is connected to two truncated edges via distinct vertices of valency 2, and thus $\soc \Omega^{n-1}(M)$ is isomorphic to two simple direct summands of $\rad P(x_{n-1}) / \soc P(x_{n-1})$.
		
		We can now use this information to understand the structure of $\Omega^n(M)$ via the exact sequence
		\begin{equation*}
			0 \rightarrow \Omega^n(M) \rightarrow P(x_{n-1}) \rightarrow \Omega^{n-1}(M) \rightarrow 0.
		\end{equation*}
		Firstly, we note that since $A$ is multiserial, we have $\rad P(x_{n-1}) = \sum_{j=1}^r U_j$ for some uniserial modules such that $U_j \cap U_k$ is simple or zero for each $j \neq k$. In fact since $A$ is symmetric special multiserial, we have $U_j \cap U_k \cong S(x_{n-1})$ for any $j \neq k$. Now $\Omega^n(M)$ is a submodule of $P(x_{n-1})$, so $\soc \Omega^n(M)\cong S(x_{n-1})$ and we must have
		\begin{equation*}
			\Omega^n(M) = \sum_{j=1}^m U'_j,
		\end{equation*}
		where $m \leq r$ and each $U'_j \subseteq U_{k_j}$ for some $1 \leq k_j \leq r$. But $\Omega^{n-1}(M)$ is a maximal uniserial or biserial epimorphic image of $P(x_{n-1})$. So the above exact sequence implies that each $U'_j$ must be a maximal uniserial submodule of $P(x_{n-1})$. That is, each $U'_j = U_{k_j}$ and $m<r$. Thus, $\Omega^n(M)$ has the structure described in the Theorem statement.
		
		To show that $\Omega^n(M) \not\in \Mchi$, we note that $\Omega^n(M)$ cannot be uniserial (that is, $m \neq 1$), since $P(x_{n-1})$ is then either biserial (if $X_{n-1} \in \hstep^1$) or is the triserial projective associated to the polygon $x$ in Lemma~\ref{TriserialWalk} (if $X_{n-1} \in \hstep^2$). In both of these cases, we would then have $X_n \neq \emptyset$, which contradicts our assumption that $n$ is maximal. Similarly, $\Omega^n(M) \not\in \Mchi^3$ since this would imply that $P(x_{n-1})$ is either the triserial projective in Lemma~\ref{TriserialWalk} (if $X_{n-1} \in \hstep^1$) or the quadserial projective in Lemma~\ref{QuadserialWalk} (if $X_{n-1} \in \hstep^2$), and both of these cases would imply that $X_n \neq \emptyset$. Trivially, $\Omega^n(M) \not\in \Mchi^2$ since $\Omega^n(M)$ would then not be a submodule of $P(x_{n-1})$. So $\Omega^n(M) \not\in \Mchi$, as required.
	\end{proof}
	
	\begin{exam}
		Let $\chi$ be the Brauer configuration from Example~\ref{ex:GWs} and let $A$ be the associated symmetric special multiserial algebra. Assuming $\mult(v_2)=1$ (which is not strictly necessary, but convenient for the sake of being explicit), the periodic Green hyperwalk starting from $X_0 = \{y_4^{v_2}\}$ corresponds to the following (periodic) minimal projective resolution.
		\begin{multline*}
			\cdots \rightarrow P(x) \rightarrow P(y_2) \oplus P(y_3) \rightarrow P(x) \rightarrow P(y_6) \rightarrow P(y_5) \rightarrow P(y_4) \rightarrow
			\\ \rightarrow P(x) \rightarrow P(y_2) \oplus P(y_3) \rightarrow P(x) \rightarrow P(y_6) \rightarrow P(y_5) \rightarrow P(y_4) \rightarrow
			\begin{smallmatrix}
				y_4 \\ y_1 \\ y_6 \\ x
			\end{smallmatrix}
			\rightarrow 0
		\end{multline*}
		The terminating Green hyperwalk given by $X_0=\{y_1^{v_2}\}$ corresponds to the following (non-periodic) minimal projective resolution.
		\begin{equation*}
			\cdots \rightarrow P(z) \rightarrow P(y_6) \rightarrow P(y_1) \rightarrow P(y_1) \rightarrow 
			\begin{smallmatrix}
				y_1 \\ y_6 \\ x \\ y_4
			\end{smallmatrix}
			\rightarrow 0
		\end{equation*}
		The projective in the next term of the above projective resolution is $P(z) \oplus P(z) \oplus P(z)$. It is thus easy to verify that beyond the term at which the Green hyperwalk terminates, the number of indecomposable direct summands in each term of the projective resolution grows increasingly large. This observation is true for many terminating Green hyperwalks in other examples of symmetric special multiserial algebras.
	\end{exam}
	
	\section{Tubes Arising from Green Hyperwalks} \label{sec:HyperwalkTubes}
	In \cite{DuffieldBGA}, we showed that the double-stepped Green walks in a Brauer graph correspond to the tubes in the stable Auslander-Reiten of the associated symmetric special biserial algebra. We will prove a similar result for Green hyperwalks in a Brauer configuration. Namely, that there is a distinct tube in a symmetric special multiserial algebra for each double-stepped Green hyperwalk in the corresponding Brauer configuration. For this, we first need the following.
	
	\begin{lem} \label{lem:EpiImage}
		Let $A$ be a self-injective special multiserial $K$-algebra and suppose $P$ is an $n$-serial indecomposable projective-injective $A$-module with $n>1$.
		\begin{enumerate}[label=(\alph*)]
			\item Let $M_1, \ldots, M_m$ be a collection of pairwise non-isomorphic maximal uniserial epimorphic images of $P$ with $m<n$. Let $M$ be the $A$-module that is the (multiserial) epimorphic image of $P$ given by
			\begin{equation*}
				M= \sum_{i=1}^m M_i.
			\end{equation*}
			Then the Auslander-Reiten sequence ending in $M$ has an indecomposable middle term, and the middle term is not projective-injective.
			\item Let $M_1, \ldots, M_m$ be a collection of pairwise non-isomorphic maximal uniserial submodules of $P$ with $m<n$. Let $M$ be the $A$-module that is the (multiserial) submodule of $P$ given by
			\begin{equation*}
				M= \sum_{i=1}^m M_i.
			\end{equation*}
			Then the Auslander-Reiten sequence starting in $M$ has an indecomposable middle term, and the middle term is not projective-injective.
		\end{enumerate}
	\end{lem}
	\begin{proof}
		Write
		\begin{equation*}
			\rad P / \soc P = \bigoplus_{i=1}^n U_i,
		\end{equation*}
		where each $U_i$ is uniserial. Since $P$ is non-simple projective-injective, there exists by \cite[IV.3.11]{BlueBookI} an Aulsander-Reiten sequence
		\begin{equation*}
			\xymatrix@C=1.5em{0 \ar[r] & \rad P \ar[rr]^-{\left(\begin{smallmatrix} g_0 \\ \vdots \\ g_n \end{smallmatrix}\right)} && P \oplus \bigoplus_{i=1}^n U_i \ar[rr]^-{\left(\begin{smallmatrix} f_0 & \cdots & f_n \end{smallmatrix}\right)} && P / \soc P \ar[r] & 0}.
		\end{equation*}
		
		(a) Define an index set
		\begin{equation*}
			J=\{j \in \{1,\ldots, n\}: U_j \not\cong \rad M_i \text{ for any } 1 \leq i \leq m\}.
		\end{equation*}
		Clearly $|J|=n-m >0$ and
		\begin{equation*}
			\bigoplus_{k \not\in J} U_k \cong \bigoplus_{i=1}^m \rad M_i =\rad M.
		\end{equation*}
		It therefore follows that there exists an exact sequence
		\begin{equation*}
			\xymatrix{0 \ar[r] & \bigoplus_{j \in J} U_j \ar[r]^-{(f_j)_{j \in J}} & P / \soc P \ar[r] & M \ar[r] & 0}.
		\end{equation*}
		Suppose that the above exact sequence is an Auslander-Reiten sequence (which is only possible if $|J|=1$). Note that $P / \soc P$ is indecomposable and non-projective-injective. So the Auslander-Reiten sequence ending in $M$ then has precisely one indecomposable middle term that is not projective-injective, as required.
		
		On the other hand, suppose that the above exact sequence is not an Auslander-Reiten sequence. The map $(f_j)_{j \in J}$ is irreducible and $M$ is non-simple (since each $M_i$ is non-simple, which implies that $\rad M \neq 0$). Thus, $M$ is a non-simple cokernel of an irreducible monomorphism that is not left minimal almost split. It then follows from the dual of the results in \cite[Theorem 1, Remark 2]{Brenner} and \cite[p. 57, Theorem]{KernelMap} that the Auslander-Reiten sequence ending in $M$ has precisely one indecomposable middle term. By the dual result of \cite[Proposition 3.7]{KernelMap} and the fact that $A$ is self-injective, the middle term of the Auslander-Reiten sequence ending in $M$ is non-projective-injective. Thus, the Auslander-Reiten sequence ending in $M$ has at most one non-projective-injective indecomposable middle term, as required.
		
		(b) The proof is dual to (a). Define
		\begin{equation*}
			J=\{j \in \{1,\ldots, n\}: U_j \not\cong M_i / \soc M_i \text{ for any } 1 \leq i \leq m\}.
		\end{equation*}
		Then
		\begin{equation*}
			\bigoplus_{k \not\in J} U_k \cong \bigoplus_{i=1}^m M_i / \soc M_i = M / \soc M
		\end{equation*}
		and we have an exact sequence 
		\begin{equation*}
			\xymatrix{0 \ar[r] & M \ar[r] & \rad P \ar[r]^-{(g_j)_{j \in J}} & \bigoplus_{j \in J} U_j \ar[r] \ar[r] & 0},
		\end{equation*}
		which is either an Auslander-Reiten sequence with an indecomposable non-projective-injective middle term (if $|J|=1$, and so we are done) or is not an Auslander-Reiten sequence. The result for the latter case follows by similar arguments to (a). It follows that $M$ is the non-simple kernel of irreducible epimorphism that is not right minimal almost split. By \cite[Theorem 1]{Brenner} and \cite[p. 57, Theorem]{KernelMap}, this implies that the Auslander-Reiten sequence starting in $M$ has precisely one indecomposable middle term, which is non-projective-injective by \cite[Proposition 3.7]{KernelMap} and the fact that $A$ is self-injective.
	\end{proof}
	
	Recall that $\tau = \Omega^2$ for any (weakly) symmetric algebra. This fact along with the above lemma can be used to prove the following.
	
	\begin{prop} \label{prop:MchiMiddleTerm}
		Let $A=KQ/I$ be a symmetric special multiserial algebra corresponding to a Brauer configuration $\chi$ and let $M \in \mathcal{M}_\chi$. Then the Auslander-Reiten sequence
		\begin{equation*}
			0 \rightarrow \tau M \rightarrow E \rightarrow M \rightarrow 0
		\end{equation*}
		has either
		\begin{enumerate}[label=(\roman*)]
			\item an indecomposable middle term $E$ that is not projective-injective; or
			\item a middle term $E \cong P \oplus U$ for some indecomposable uniserial projective-injective module $P$ and some indecomposable uniserial non-projective-injective module $U$.
		\end{enumerate}
	\end{prop}
	
	\begin{proof}
		Case 0: Suppose $M=S(x) \in \Mchi^0$. Note that we have exact sequences
		\begin{equation*}
			0 \rightarrow \Omega(S(x)) \rightarrow P(x) \rightarrow S(x) \rightarrow 0,
		\end{equation*}
		\begin{equation*}
			0 \rightarrow \tau S(x) \rightarrow P' \rightarrow \Omega(S(x)) \rightarrow 0.
		\end{equation*}
		Since $P(x)$ is uniserial, we have $\Omega(S(x)) = \rad P(x)$, which is the unique maximal non-projective uniserial submodule of $P(x)$. In addition, we have $P'=P(\tp (\rad P(x)))$. There are now two subcases to consider: either $P'$ is uniserial or it is not uniserial.
		
		Case 0a: If $P'$ is uniserial then $\tau S(x) \cong \soc P'$ and is thus a simple module $S(y)$ for some truncated edge $y \in \Pchi$. In particular, by considering the Green hyperwalk from the germ $g$ associated to the truncated vertex of $x=\{g,g'\}$ and by using Theorem~\ref{ProjRes}(a), we can conclude that $\sigma(g') \in y$ and hence, there exists precisely one arrow $\alpha\colon x \rightarrow y$ in $Q$. This implies that we have $\dim_K \Ext_A^1(S(x),S(y)) = 1$ (c.f. \cite[III.2.12(b)]{BlueBookI}). But since $\tau S(x) \cong S(y)$, this means that the exact sequence
		\begin{equation*}
			0 \rightarrow S(y) \rightarrow M(\alpha) \rightarrow S(x) \rightarrow 0
		\end{equation*}
		is an Auslander-Reiten sequence. The middle term is indecomposable, and thus the Auslander-Reiten sequence ending in $M$ has the form given by (i) of the proposition statement.
		
		Case 0b: If $P'$ is not uniserial, then we note that since $\Omega(S(x))$ is a maximal uniserial epimorphic image of $P'$, it follows that $\tau S(x)$ is a submodule of $P'$ such that
		\begin{equation*}
			\tau S(x)=\sum_{i=1}^{m} M_i,
		\end{equation*}
		where each $M_i$ is a maximal uniserial submodule of $P'$. Since $P'$ is not uniserial, the Auslander-Reiten sequence starting in $\tau M$ has an indecomposable middle-term by Lemma~\ref{lem:EpiImage}(b). Thus, we have the Auslander-Reiten sequence given by (i) in the proposition statement.
	
		Case 1: Suppose we instead have $M \in \Mchi^1$, then we have two subcases to consider. The first case is where the projective cover $P \rightarrow M$ is such that $P$ is not uniserial, in which case the the result follows directly from Lemma~\ref{lem:EpiImage}(a). The second subcase is where $P$ is uniserial, in which case $M \cong P / \soc P$ and thus by \cite[IV.3.11]{BlueBookI} we have an Auslander-Reiten sequence
		\begin{equation*}
			0 \rightarrow \rad P \rightarrow P \oplus U \rightarrow M \rightarrow 0,
		\end{equation*}
		where $U$ is the indecomposable uniserial module $\rad P / \soc P$. This is precisely the Auslander-Reiten sequence given by (ii) in the Proposition statement.
		
		 Case 2: If $M=M(\alpha\inv\beta) \in \Mchi^2$, then the result is a straightforward consequence of Lemma~\ref{lem:EpiImage}(a), since $M$ is an epimorphic image of the $n$-serial (with $n>2$) indecomposable projective $P(s(\beta))$ and
		\begin{equation*}
			M=M(\alpha)+M(\beta),
		\end{equation*}
		where $M(\alpha)$ and $M(\beta)$ are maximal uniserial epimorphic images of $P(s(\beta))$. So in this case, we have the Auslander-Reiten sequence given by (i) in the Proposition statement.
		
		Case 3: Finally, suppose that $M=M(\alpha\beta\inv) \in \Mchi^3$. Then $M = \omega\inv(X_0)$ for some Green hyperwalk step $X_0 \in \hstep^3$. Let $(X_i)_{i \in \mathbb{Z}_{\geq 0}}$ be the Green hyperwalk of $\chi$ from $X_0$. Necessarily, $X_1 \in \hstep^2$. There are then two possibilities for $X_2$: either the Green hyperwalk terminates precisely at $X_2$ or it does not.
		
		Case 3a: If the Green hyperwalk terminates at $X_2$, then by Theorem~\ref{ProjRes}(d), $\tau M = \Omega^2(M)=\omega\inv(X_2)$ is a sum of some maximal uniserial submodules of some indecomposable (non-uniserial) projective-injective $P$. The result then follows by Lemma~\ref{lem:EpiImage}(b), where we obtain the sequence given by (i) of the Proposition statement.
		
		Case 3b: If the Green hyperwalk does not terminate at $X_2$, then $X_1 \subset x \in \Pchi$ with $3 \leq |x| \leq 4$ by Lemma~\ref{QuadserialWalk} and the definition of $\hstep^2$. Thus by Lemma~\ref{QuadserialWalk} and Lemma~\ref{TriserialWalk}, we have $X_2 \in \hstep^1 \cup \hstep^3$.
		
		Case 3bi: If $X_2 \in \hstep^3$, then Theorem~\ref{ProjRes}(a) implies $\tau M = \Omega^2(M)=\omega\inv(X_2)=M(\gamma\delta\inv) \in \Mchi^3$ for some string $\gamma\delta\inv$. In particular $\tau M = M(\gamma) + M(\delta)$, and $M(\gamma)$ and $M(\delta)$ are maximal uniserial submodules of the projective-injective module $P(x)$. Thus, we obtain sequence (i) from the Proposition statement by applying Lemma~\ref{lem:EpiImage}(b).
		
		Case 3bii: If we instead have $X_2 = \{g\} \in \hstep^1$ then $x$ must be of the form given in Lemma~\ref{TriserialWalk} and thus $\kappa(g)$ is non-truncated. So $\tau M = \Omega^2(M)=\omega\inv(X_2)=M(w) \in \Mchi^1$ for some maximal direct string $w$. In particular, $\tau M$ is a maximal non-projective uniserial submodule of the projective-injective $P(t(w))$. In fact, $t(w)=x$, which is not uniserial. So we can use Lemma~\ref{lem:EpiImage}(b) to conclude that the Auslander-Reiten sequence ending in $M$ is of the form given by (i) of the Proposition statement. This is the final case to consider, and so we are done.
	\end{proof}
	
	In light of Theorem~\ref{ProjRes} and the fact that $\tau = \Omega^2$, we make the following definition.
	\begin{defn}
		A \emph{double-stepped Green hyperwalk} of $\chi$ is a subsequence $(X_{2i})_{i \in \mathbb{Z}_{\geq 0}}$ of a Green hyperwalk $(X_i)_{i \in \mathbb{Z}_{\geq 0}}$ of $\chi$.
	\end{defn}
	
	Clearly if a Green hyperwalk is periodic, any double-stepped subsequence is also periodic.
	
	\begin{defn}
		Let $(X_i)_{i \in \mathbb{Z}_{\geq 0}}$ and $(X'_i)_{i \in \mathbb{Z}_{\geq 0}}$ be periodic Green hyperwalks of $\chi$ (either both single-stepped or both double-stepped). We say $(X_i)$ and $(X'_i)$ are \emph{distinct} if there exists no integer $j>0$ such that $X_i = X'_{j+i}$ for all $i \geq 0$.
	\end{defn}
	
	\begin{thm}
		Let $A$ be a connected symmetric special multiserial algebra of infinite-representation type corresponding to a Brauer configuration $\chi$.
		\begin{enumerate}[label=(\alph*)]
			\item Let $X=(X_i)_{i \in \mathbb{Z}_{\geq 0}}$ be a periodic double-stepped Green hyperwalk of $\chi$ of period $n$. Then there exists a corresponding tube $\mathcal{T}_X$ of the form $\mathbb{Z}\mathbb{A}_\infty / \langle \tau^n \rangle$ in $_s\Gamma_A$.
			\item Given a tube $\mathcal{T}_X$ in $_s\Gamma_A$ corresponding to a periodic double-stepped Green hyperwalk $X=(X_i)_{i \in \mathbb{Z}_{\geq 0}}$, the modules at the mouth of $\mathcal{T}_X$ are precisely the modules $\omega^{-1}(X_i)$ for each $i$.
			\item Let $X=(X_i)_{i \in \mathbb{Z}_{\geq 0}}$ and $X'=(X'_i)_{i \in \mathbb{Z}_{\geq 0}}$ be distinct periodic double-stepped Green hyperwalks of $\chi$. Then $\mathcal{T}_X$ and $\mathcal{T}_{X'}$ are distinct tubes in $_s\Gamma_A$.
		\end{enumerate}
	\end{thm}
	\begin{proof}
		(a) Let $X=(X_i)_{i \in \mathbb{Z}_{\geq 0}}$ be a periodic double-stepped Green hyperwalk of $\chi$. By Theorem~\ref{ProjRes}(c), we know that for each $i$, the modules $\omega\inv(X_i) \in \Mchi$ have periodic projective resolutions. In particular, we know by Theorem~\ref{ProjRes}(a) that $\omega\inv(X_{i+1})=\tau\omega\inv(X_i)$ and hence the modules $\omega\inv(X_i) \in \Mchi$ are $\tau$-periodic. Thus, the modules $\omega\inv(X_i) \in \Mchi$ all belong to the same stable component $\mathcal{T}_X$ of $_s\Gamma_A$.
		
		By the Riedtmann structure theorem, $\mathcal{T}_X$ is a valued translation quiver of the form $\mathbb{Z}\Delta / G$ for some valued quiver $\Delta$ and group $G$ of automorphisms of $\mathbb{Z}\Delta$. Since $\mathcal{T}_X$ contains a $\tau$-periodic module, it follows from the main result in \cite[pp. 602-603, Theorem]{Todorov} that the underlying graph of $\Delta$ is either a (finite) Dynkin diagram or an infinite tree. But the algebra $A$ is connected and of infinite-representation type. So $\Delta$ cannot be finite Dynkin diagram, as $\mathcal{T}_X$ would otherwise be a disconnected component of $_s\Gamma_A$. It thus follows directly from \cite[p. 280, Theorem]{Vinberg} that $\Delta = \mathbb{A}_\infty$. Finally, it follows from Theorem~\ref{ProjRes}(a) that if the period of the double-stepped Green hyperwalk $X$ is $n$, then for any step $X_i$ of $X$, we have $\tau^n \omega\inv(X_i) \cong \omega\inv(X_i)$ and there exists no $m < n$ such that $\tau^m \omega\inv(X_i) \cong \omega\inv(X_i)$. Thus $G= \langle \tau^n \rangle$.
		
		In conclusion, we have shown that each step $X_i$ of $X$ corresponds to a module $\omega\inv(X_i)$, and every such step resides in the same component $\mathcal{T}_X$ of $_s\Gamma_A$, and $\mathcal{T}_X$ is a tube of the form $\mathbb{Z}\mathbb{A}_\infty / \langle \tau^n \rangle$, as required.
		
		(b) This is an immediate consequence of Theorem~\ref{ProjRes} and Proposition~\ref{prop:MchiMiddleTerm}.
		
		(c) This follows immediately from (b), since a tube in an algebra of infinite-representation type has only one mouth.
	\end{proof}
	
	\begin{exam}
		Let $\chi$ be the Brauer configuration from Example~\ref{ex:GWs} with the additional assumption (for the sole purpose of being explicit) that $\mult(v_2)=\mult(v_8)=1$ and $\mult(v_7)=3$. Let $A$ be the symmetric special multiserial algebra corresponding to $\chi$. Then $\chi$ has only one periodic Green hyperwalk of period 6, which is given by
		\begin{equation*}
			(\{y_4^{v_2}\}, \{y_5^{v_7}\}, \{y_6^{v_8}\}, \{x^{v_2}\}, \{y_2^{v_3},y_3^{v_5}\}, \{x^{v_3},x^{v_5}\}, \{y_4^{v_2}\},\{y_5^{v_7}\}, \{y_6^{v_8}\},\ldots).
		\end{equation*}
		However, this splits into two distinct double-stepped Green hyperwalks of period 3. Namely,
		\begin{equation*}
			(\{y_4^{v_2}\}, \{y_6^{v_8}\}, \{y_2^{v_3},y_3^{v_5}\},\{y_4^{v_2}\},\ldots) \qquad \text{and} \qquad (\{y_5^{v_7}\}, \{x^{v_2}\}, \{x^{v_3},x^{v_5}\},\{y_5^{v_7}\},\ldots).
		\end{equation*}
		Thus, there are two tubes of rank 3 in $_s\Gamma_A$ arising from Green hyperwalks in $\chi$. The modules at the mouth of the former tube are
		\begin{equation*}
			\begin{smallmatrix}
				y_4 \\ y_1 \\ y_6 \\ x
			\end{smallmatrix}, \quad
			\begin{smallmatrix}
				y_6 \\ z \\ y_5
			\end{smallmatrix} \quad
			\text{ and } \quad
			\begin{smallmatrix}
				y_2 &  & y_3 \\ & x &
			\end{smallmatrix},
		\end{equation*}
		whereas the modules at the mouth of the latter tube are
		\begin{equation*}
			\begin{smallmatrix}
				y_5 \\ y_4 \\ y_5 \\ y_4 \\ y_5 \\ y_4
			\end{smallmatrix}, \quad
			\begin{smallmatrix}
				x \\ y_4 \\ y_1 \\ y_6
			\end{smallmatrix} \quad
			\text{ and } \quad
			\begin{smallmatrix}
				& x & \\
				y_2 &  & y_3
			\end{smallmatrix}.
		\end{equation*}
	\end{exam}
	
	\section{Additional Rank 2 Tubes Containing String Modules} \label{sec:rank2}
	As we have seen previously, Brauer subconfiguations of the form given in Lemma~\ref{TriserialWalk} are useful in the context of Green hyperwalks (since they potentially permit periodic behaviour) and Auslander-Reiten theory (since they potentially allow for tubes). Informally, this is because the representation theory of the associated symmetric special multiserial algebra exhibits a similar local behaviour to a quiver of type $\mathbb{D}$ or $\widetilde{\mathbb{D}}$ around these polygons. We will formalise this notion and show how one obtains additional tubes of rank 2 that contain string modules from these Brauer subconfigurations. The constructions in this section are again of a combinatorial nature, and so we will provide numerous examples throughout.
	
	\begin{defn} \label{defn:Dtriple}
	Let $\chi$ be a Brauer configuration. An ordered triple $(x, y_1, y_2)$ of polygons of $\chi$ is called a \emph{$\mathbb{D}$-triple} if $x=\{g_0,g_1,g_2\}$, $y_1=\{\sigma(g_1),g'_1\}$ and $y_2=\{\sigma(g_2),g'_2\}$ such that $x$ is not self-folded, $y_1$ and $y_2$ are truncated, $\val(\kappa(g_1))=\val(\kappa(g_2))=2$, and $\mathfrak{m}(\kappa(g_1))=\mathfrak{m}(\kappa(g_2))=1$. We denote the set of all $\mathbb{D}$-triples of $\chi$ by $\Dchi$.
	\end{defn}
	Essentially, the polygons $x$, $y_1$ and $y_2$ in the definition above are the same as those in Lemma~\ref{TriserialWalk}.
	
	\begin{defn} \label{defn:Wchi}
		Let $\chi$ be a Brauer configuration associated to a Brauer configuration algebra $A=KQ/I$. Let $d=(x,y_1,y_2)$ and $d'=(x',y'_1,y'_2)$ be two (possibly equal) $\mathbb{D}$-triples in $\chi$. We define a set of strings $\mathcal{W}(d, d')$ by $w \in \mathcal{W}(d, d')$ if and only if $w=\alpha \beta_1 \ldots \beta_n \gamma$ satisfies all of the following properties.
		\begin{enumerate}[label=(W\arabic*)]
			\item $|w| >2$
			\item For any substring $\beta_i\beta_{i+1}$ of $w$ with $\beta_i \in Q_1$ and $\beta_{i+1} \in Q^{-1}_1$, we have $|t(\beta_i)|=|s(\beta_{i+1})|=2$.
			\item For any substring $\beta_i\beta_{i+1}$ of $w$ with $\beta_i \in Q^{-1}_1$ and $\beta_{i+1} \in Q_1$, we have $|t(\beta_i)|=|s(\beta_{i+1})|=2$.
			\item $s(\alpha)=y_1$ and $t(\gamma)=y'_1$.
		\end{enumerate}
		Finally, define
		\begin{equation*}
			\Wchi=\bigcup_{d,d' \in \Dchi} \mathcal{W}(d,d') .
		\end{equation*}
	\end{defn}
	
	Property (W1) of the above definition ensures that the substring $\beta_1\ldots\beta_n$ of $w$ is non-zero. Properties (W2) and (W3) ensure that the polygons at a deep or peak of $\beta_2\ldots\beta_{n-1}$ are edges in $\chi$. It is easy to see that the sets $\mathcal{W}(d, d')$ and $\mathcal{W}(d', d)$ are equivalent, in the sense that each string in the latter set is the inverse of some string in the former set.
	
	\subsection{Involutions on $\Wchi$} \label{sec:Involutions}
	There exist three involution operations $\mu_0$, $\mu_1$ and $\mu_2$ on the set $\Wchi$ that are useful in describing the projective resolutions of string modules of strings in $\Wchi$. We will define them below.
	
	Let $w=\alpha w' \gamma \in \Wchi$. We will define a corresponding string $\mu_0(w) \in \Wchi$ such that $w \neq \mu_0(w)$. First note that since $s(w)$ and $t(w)$ are truncated, there exist unique symbols $\overline{\alpha}, \overline{\gamma} \in Q_1 \cup Q_1^{-1}$ such that $\overline{\alpha} \neq \alpha$, $\overline{\gamma} \neq \gamma$, $s(\overline{\alpha})=s(\alpha)$ and $t(\overline{\gamma})=t(\gamma)$. Now note that $w'$ can be written as a concatenation of direct substrings $w^+_i$ and inverse strings $w^-_i$ such that $w' = w^{\pm}_1 w^{\mp}_2\ldots$. Recall that each vertex $v \in \Vchi$ induces a cycle of arrows $\cycle_v$ in $Q_1$ and a cycle of formal inverses $\cycle\inv_v$. Thus, for each $w^+_i$, there exists a string $\overline{w}^+_i$ such that $w^+_i \overline{w}^+_i$ is a rotation of $\cycle_{v}^{\mult(v)}$, where $v=\kappa(\widehat{s}(w^+_i))$. Similarly, for each $w^-_i$, there exists a string $\overline{w}^-_i$ such that $w^-_i\overline{w}^-_i$ is a rotation of $\cycle_{v}^{-\mult(v)}$, where $v=\kappa(\widehat{s}(w^-_i))$. We may then define
	\begin{equation*}
		\mu_0(w)=\overline{\alpha} \overline{w}' \overline{\gamma},
	\end{equation*}
	where
	\begin{equation*}
		\overline{w}'=(\overline{w}^\pm_1)\inv (\overline{w}^\mp_2)\inv\ldots.
	\end{equation*}
	It is easy to see that $\mu_0^2(w)=w$, and thus is an involution. Moreover, if $w \in \mathcal{W}(d,d')$ for some $d,d' \in \Dchi$, then $\mu_0(w) \in \mathcal{W}(d,d')$.
	
	We will now define strings $\mu_1(w)$ and $\mu_2(w)$. If $w=\alpha w' \gamma \in \mathcal{W}(d,d')$ for some $\mathbb{D}$-triples $d=(x,y_1,y_2)$ and $d'=(x',y'_1,y'_2)$, then there exist unique symbols $\alpha^\ast,\gamma^\ast \in Q_1 \cup Q_1\inv$ such that $s(\alpha^\ast)=y_2$ and $t(\gamma^\ast)=y'_2$ and such that $\alpha^\ast w'\gamma^\ast$ is a string. This is because $|x|=|x'|=3$. We thus define
	\begin{equation*}
		\mu_1(w)=\alpha^\ast w'\gamma \qquad \text{and} \qquad \mu_2(w)=\alpha w'\gamma^\ast.
	\end{equation*}
	In addition, there exist $\mathbb{D}$-triples $\widehat{d}=(x,y_2,y_1)$ and $\widehat{d}'=(x',y'_2,y'_1)$ such that $\mu_1(w) \in \mathcal{W}(\widehat{d},d')$ and $\mu_2(w) \in \mathcal{W}(d,\widehat{d}')$. The operations $\mu_1$ and $\mu_2$ are clearly involutions. It is also easy to verify that $\mu_0$, $\mu_1$ and $\mu_2$ are pairwise commutative. 
	
	\begin{exam} \label{ex:WchiFinite}
		Consider the symmetric special multiserial algebra $A=KQ/I$ corresponding to the following Brauer configuration $\chi$ with $\mathfrak{m} \equiv 1$ (and with the arrows of $Q$ superimposed).
		
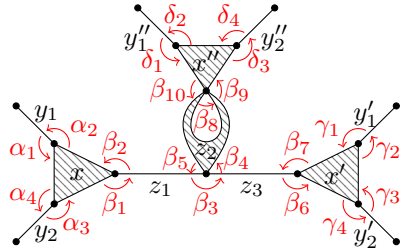
\begin{center}
			\begin{tikzpicture}
				\draw[pattern = north west lines, pattern color=gray] (-2,-0.4) -- (-2,0.4) -- (-1.2,0) -- (-1.2,0) -- (-2,-0.4);
				\draw (-2.5,0.9) -- (-2,0.4);
				\draw (-2.5,-0.9) -- (-2,-0.4);
				\draw[pattern = north west lines, pattern color=gray] (2,-0.4) -- (2,0.4) -- (1.2,0) -- (1.2,0) -- (2,-0.4);
				\draw (2.5,0.9) -- (2,0.4);
				\draw (2.5,-0.9) -- (2,-0.4);
				\draw (-1.2,0) -- (1.2,0);
				\draw[pattern = north west lines, pattern color=gray] (0,1.1) -- (0.4,1.7) -- (-0.4,1.7) -- (-0.4,1.7) -- (0,1.1);
				\draw [pattern = north west lines, pattern color=gray](0,0) .. controls (-0.3,0.4) and (-0.5,0.8) .. (0,1.1) .. controls (-0.7,0.2) and (0.7,0.2) .. (0,1.1) .. controls (0.5,0.8) and (0.3,0.4) .. (0,0);
				\draw (-0.9,2.2) -- (-0.4,1.7);
				\draw (0.9,2.2) -- (0.4,1.7);
				
				\draw (-1.7,0) node {\footnotesize$x$};
				\draw (-2.1,0.8) node {\footnotesize$y_1$};
				\draw (-2.1,-0.8) node {\footnotesize$y_2$};
				\draw (1.7,0) node {\footnotesize$x'$};
				\draw (2.1,0.8) node {\footnotesize$y'_1$};
				\draw (2.1,-0.8) node {\footnotesize$y'_2$};
				\draw (0,1.5) node {\footnotesize$x''$};
				\draw (-0.9,1.8) node {\footnotesize$y''_1$};
				\draw (0.9,1.8) node {\footnotesize$y''_2$};
				\draw (-0.6,-0.2) node {\footnotesize$z_1$};
				\draw (0,0.3) node {\footnotesize$z_2$};
				\draw (0.6,-0.2) node {\footnotesize$z_3$};
				\draw[red] (-2.4,0.3) node {\footnotesize$\alpha_1$};
				\draw[red] (-1.6,0.6) node {\footnotesize$\alpha_2$};
				\draw[red] (-1.7,-0.7) node {\footnotesize$\alpha_3$};
				\draw[red] (-2.4,-0.3) node {\footnotesize$\alpha_4$};
				\draw[red] (-1.2,-0.4) node {\footnotesize$\beta_1$};
				\draw[red] (-1.2,0.4) node {\footnotesize$\beta_2$};
				\draw[red] (0,-0.4) node {\footnotesize$\beta_3$};
				\draw[red] (0.4,0.2) node {\footnotesize$\beta_4$};
				\draw[red] (-0.4,0.2) node {\footnotesize$\beta_5$};
				\draw[red] (1.2,-0.4) node {\footnotesize$\beta_6$};
				\draw[red] (1.2,0.4) node {\footnotesize$\beta_7$};
				\draw[red] (0,0.7) node {\footnotesize$\beta_8$};
				\draw[red] (0.4,1.1) node {\footnotesize$\beta_9$};
				\draw[red] (-0.5,1.1) node {\footnotesize$\beta_{10}$};
				\draw[red] (1.6,0.6) node {\footnotesize$\gamma_1$};
				\draw[red] (2.4,0.3) node {\footnotesize$\gamma_2$};
				\draw[red] (2.4,-0.3) node {\footnotesize$\gamma_3$};
				\draw[red] (1.7,-0.7) node {\footnotesize$\gamma_4$};
				\draw[red] (-0.7,1.5) node {\footnotesize$\delta_1$};
				\draw[red] (-0.4,2.1) node {\footnotesize$\delta_2$};
				\draw[red] (0.7,1.5) node {\footnotesize$\delta_3$};
				\draw[red] (0.3,2.1) node {\footnotesize$\delta_4$};
				\draw [fill=black] (-2.5,0.9) ellipse (0.04 and 0.04);
				\draw [fill=black] (-2,0.4) ellipse (0.04 and 0.04);
				\draw [fill=black] (-2.5,-0.9) ellipse (0.04 and 0.04);
				\draw [fill=black] (-2,-0.4) ellipse (0.04 and 0.04);
				\draw [fill=black] (-1.2,0) ellipse (0.04 and 0.04);
				\draw [fill=black] (2.5,0.9) ellipse (0.04 and 0.04);
				\draw [fill=black] (2,0.4) ellipse (0.04 and 0.04);
				\draw [fill=black] (2.5,-0.9) ellipse (0.04 and 0.04);
				\draw [fill=black] (2,-0.4) ellipse (0.04 and 0.04);
				\draw [fill=black] (1.2,0) ellipse (0.04 and 0.04);
				\draw [fill=black] (0,0) ellipse (0.04 and 0.04);
				\draw [fill=black] (0,1.1) ellipse (0.04 and 0.04);
				\draw [fill=black] (-0.4,1.7) ellipse (0.04 and 0.04);
				\draw [fill=black] (-0.9,2.2) ellipse (0.04 and 0.04);
				\draw [fill=black] (0.4,1.7) ellipse (0.04 and 0.04);
				\draw [fill=black] (0.9,2.2) ellipse (0.04 and 0.04);
			
				\draw [->,red](-2.1732,0.5) arc (149.9993:255:0.2);
				\draw [->,red](-1.8068,0.3482) arc (-15.0089:120:0.2);
				\draw [->,red](-2.1,-0.5732) arc (-120.0007:15:0.2);
				\draw [->,red](-2.0518,-0.2068) arc (105.0089:210:0.2);
				\draw [->,red](-1.3414,-0.1414) arc (-135:-15:0.2);
				\draw [->,red](-1.0068,0.0518) arc (15.0089:135:0.2);
				\draw [->,red](1.0068,-0.0518) arc (-164.9911:-45:0.2);
				\draw [->,red](1.3414,0.1414) arc (45:165:0.2);
				\draw [->,red](1.8068,-0.3482) arc (164.9911:300:0.2);
				\draw [->,red](2.1732,-0.5) arc (-30.0007:75:0.2);
				\draw [->,red](2.1,0.5732) arc (59.9993:195:0.2);
				\draw [->,red](2.0518,0.2068) arc (-74.9911:30:0.2);
				\draw [->,red](-0.1932,-0.0518) arc (-164.9911:-15:0.2);
				\draw [->,red](0.2,0) arc (0:45:0.2);
				\draw [->,red](-0.1414,0.1414) arc (135:180:0.2);
				\draw [->,red](-0.1414,1.2414) arc (135:210:0.2);
				\draw [->,red](-0.1,0.9268) arc (-120.0007:-60:0.2);
				\draw [->,red](0.1732,1) arc (-30.0007:45:0.2);
			
				\draw [->,red](0.3482,1.5068) arc (-105.0089:30:0.2);
				\draw [->,red](0.5,1.8732) arc (59.9993:165:0.2);
			\draw [->,red](-0.5732,1.8) arc (149.9993:285:0.2);
			\draw [->,red](-0.2,1.7) arc (0:120:0.2);
			\end{tikzpicture}
		\end{center}
		There are precisely 6 distinct $\mathbb{D}$-triples in $\chi$. Namely, we have
		\begin{align*}
			d_1 &= (x,y_1,y_2), & d'_1 &= (x',y'_1,y'_2), & d''_1 &= (x'',y''_1,y''_2)\\
			d_2 &= (x,y_2,y_1), &	d'_2 &= (x',y'_2,y'_1), & d''_2 &= (x'',y''_2,y''_1).
		\end{align*}
		In addition, we have
		\begin{align*}
			\mathcal{W}(d_1,d'_1) &= \{w=\alpha_1\beta_2^{-1}\beta_3\beta_7^{-1}\gamma_2, \mu_0(w)=\alpha_2^{-1}\beta_1\beta_5^{-1}\beta_4^{-1}\beta_6\gamma_1^{-1}\} \\
			\mathcal{W}(d_1,d'_2) &= \{\mu_2(w)=\alpha_1\beta_2^{-1}\beta_3\beta_7^{-1}\gamma_4, \mu_0\mu_2(w)=\alpha_2^{-1}\beta_1\beta_5^{-1}\beta_4^{-1}\beta_6\gamma_3^{-1}\} \\
			\mathcal{W}(d_2,d'_1) &= \{\mu_1(w)=\alpha_3\beta_2^{-1}\beta_3\beta_7^{-1}\gamma_2, \mu_0\mu_1(w)=\alpha_4^{-1}\beta_1\beta_5^{-1}\beta_4^{-1}\beta_6\gamma_1^{-1}\} \\
			\mathcal{W}(d_2,d'_2) &= \{\mu_1\mu_2(w)=\alpha_3\beta_2^{-1}\beta_3\beta_7^{-1}\gamma_4, \mu_0\mu_1\mu_2(w)=\alpha_4^{-1}\beta_1\beta_5^{-1}\beta_4^{-1}\beta_6\gamma_3^{-1}\} \\
			\mathcal{W}(d_i, d''_j) &= \mathcal{W}(d'_i, d''_j) = \mathcal{W}(d_i, d_j)=\mathcal{W}(d'_i, d'_j)=\mathcal{W}(d''_i, d''_j)=\emptyset
		\end{align*}
		for any $i,j \in \{1,2\}$. Thus, we have $|\Wchi|=8$.
	\end{exam}

	\subsection{Periodic projective resolutions arising from $\Wchi$}
	The strings in $\Wchi$ give rise to string modules with periodic minimal projective resolutions of period 4, which in turn give rise to tubes of rank 2 in the Auslander-Reiten quiver, as the next few results will show.
	
	\begin{thm}
		Let $w=\alpha w' \gamma \in \Wchi$. Then $M(w)$ has a periodic minimal projective resolution
		\begin{equation*}
			\cdots \rightarrow P_3 \rightarrow P_2 \rightarrow P_1 \rightarrow P_0 \rightarrow P_3 \rightarrow P_2 \rightarrow P_1 \rightarrow P_0 \rightarrow M(w) \rightarrow 0
		\end{equation*}
		with
		\begin{equation*}
			\Omega^2(M(w)) = M(\mu_1\mu_2(w)) \qquad \text{and} \qquad \Omega^4(M(w)) = M(w).
		\end{equation*}
		In addition, $\Omega(M(w))$ and $\Omega^3(M(w))$ are the following.
		\begin{enumerate}[label=(\alph*)]
			\item If $\alpha,\gamma \in Q_1$ then
			\begin{equation*}
				\Omega(M(w)) = M(\mu_0\mu_2(w)) \qquad \text{and} \qquad 	\Omega^3(M(w)) = M(\mu_0\mu_1(w)).
			\end{equation*}
			\item If $\alpha,\gamma \in Q_1\inv$ then
			\begin{equation*}
				\Omega(M(w)) = M(\mu_0\mu_1(w)) \qquad \text{and} \qquad 	\Omega^3(M(w)) = M(\mu_0\mu_2(w)).
			\end{equation*}
			\item If $\alpha \in Q_1\inv$ and $\gamma \in Q_1$ then
			\begin{equation*}
				\Omega(M(w)) = M(\mu_0\mu_1\mu_2(w)) \qquad \text{and} \qquad 	\Omega^3(M(w)) = M(\mu_0(w)).
			\end{equation*}
			\item If $\alpha \in Q_1$ and $\gamma \in Q_1\inv$ then
			\begin{equation*}
				\Omega(M(w)) = M(\mu_0(w)) \qquad \text{and} \qquad \Omega^3(M(w)) = M(\mu_0\mu_1\mu_2(w)).
			\end{equation*}
		\end{enumerate}
	\end{thm}
	\begin{proof}
		Suppose $w=\alpha w' \gamma \in \mathcal{W}(d,d')$ for some $\mathbb{D}$-triples $d=(x,y_1,y_2)$ and $d'=(x',y'_1, y'_2)$. We will begin by calculating the first syzygy in all four cases, and then show that this implies the results for the second, third and fourth syzygies.
	
		Case (a): $\alpha,\gamma \in Q_1$. In this case, $w'=w^-_1 w^+_2 \ldots w^-_n$ for some direct strings $w^+_i$ and inverse strings $w^-_i$. We then have
		\begin{equation*}
			P_0= P(y_1) \oplus \bigoplus_{i=1}^{\frac{n-1}{2}} P(s(w^+_{2i})) \oplus P(x').
		\end{equation*}
		The indecomposable projective $P(y_1)$ is uniserial radical cube zero such that we have an exact sequence
		\begin{equation} \label{eq:seqP(y_1)}
			0 \rightarrow \rad M(\overline{\alpha}) \rightarrow P(y_1) \rightarrow M(\alpha) \rightarrow 0,
		\end{equation}
		where $\overline{\alpha}$ is as defined in Section~\ref{sec:Involutions}. Each $P(s(w^+_{2i}))$ is biserial such that we have exact sequences
		\begin{equation} \label{eq:seq-+-}
			0 \rightarrow \rad M((\overline{w}^-_{2i-1})\inv(\overline{w}^+_{2i})\inv) \rightarrow P(s(w^+_{2i})) \rightarrow M(w^-_{2i-1}w^+_{2i}) \rightarrow 0,
		\end{equation}
		where $\overline{w}^-_{2i-1}$ and $\overline{w}^+_{2i}$ are as defined in Section~\ref{sec:Involutions}. The indecomposable projective $P(x')$ is triserial such that
		\begin{equation} \label{eq:heartP(x')}
			\rad P(x') / \soc P(x') = M(w'') \oplus S(y'_1) \oplus S(y'_2)
		\end{equation}
		where $w''$ is the inverse string obtained from the word $\overline{w}^-_n w^-_n$ by removing the first and last symbols.
		
		Since $M(w)$ is a string module, one can deduce from the above exact sequences that 
		\begin{equation*}
			\Omega(M(w)) = M(\overline{\alpha} (\overline{w}^-_1)\inv(\overline{w}^+_2)\inv\ldots(\overline{w}^-_n)\inv\overline{\gamma}^\ast),
		\end{equation*}
		where $\overline{\gamma}^\ast$ is the unique formal inverse such that $t(\overline{\gamma}^\ast)=y'_2$. Thus, $\Omega(M(w))=M(\mu_0\mu_2(w))$, as required.
		
		Case (b): $\alpha,\gamma \in Q_1\inv$. This is similar to Case (a). We instead have $w'=w^+_1 w^-_2 \ldots w^+_n$ for some direct strings $w^+_i$ and inverse strings $w^-_i$. So
		\begin{equation*}
			P_0= P(x) \oplus \bigoplus_{i=1}^{\frac{n-1}{2}} P(t(w^-_{2i})) \oplus P(y'_1).
		\end{equation*}
		In addition, we have 
		\begin{equation} \label{eq:heartP(x)}
			\rad P(x) / \soc P(x) = S(y_1) \oplus S(y_2) \oplus M(w''),
		\end{equation}
		where $w''$ is the direct string obtained from the word $w^+_1\overline{w}^+_1$ by removing the first and last symbols. In addition, we have exact sequences
		\begin{equation} \label{eq:seq+-+}
			0 \rightarrow \rad M((\overline{w}^-_{2i})\inv(\overline{w}^+_{2i+1})\inv) \rightarrow P(t(w^-_{2i})) \rightarrow M(w^-_{2i}w^+_{2i+1}) \rightarrow 0,
		\end{equation}
		\begin{equation} \label{eq:seqP(y'_1)}
			0 \rightarrow \rad M(\overline{\gamma}) \rightarrow P(y'_1) \rightarrow M(\gamma) \rightarrow 0,
		\end{equation}
		where $\overline{w}^\pm_{i}$ and $\overline{\gamma}$ are as defined in Section~\ref{sec:Involutions}. It follows that we have
		\begin{equation*}
			\Omega(M(w)) = M(\overline{\alpha}^\ast (\overline{w}^+_1)\inv(\overline{w}^-_2)\inv\ldots(\overline{w}^+_n)\inv\overline{\gamma}),
		\end{equation*}
		where $\overline{\alpha}^\ast$ is the unique arrow such that $s(\overline{\alpha}^\ast)=y_2$. Thus, $\Omega(M(w))=M(\mu_0\mu_1(w))$, as required.
		
		Case (c): $\alpha \in Q_1\inv$ and $\gamma \in Q_1$. We have $w'=w^+_1 w^-_2 w^+_3 \ldots w^-_n$ for some direct strings $w^+_i$ and inverse strings $w^-_i$. So
		\begin{equation*}
			P_0= P(x) \oplus \bigoplus_{i=1}^{\frac{n-2}{2}} P(t(w^-_{2i})) \oplus P(x').
		\end{equation*}
		Note that $\rad P(x') / \soc P(x')$ and $\rad P(x) / \soc P(x)$ are as in (\ref{eq:heartP(x')}) and (\ref{eq:heartP(x)}) respectively and we have exact sequences as in (\ref{eq:seq+-+}). Thus, we have
		\begin{equation*}
			\Omega(M(w)) = M(\overline{\alpha}^\ast (\overline{w}^+_1)\inv(\overline{w}^-_2)\inv(\overline{w}^+_3)\inv\ldots(\overline{w}^-_n)\inv\overline{\gamma}^\ast),
		\end{equation*}
		where $\overline{\alpha}^\ast$ is the unique arrow such that $s(\overline{\alpha}^\ast)=y_2$ and $\overline{\gamma}^\ast$ is the unique formal inverse such that $t(\overline{\gamma}^\ast)=y'_2$. Thus, $\Omega(M(w))=M(\mu_0\mu_1\mu_2(w))$, as required.
		
		Case (d): $\alpha \in Q_1$ and $\gamma \in Q_1\inv$. We have $w'=w^-_1 w^+_2 w^-_3 \ldots w^+_n$ for some direct strings $w^+_i$ and inverse strings $w^-_i$. So
		\begin{equation*}
			P_0= P(y_1) \oplus \bigoplus_{i=1}^{\frac{n}{2}} P(s(w^+_{2i})) \oplus P(y'_1).
		\end{equation*}
		Now note that we have exact sequences as in (\ref{eq:seqP(y_1)}), (\ref{eq:seq-+-}) and (\ref{eq:seqP(y'_1)}). So we have
		\begin{equation*}
			\Omega(M(w)) = M(\overline{\alpha} (\overline{w}^-_1)\inv(\overline{w}^+_2)\inv(\overline{w}^-_3)\inv\ldots(\overline{w}^+_n)\inv\overline{\gamma})=M(\mu_0(w)),
		\end{equation*}
		as required.
		
		To calculate the remaining syzygies we note that if $w$ is as in Case (a), then $\mu_0\mu_2(w) \in \Wchi$ is a string of the same form as in Case (b). On the other hand, if $w$ is as in Case (b), then $\mu_0\mu_1(w) \in \Wchi$ is a string of the same form as in Case (a). Thus, we have for Case (a)
		\begin{align*}
			\Omega^2(M(w)) &= M(\mu_0\mu_1\mu_0\mu_2(w))= M(\mu_1\mu_2(w)), \\
			\Omega^3(M(w)) &= M(\mu_0\mu_2\mu_1\mu_2(w))= M(\mu_0\mu_1(w)), \\
			\Omega^4(M(w)) &= M(\mu_0\mu_1\mu_0\mu_1(w))= M(w)
		\end{align*}
		and we have for Case (b)
		\begin{align*}
			\Omega^2(M(w)) &= M(\mu_0\mu_2\mu_0\mu_1(w))= M(\mu_1\mu_2(w)), \\
			\Omega^3(M(w)) &= M(\mu_0\mu_1\mu_1\mu_2(w))= M(\mu_0\mu_2(w)), \\
			\Omega^4(M(w)) &= M(\mu_0\mu_2\mu_0\mu_2(w))= M(w).
		\end{align*}
		Similarly, if $w$ is as in Case (c), then $\mu_0\mu_1\mu_2(w) \in \Wchi$ is a string of the same form as in Case (d). Likewise, if $w$ is as in Case (d), then $\mu_0(w) \in \Wchi$ is a string of the same form as in Case (c). Thus, we have for Case (c)
		\begin{align*}
			\Omega^2(M(w)) &= M(\mu_0\mu_0\mu_1\mu_2(w))= M(\mu_1\mu_2(w)), \\
			\Omega^3(M(w)) &= M(\mu_0\mu_1\mu_2\mu_1\mu_2(w))= M(\mu_0(w)), \\
			\Omega^4(M(w)) &= M(\mu_0\mu_0(w))= M(w)
		\end{align*}
		and we have for Case (d)
		\begin{align*}
			\Omega^2(M(w)) &= M(\mu_0\mu_1\mu_2\mu_0(w))= M(\mu_1\mu_2(w)), \\
			\Omega^3(M(w)) &= M(\mu_0\mu_1\mu_2(w)),\\
			\Omega^4(M(w)) &= M(\mu_0\mu_1\mu_2\mu_0\mu_1\mu_2(w))= M(w).
		\end{align*}
	\end{proof}

	\begin{cor}
		Let $w \in \Wchi$. Then $\tau M(w) = M(\mu_1\mu_2(w))$ and $\tau^2 M(w)= M(w)$.
	\end{cor}
	
	\begin{exam}
		Let $\chi$ be the Brauer configuration in Example~\ref{ex:WchiFinite} and let $A$ be the symmetric special multiserial algebra corresponding to $\chi$. Let $w=\alpha_1\beta_2^{-1}\beta_3\beta_7^{-1}\gamma_2 \in \Wchi$. Then we have a projective resolution
		\begin{multline*}
			\cdots\rightarrow P(x) \oplus P(z_3) \oplus P(y'_1) \rightarrow P(y_2) \oplus P(z_1) \oplus P(x') \rightarrow  \\
			\rightarrow P(x) \oplus P(z_3) \oplus P(y'_2) \rightarrow P(y_1) \oplus P(z_1) \oplus P(x') \rightarrow M(w) \rightarrow 0
		\end{multline*}
		with syzygies
		\begin{align*}
			\Omega(M(w))&=M(\mu_0\mu_2(w)),		&	\Omega^3(M(w))&=M(\mu_0\mu_1(w)),\\
			\Omega^2(M(w))&=M(\mu_1\mu_2(w)),	&	\Omega^4(M(w))&=M(w).
		\end{align*}
	\end{exam}
	
	\subsection{Direct hyperstrings} \label{sec:Hyperstrings}
	In special mulitserial algebras (and other generalisations of special biserial algebras), one can obtain generalisations of string modules. We will need a particular class of such generalisations in the subsection that follows, which we call \emph{(direct) hyperstring modules} and will present here.
	
	Given a symmetric special multiserial algebra $A$ associated to a Brauer configuration $\chi$, we will define a set of relations $\gg$ and $\ggg$ on the set $\Wchi$. For any $w,w'\in\Wchi$, we say that $w \gg w'$ if all of the following hold.
	\begin{enumerate}[label=(R\arabic*)]
		\item There exists $(x,t(w),s(w')) \in \Dchi$.
		\item $w=w_1 w_0 \alpha$ and $w'=\alpha' w'_0 w'_1$ for some non-zero substrings $w_1$ of $w$ and $w'_1$ of $w'$, and for some substring $w'_0=w_0\inv$ common to both $w$ and $w'$ (up to inverse) that is of maximal length.
		\item The last symbol of $w_1$ is an arrow and the first symbol of $w'_1$ is an arrow.
	\end{enumerate}
	We say that $w \ggg w'$ if and only if $w,w'\in\Wchi$ and $w'=\mu_1\mu_2(w\inv)$.
	
	\begin{defn}
		Let $\bw=(w_1, \ldots w_m)$ be a tuple such that for each $i$, we have $w_i \in \Wchi$ and $w_i \gg w_{i+1}$ or $w_i \ggg w_{i+1}$. Then we call $\bw$ a \emph{direct hyperstring} of the algebra $A$.
	\end{defn}
	
	Given a direct hyperstring $\bw=(w_1, \ldots w_m)$, one obtains a \emph{hyperstring module} $M(\bw)$ as follows. The underlying vector space structure of $M(\bw)$ is the same as
	\begin{equation*}
		\bigoplus_{i=1}^m M(w_i) = \bigoplus_{i=1}^m \langle b_{i,0}, \ldots, b_{i,n_i} \rangle
	\end{equation*}
	For each $i$, write
	\begin{equation*}
		w_i = \alpha_i \beta_{i, 2}\ldots \beta_{i,n_i-1} \gamma_i
	\end{equation*}
	and note that $\gamma_i \in Q_1$ if and only if $\alpha_{i+1} \in Q_1\inv$. We can then define linear maps $\phi_i \in \Hom_K(M(w_i), M(w_{i+1}))$ by
	\begin{equation*}
		\phi_i(b_{i,j})=
		\begin{cases}
			b_{i+1, 0}	& \text{if } \alpha_{i+1}\in Q_1\inv \text{ and } j=n_i-1 \\
			b_{i+1, 1}	& \text{if } \gamma_i\in Q_1\inv \text{ and } j=n_i \\
			0				& \text{otherwise.}
		\end{cases}
	\end{equation*}
	The action of $\delta \in A$ on $M(\bw)$ is then defined by
	\begin{equation*}
		(v_1,\ldots,v_m) \delta = 
		\begin{cases}
			(v_1\delta,\ldots, v_i\delta, v_{i+1}\delta + \phi_i(v_i), \ldots, v_m\delta)	&
			\text{if } \delta=\gamma_i\inv \text{ or } \delta=\alpha_{i+1}\inv, \\
			(v_1 \delta,\ldots,v_m \delta)	&	\text{otherwise},
		\end{cases}
	\end{equation*}
	for each $(v_1,\ldots,v_m) \in M(\bw)$, where $v_i \delta$ is defined by the action of $\delta$ on $M(w_i)$. Thus, each $\delta \in A$ that is neither $\gamma_i\inv$ nor $\alpha_{i+1}\inv$ for some $i \geq 1$ has a diagonal action on $M(\bw)$, whereas each $\gamma_i\inv$ or $\alpha_{i+1}\inv$ has a lower triangular action on $M(\bw)$. One can therefore view a (direct) hyperstring module as a number of string modules that have been `glued' together by the action of some arrow incident to a truncated edge in some $\mathbb{D}$-triple.
	
	The reasoning for calling these objects \emph{direct} hyperstrings is that we have defined them as strings (or string modules) glued together in some sort of `direct' arrangement (analogous to direct strings). One could easily generalise this definition to more of a `string-like' arrangement, where we can glue together string modules and `formal inverses' of string modules. We will be exploring this notion in a forthcoming paper. For now, we remark that direct hyperstrings modules correspond to symmetric clannish band modules in a related clannish algebra (see Appendix~\ref{sec:AppClans} and in particular Section~\ref{sec:HyperstringBand} for details). Consequently, we have the following.
	
	\begin{thm}[\cite{Clans},\cite{Deng},\cite{ClanMaps}] \label{thm:indecHyperstring}
		The module $M(\bw)$ is indecomposable for any direct hyperstring $\bw$.
	\end{thm}
	
	\subsection{Extensions between string modules from $\Wchi$} \label{sec:WchiExt}
	Suppose $w \in \Wchi$. We will explicitly describe the spaces $\Ext_A^1(M(w),\tau M(w))$ in terms of extensions. By the Auslander-Reiten formula,
	\begin{equation*}	
		\dim_K \uEnd_A(M(w)) = \dim_K \Ext_A^1(M(w),\tau M(w)),
	\end{equation*}
	where $\uEnd_A(M(w))$ is the algebra of stable endomorphisms of $M(w)$ (that is, endomorphisms that do not factor through a projective-injective module). More generally, we will denote the stable space of homomorphisms between string modules $M(w)$ and $M(w')$ by $\uHom_A(M(w),M(w'))$, and we denote the corresponding set of admissible pairs by $\sap(w,w')$. Since each pair $(w_+,w_-) \in \sap(w,w)$ corresponds to a basis element of $\uEnd_A(M(w))$, each pair $(w_+,w_-) \in \sap(w,w)$ also corresponds to a basis element of $\Ext_A^1(M(w),\tau M(w))$. We will describe this correspondence.
	
	Let $(w_+, w_-) \in \sap(w,w)$. We may (up to inverse) write $w=w_1 w_+ w_2 w_- w_3$ for some (possibly zero) substrings $w_1$, $w_2$ and $w_3$ of $w$. We will use this notation throughout the remainder of this subsection, and we will begin with some straightforward consequences of this notation.
	
	\begin{lem} \label{lem:pemWchi}
		Suppose $w=w_1 w_+ w_2 w_- w_3 \in \Wchi$ and that $(w_+,w_-) \in \sap(w,w)$ with $w_+=w_-$ and $|w_+|, |w_-|\neq 0$. Then $w_1$ and $w_3$ are not zero strings.
	\end{lem}
	\begin{proof}
		For readability, we will write $w_+=w_-=w_0$. By (W4) in the definition of $\Wchi$, there exist $(x,y_1,y_2),(x',y'_1,y'_2)\in\Dchi$ such that $s(w)=y_1$ and $t(w)=y'_1$. Thus, if $w_1$ is a zero string then $s(w_0)=y_1$. On the other hand, if $w_3$ is a zero string, then $t(w_0)=y'_1$.
		
		By the definition of $\Dchi$, both $y_1$ and $y'_1$ are truncated edges and incident to respective non-truncated vertices $v$ and $v'$ with $\mathfrak{m}(v) = \mathfrak{m}(v')=1$. Thus, the word $\beta\alpha$ is not a string, where $\beta\alpha$ is the rotation of $\mathfrak{C}_v$ with $s(\alpha) = y_1$. Now if we assume that $w_1$ is a zero string, then this implies that we must have either $\beta\alpha$ or $\beta\inv\alpha\inv$ as a subword of $w_0w_2w_0$, as we have $s(w_0)=y_1$ and there exists no other arrow $\beta'\neq \beta$ such that $t(\beta')=y_1$. This means $w$ is not a string, so our assumption that $w_1$ is a zero string must be false. We can also deduce that $w_3$ cannot be a zero string by the same reasoning with the cycle $\mathfrak{C}_{v'}$.
	\end{proof}
	
	\begin{lem} \label{lem:pemiWchi}
		Suppose $w=w_1 w_+ w_2 w_- w_3 \in \Wchi$ and that $(w_+,w_-) \in \sap(w,w)$ with $w_+=w_-\inv$ and $|w_+|,|w_-|\neq 0$. Then $w_1$ is zero if and only if $w_3$ is zero. Moreover, $w_2$ is non-zero.
	\end{lem}
	\begin{proof}
		Note that the first symbol of $w_-$ is the formal inverse of the last symbol of $w_+$. Now if $w_2$ is a zero string, then there exists a subword $\alpha\alpha\inv$ of $w$, which means that $w$ is not a string --- a contradiction. Thus, $w_2$ is non-zero.
	
		It remains to prove that $w_1$ is zero if and only if $w_3$ is zero. If $|w_1|=0$, then $s(w_+)=t(w_-)$ is a truncated edge in a $\mathbb{D}$-triple by the definition of $\Wchi$. This is possible only if $|w_3|=0$. The converse argument is identical, thus proving the lemma.
	\end{proof}
	
	\begin{lem}\label{lem:pmzWchi}
		Suppose $w=w_1 w_+ w_2 w_- w_3 \in \Wchi$ and that $(w_+,w_-) \in \sap(w,w)$ with $|w_+|=|w_-|=0$. Then $w_1$, $w_2$ and $w_3$ are non-zero.
	\end{lem}
	\begin{proof}
		Write $w_0=w_+=w_-$ and suppose for a contradiction that $w_2$ is a zero string. Then $w=w_1 w_0 w_3$, and $w_1$ and $w_3$ would then also have to be zero strings for $(w_0,w_0)$ to be an admissible pair of $w$, since $w_0$ cannot otherwise be a factor substring and image substring simultaneously. But then this would imply $|w|=0$, which violates property (W1) of $\Wchi$. So $|w_2| \neq 0$, as required.
	
		Now suppose for a contradiction that either $|w_1|=0$ or $|w_3|=0$. It then follows from the combinatorics defining elements of $\Dchi$ that we can only have $(w_0,w_0)\in\ap(w,w)$ if $(x,y_1,y_2)=(x',y'_1,y'_2)$ and both $|w_1|=0$ and $|w_3|=0$. In particular, the map corresponding to $(w_0,w_0)\in\ap(w,w)$ factors through $P(x)$, so $(w_0,w_0)\not\in\sap(w,w)$, which contradicts the lemma statement. Thus, neither $w_1$ nor $w_3$ are zero strings.
	\end{proof}
	
	The lemmata above show that we may partition the set $\sap(w,w)$ into five subsets, defined as follows.
	\begin{enumerate}[label = (\roman*)]
		\item $\sap^0_w = \{(w_+,w_-) \in \sap(w,w) : |w_+|=|w_-|=0\}$
		\item $\sap^1_w = \{(w_+,w_-) \in \sap(w,w) : w_+=w_- \neq w  \text{ and } |w_+|,|w_-|\neq0\}$
		\item $\sap^2_w = \{(w_+,w_-) \in \sap(w,w) : w_+=w_-^{-1} \text{ and } |w_+|,|w_-|,|w_1|, |w_3|\neq 0\}$
		\item $\sap^3_w = \{(w_+,w_-) \in \sap(w,w) : w_+=w_-^{-1},  |w_1|=|w_3|=0 \text{ and } |w_+|,|w_-|\neq0\}$
		\item $\sap^4_w = \{(w,w)\}$
	\end{enumerate}
	This defines five cases that we must consider. We will now prove a number of technical results that will be used to construct the corresponding short exact exact sequences.
		
	\begin{lem} \label{lem:sap1}
		Suppose $w=w_1 w_+ w_2 w_- w_3 \in \Wchi$ and that $(w_+,w_-) \in \sap^1_w$. Then there exist strings
		\begin{align*}
			w'_1 &= w_1 w_0 w_3\in \Wchi \text{ and} \\
			w'_2 &= w_1 w_0 w_2 w_0 w_2 w_0 w_3\in \Wchi,
		\end{align*}
		where $w_0=w_+=w_-$.
	\end{lem}
	\begin{proof}
		If $w_0$ is neither a direct nor inverse string, then the result is straightforward, as then the existence of $w_0w_2w_0 \in \str_A$ implies that $w_0w_2$ is a band (or a proper power of a cyclic substring), and hence $w_1(w_0w_2)^r w_0 w_3 \in \Wchi$ for any $r \geq 0$. Thus we will assume that $w_0$ is direct or inverse.
		
		Firstly, we will prove the existence of $w'_1$. By the properties of admissible pairs, note that there exists an inverse substring $\delta_k\inv\ldots\delta_1\inv$ at the end of $w_1$, where $k$ is assumed to be maximal. Further note that there exists an inverse substring $\zeta_l\inv\ldots\zeta_1\inv$ at the start of $w_3$, where $l$ is assumed to be maximal. There are now two subcases to consider in this proof --- either $w_0$ is direct or $w_0$ is an inverse string. In the former case $s(w_0)$ is a peak of $w$ and $t(w_0)$ is a deep of $w$, and since $w_1 w_0, w_0 w_3 \in \str_A$, it follows that $w'_1=w_1 w_0 w_3 \in \str_A$.
		
		Now suppose instead that $w_0$ is an inverse string. In this subcase, we suppose for a contradiction that $w'_1 \not\in \str_A$. This is possible only if
		\begin{equation*}
			\zeta_1\ldots\zeta_l w_0\inv \delta_1\ldots\delta_k \not\in \str_A,
		\end{equation*}
		which is possible only if
		\begin{equation*}
			\cycle_{v,\zeta_1}^{\mult(v)} = \zeta_1\ldots\zeta_l w_0\inv \delta_1\ldots\delta_i
		\end{equation*}
		for some $0 \leq i < k$, where $v = \kappa(\sgerm(\zeta_1))$. It is easy to see from this that we have a commutative diagram
		\begin{equation*}
			 \xymatrix{
			 	M(w)
			 		\ar[r]^f
			 		\ar@{->>}[d]
				& M(w)	\\
			 	M(w_0\inv \delta_1\ldots\delta_i)
			 		\ar@{^{(}->}[r]
				& P(s(\zeta_1))
					\ar[u]
			}
		\end{equation*}
		where $f \in \End_A(M(w))$ corresponds to $(w_+,w_-) \in \ap(w,w)$. But then $(w_+,w_-) \not\in \sap(w,w)$, which completes the contradiction. So we must have $w'_1 \in \str_A$.  It is straightforward to see that we also have $w'_1 \in \Wchi$.
		
		Now we will prove the existence of $w'_2$. By the properties of admissible pairs, note that the symbols immediately succeeding $w_+$ and preceding $w_-$ in $w$ are both arrows. We have two cases to consider --- either $w_0$ is inverse or $w_0$ is a direct string. In the former case, we simply note that $t(w_0)$ is a peak of $w$ and that $t(w_2)$ is a deep of $w$. Since $w_0w_2w_0 \in \str_A$, it follows from this that $(w_0w_2)^r \in \str_A$ for any $r \geq 1$, which implies that $w'_2 \in \str_A$.
		
		Now suppose that $w_0$ is a direct string. This splits into two subcases. Namely, either $w_2$ is direct or not direct. We will work with the subcase that $w_2$ is not direct first. Let $\delta_1\ldots\delta_k$ and $\zeta_1\ldots \zeta_l$ be the direct strings at the start and end of $w_2$, where $k$ and $l$ are assumed to be maximal. Suppose for a contradiction that $w'_2$ does not exist. Since we know that $w_0w_2w_0 \in \str_A$, this is possible only if $w_2 w_0 w_2 \not\in \str_A$, which implies
		\begin{equation*}
			\cycle_{v,\zeta_1}^{\mult(v)} = \zeta_1\ldots\zeta_l w_0 \delta_1\ldots\delta_i
		\end{equation*}
		for some $0 \leq i < k$, where $v = \kappa(\sgerm(\zeta_1))$. It is then easy to see that the map corresponding to $(w_+,w_-) \in \ap(w,w)$ factors through $P(s(\zeta_1))$ and so $(w_+,w_-) \not\in \sap(w,w)$, which is a contradiction.
		
		We have the final subcase where $w_2$ is direct to consider. Write $w_2=\delta_1\ldots\delta_k$ (or write $k=0$ if $w_2$ is a zero string) and $w_0=\delta_{k+1}\ldots\delta_l$. Suppose for a contradiction that $w'_2 \not\in \str_A$. This is possible only if
		\begin{equation*}
			\cycle_{v,\delta_{k+1}}^{\mult(v)} = w_0 w_2 w_0 \delta_1\ldots\delta_i
		\end{equation*}
		for some $0 \leq i < l$, where $v = \kappa(\sgerm(\delta_{k+1}))$. It is then easy to see that the map corresponding to $(w_+,w_-) \in \ap(w,w)$ factors through $P(s(w_0))$ and so $(w_+,w_-) \not\in \sap(w,w)$, which is a contradiction. All possible cases have been considered, and hence we must have $w'_2 \in \str_A$. It is straightforward to see that we also have $w'_2 \in \Wchi$.
	\end{proof}
	
	\begin{lem} \label{lem:sap2}
		Suppose $w=w_1 w_+ w_2 w_- w_3 \in \Wchi$ and that $(w_+,w_-) \in \sap^2_w$. Then there exist strings
		\begin{align*}
			w'_1 &= w_3\inv w_+ w_2\inv w_-w_3 \in \Wchi \text{ and} \\
			w'_2 &= w_1 w_+ w_2 w_- w_1\inv \in \Wchi.
		\end{align*}
	\end{lem}
	\begin{proof}
		We will first prove the existence of $w'_1$. Since $w_+ w_2 w_- w_3 \in \str_A$, we know that
		\begin{equation*}
			w_3\inv w_-\inv w_2\inv w_+\inv =w_3\inv w_+ w_2\inv w_- \in \str_A.
		\end{equation*}
		Since $w_- w_3 \in \str_A$ and the first symbol of $w_3$ is a formal inverse, if $w_-$ has an arrow, it is easy to see that $w'_1 \in \str_A$. On the other hand, if $w_-$ is an inverse string, then let $\delta_k\inv\ldots\delta_1\inv$ be the inverse string (which is necessarily non-zero) at the end of $w_2\inv$ and let $\zeta_l\inv\ldots\zeta_1\inv$ be the inverse substring at the start of $w_3$, where $k$ and $l$ are assumed to be maximal. Suppose for a contradiction that $w'_1 \not\in \str_A$. Then
		\begin{equation*}
			\cycle_{v,\zeta_1}^{\mult(v)} = \zeta_1\ldots\zeta_l w_+ \delta_1\ldots\delta_i
		\end{equation*}
		for some $0 \leq i <k$, where $v = \kappa(\sgerm(\zeta_1))$. It is then easy to see that the map corresponding to $(w_+,w_-) \in \ap(w,w)$ factors through $P(s(\zeta_1))$ and so $(w_+,w_-) \not\in \sap(w,w)$, which is a contradiction. So $w'_1 \in \str_A$.
		
		The proof for the existence of $w'_2$ is similar to the proof for $w'_1$. It is also straightforward to see that both strings belong to $\Wchi$.
	\end{proof}
	
	\begin{rem} \label{rem:sap2}
		For any $(w_+,w_-) \in \sap^0_w \cup \sap^2_w$, we know that the first and last symbols of $w_2$ (which is non-zero) are arrows, and both the last symbol of $w_1$ and first symbol of $w_3$ are formal inverses. In particular, the symbol immediately preceding $w_\pm$ in $w$ (resp. $w\inv$) is a formal inverse if and only if the symbol immediately succeeding $w_\pm$ in $w$ (resp. $w\inv$) is an arrow.  In contrast to this, the symbols immediately preceding and succeeding $w_-$ in $w'_1$ and $w'_2$ of Lemma~\ref{lem:sap2} are either both arrows or both formal inverses. A consequence of this is that $w'_1 \neq w, w\inv$ and $w'_2 \neq w, w\inv$.
		
		By extension to this, $\mu_1(w'_1)\neq w, w\inv$ and $\mu_2(w'_2) \neq w, w\inv$, since $\mu_1$ and $\mu_2$ respectively replace the first and last symbols of a string in $\Wchi$ with a corresponding symbol (arrows are replaced with arrows, and formal inverses are replaced with formal inverses).
	\end{rem}
	
	\begin{lem}\label{lem:sap0}
		Suppose $w=w_1 w_+ w_2 w_- w_3 \in \Wchi$ and that $(w_+,w_-) \in \sap^0_w$. Then precisely one of the following statements is true.
		\begin{enumerate}[label=(\roman*)]
			\item There exist strings $w'_1 = w_1 w_3\in \Wchi$ and $w'_2 = w_1 w^2_2 w_3\in \Wchi$.
			\item There exist strings $w'_1 = w_3\inv w_2\inv w_3 \in \Wchi$ and $w'_2 = w_1 w_2 w_1\inv \in \Wchi$.
		\end{enumerate}
	\end{lem}
	\begin{proof}
		By Lemma~\ref{lem:pmzWchi}, $w_1$, $w_2$ and $w_3$ are non-zero strings. By Remark~\ref{rem:sap2}, $w_+$ is at a peak of the string, which by property (W3), implies that $|s(w_+)|=|s(w_-)|=2$. A consequence of this is that if $\tgerm(w_1) = \sgerm(w_3)$ then $\sgerm(w_2) = \tgerm(w_2) \neq \tgerm(w_1) = \sgerm(w_3)$.  On the other hand, if $\tgerm(w_1) \neq \sgerm(w_3)$, then $\sgerm(w_2) = \sgerm(w_3)$ and $\tgerm(w_2) = \tgerm(w_1)$. This outlines the two cases that we must consider.
		
		 So first suppose that $\tgerm(w_1) = \sgerm(w_3)$. Then the last symbol $\alpha$ of $w_2$ and the first symbol $\beta$ of $w_1\inv$ is such that $\alpha\beta$ is not a subpath of any cycle $\mathfrak{C}_v$ for any vertex $v \in \Vchi$. Thus, $\alpha\beta$ is a relation of the algebra and the string $w'_2$ in (ii) cannot exist. On the other hand, the string $w'_1$ from (i) exists by an identical proof to that used in Lemma~\ref{lem:sap1} with $w_0=w_+=w_-$ assumed to be inverse. Similarly, the string $w'_2$ from (i) exists by an identical proof to that used in Lemma~\ref{lem:sap1} with $w_0=w_+=w_-$ assumed to be direct.
		 
		Now suppose that $\tgerm(w_1) \neq \sgerm(w_3)$. Then the last symbol $\alpha$ of $w_1$ and the first symbol $\beta$ of $w_3$ is such that the path $\beta\inv\alpha\inv$ is not a subpath of any cycle $\mathfrak{C}_v$ for any vertex $v \in \Vchi$. Thus, $\beta\inv\alpha\inv$ is a relation of the algebra and the string $w'_1$ in (i) cannot exist. On the other hand, the strings $w'_1$ and $w'_2$ from (ii) exist by an identical proof to that used in Lemma~\ref{lem:sap2}.
	\end{proof}
	
	\begin{lem} \label{lem:sap3card}
		$|\sap^3_w| \leq 1$
	\end{lem}
	\begin{proof}
		Suppose that $\sap^3_w$ is non-empty and let $(w_+,w_-) \in \sap^3_w$. Suppose for a contradiction that there exists a pair $(z_+,z_-) \in \sap^3_w$ distinct from $(w_+,w_-)$. Then either $z_+$ is a substring of $w_+$ or vice versa. If $z_+$ is a substring of $w_+$, then we may write $w_+ = z_+ z_0$ and $w_- = z_0\inv z_-$. Since $(z_+,z_-)$ is an admissible pair, the first symbol of $z_0$ and the last symbol of $z_0\inv$ are arrows. But this implies that the first symbol of $z_0$ is both an arrow and formal inverse, which is impossible. So $(z_+,z_-) \not\in \sap^3_w$. A similar argument follows if $w_+$ is a substring of $z_+$. Hence, $|\sap^3_w| \leq 1$.
	\end{proof}
	
	\begin{lem} \label{lem:sap3}
		Suppose $w=w_+ w_2 w_- \in \Wchi$ and that $(w_+,w_-) \in \sap^3_w$. Then there exists a direct hyperstring $\bw=(w,\mu_1\mu_2(w))$.
	\end{lem}
	\begin{proof}
		Since we do not have $w \ggg \mu_1\mu_2(w)$, we are required to show that $w \gg \mu_1\mu_2(w)$. Property (R1) is automatically satisfied by the fact that $w \in \Wchi$ and the definition of $\mu_1$. Now note that $|w_+|>2$, which follows from property (W4) of the definition of $\Wchi$ and the combinatorics of elements in $\Dchi$. We may thus rewrite $w$ and $\mu_1\mu_2(w)$ as 
		\begin{equation*}
			w = w_+ w_2 w_0 \alpha \qquad \text{and} \qquad
			\mu_1\mu_2(w) = \alpha' w'_0 w_2 w'_-,
		\end{equation*}
		where $w_0$ is the substring of $w_-$ obtained by removing the last symbol, $w'_0$ is the substring of $w_+$ obtained by removing the first symbol, and $w'_-$ is obtained from $w_-$ by changing the last symbol in accordance with $\mu_2$. It then follows that $w_0=w_0\inv$ is the common substring (up to inverse) of $w$ and $\mu_1\mu_2(w)$ in property (R2), which is maximal due to Lemma~\ref{lem:sap3card}. Since $(w_+,w_-) \in \ap(w,w)$, we know that the first symbol of $w_2$ is an arrow (also, the last symbol). Since this is trivially the first symbol of $w_2 w'_-$ (also, the last symbol of $w_+ w_2$), (R3) is satisfied. Thus $w \gg \mu_1\mu_2(w)$ and hence $\bw=(w,\mu_1\mu_2(w))$ is a direct hyperstring.
	\end{proof}
	
	\subsubsection{The exact sequences} \label{sec:ExactSequences}
	For any $(w_+,w_-) \in \sap^1_w$, it follows from Lemma~\ref{lem:sap1} that there exists a corresponding non-split exact sequence
	\begin{equation} \tag{E1} \label{eq:E1}
		\xymatrix@1{
			0
				\ar[r]
			& M(\mu_1\mu_2(w))
				\ar[r]
			& M(\mu_1(w'_1)) \oplus M(\mu_2(w'_2))
				\ar[r]^-{\left(\begin{smallmatrix} f_1 & f_2 \end{smallmatrix}\right)}
			& M(w)
				\ar[r]
			& 0
		},
	\end{equation}
	where
	\begin{align*}
		w'_1 &= w_1 w_0 w_3, \\
		w'_2 &= w_1 w_0 w_2 w_0 w_2 w_0 w_3,
	\end{align*}
	and $f_1$ and $f_2$ correspond to the admissible pairs
	\begin{align*}
		(w_0 w_3, w_0 w_3) &\in \ap(\mu_1(w'_1),w), \\
		(w_1 w_0 w_2 w_0, w_1 w_0 w_2 w_0) &\in \ap(\mu_2(w'_2),w),
	\end{align*}
	respectively. We will denote the set of all exact sequences corresponding to elements of $\sap^1_w$ by $\es^1_w$.
	
	For any $(w_+,w_-) \in \sap^2_w$, it follows from Lemma~\ref{lem:sap2} and Remark~\ref{rem:sap2} that there exists a corresponding non-split exact sequence
	\begin{equation} \tag{E2} \label{eq:E2}
		\xymatrix@1{
			0
				\ar[r]
			& M(\mu_1\mu_2(w))
				\ar[r]
			& M(\mu_1(w'_1)) \oplus M(\mu_2(w'_2))
				\ar[r]^-{\left(\begin{smallmatrix} f_1 & f_2 \end{smallmatrix}\right)}
			& M(w)
				\ar[r]
			& 0
		},
	\end{equation}
	where
	\begin{align*}
		w'_1 &= w_3\inv w_+ w_2\inv w_-w_3 \\
		w'_2 &= w_1 w_+ w_2 w_- w_1\inv,
	\end{align*}
	and $f_1$ and $f_2$ correspond to the admissible pairs
	\begin{align*}
		(w_-w_3,w_-w_3) &\in \ap(\mu_1(w'_1),w), \\
		(w_1 w_+ w_2 w_-,w_1 w_+ w_2 w_-) &\in \ap(\mu_2(w'_2),w),
	\end{align*}
	respectively. We will denote the set of all exact sequences corresponding to elements of $\sap^2_w$ by $\es^2_w$.
	
	For any $(w_+,w_-) \in \sap^0_w$, it follows from Lemma~\ref{lem:sap0} and Remark~\ref{rem:sap2} that there exists a corresponding non-split exact sequence that is of the form of either (\ref{eq:E1}) or (\ref{eq:E2}). We will denote the set of all exact sequences corresponding to elements of $\sap^0_w$ by $\es^0_w$. 
	
	If there exists $(w_+,w_-) \in \sap^3_w$, we know from Theorem~\ref{thm:indecHyperstring}, Lemma~\ref{lem:sap3card} and Lemma~\ref{lem:sap3} that there exists a corresponding non-split exact sequence
	\begin{equation*}
		\xymatrix@1{
		0 
			\ar[r] 
		& M(\mu_1\mu_2(w)) 
			\ar[rrr]^-{\left(\begin{smallmatrix} 0 \\ 1_{M(\mu_1\mu_2(w))} \end{smallmatrix}\right)}
		&&& M(\bw_3) 
			\ar[rrr]^-{\left(\begin{smallmatrix} 1_{M(w)} & 0 \end{smallmatrix}\right)}
		&&& M(w) 
			\ar[r] 
		& 0},
	\end{equation*}
	where $\bw_3=(w,\mu_1\mu_2(w))$ and we use the fact that $M(\bw_3)$ has the vector space structure $M(w)  \oplus M(\mu_1\mu_2(w))$. We denote the (empty or singleton) set of the exact sequence corresponding to $(w_+,w_-) \in \sap^3_w$ by $\es^3_w$.
	
	Finally for the identity map, which corresponds to $(w,w) \in \sap^4_w$, we have a non-split exact sequence
	\begin{equation*}
		\xymatrix@1{
		0 
			\ar[r] 
		& M(\mu_1\mu_2(w)) 
			\ar[rrr]^-{\left(\begin{smallmatrix} 0 \\ 1_{M(\mu_1\mu_2(w))} \end{smallmatrix}\right)}
		&&& M(\bw_4) 
			\ar[rrr]^-{\left(\begin{smallmatrix} 1_{M(w)} & 0 \end{smallmatrix}\right)}
		&&& M(w) 
			\ar[r] 
		& 0},
	\end{equation*}
	where $\bw_4=(w,\mu_1\mu_2(w\inv))$ and we use the fact that $M(\bw_4)$ has the vector space structure $M(w)  \oplus M(\mu_1\mu_2(w\inv)) \cong M(w) \oplus M(\mu_1\mu_2(w))$. We denote the (singleton) set of the exact sequence corresponding to $(w,w) \in \sap^4_w$ by $\es^4_w$.
	
	\begin{rem}
		Recall that whilst $M(\bw_3)$ and $M(\bw_4)$ have a vector space structure isomorphic to $M(w) \oplus M(\mu_1\mu_2(w))$, this is not an isomorphism of $A$-modules. They are in fact indecomposable $A$-modules, as per Theorem~\ref{thm:indecHyperstring}.
	\end{rem}
	\begin{rem}
		Despite the similarities of $M(\bw_3)$ and $M(\bw_4)$, they are non-isomorphic. This is because they arise from distinct symmetric clannish bands. See Section~\ref{sec:ClanExample} for details, where $\bw_3$ in this section corresponds to the direct hyperstring $\bw_2$ in Section~\ref{sec:ClanExample} and $\bw_4$ in this section corresponds to the direct hyperstring $\bw_1$ in Section~\ref{sec:ClanExample}.
	\end{rem}
	
	\subsection{Tubes arising from $\Wchi$} \label{sec:WchiTubes}
	We will now show that as well as being $\tau$-periodic, the modules $M(w)$ arising from strings $w \in \Wchi$ are mouth modules of tubes of rank 2. To prove this, we will construct a quiver $Q^w$ and a (possibly trivial) admissible ideal $I^w$ such that the algebra $A_w = K Q^w / I^w$ is clannish. For a recollection on clannish algebras, see Appendix~\ref{sec:AppClans}. In the simplest cases, where the start and end of $w$ belong to distinct $\mathbb{D}$-triples in $\chi$, $A_w$ is simply a hereditary algebra of type $\Dtilde_{|w|+2}$. We will then construct a functor $T_w\colon \Mod* A_w \rightarrow \Mod*A$. The idea behind this construction is that $T_w$ will map the Auslander-Reiten sequences in the rank 2 tubes of $\Mod* A_w$ to Auslander-Reiten sequences in some rank 2 tubes of $\Mod* A$.
	
	So let $w=\alpha\beta_1\ldots\beta_n\gamma \in \mathcal{W}(d,d')$, for some $\mathbb{D}$-triples $d=(x,y_1,y_2)$ and $d'=(x',y'_1,y'_2)$. If $\es^3_w = \emptyset$, then proceed with the following construction. From $w$, first define the quiver $Q^w$ to be the following quiver of type $\Dtilde_{|w|+2}$
	\begin{equation*}
		\xymatrix{
		-1
			\ar@{-}[dr]^-{a}
		&&&& n+1
			\ar@{-}[dl]_-{c}
		\\
		& 0
			\ar@{-}[r]^-{b_1}
		& \cdots
			\ar@{-}[r]^-{b_n}
		& n	
		\\
		-2
			\ar@{-}[ur]^-{a^\ast}
		&&&& n+2
			\ar@{-}[ul]_-{c^\ast}
		}
	\end{equation*}
	where the orientation of both $a$ and $a^\ast$ point to the right if $\alpha \in Q_1$ and point to the left if $\alpha \in Q_1\inv$. Similarly, each $b_i$ points to the right (resp. left) if $\beta_i \in Q_1$ (resp. $\beta_i \in Q_1\inv$). Finally, both $c$ and $c^\ast$ point to the right (resp. left) if $\gamma \in Q_1$ (resp. $\gamma \in Q_1\inv$). We will define $I^w$ to be the zero ideal in this case.
	
	We note that in this construction, the algebra $A_w$ is clannish (in fact, skewed-gentle in the sense of \cite{Pena} and \cite{BekkertSkewGentle}). See Appendix~\ref{sec:AppClans} for a recollection. Specifically, in the first construction, $A_w$ is isomorphic to the clannish algebra $K\widehat{Q}^w / \langle \widehat{h}_1^2 - \widehat{h}_1, \widehat{h}_2^2 - \widehat{h}_2 \rangle$, where $\widehat{Q}^w$ is the quiver
	\begin{center}
		\begin{tikzpicture}
			\draw (0,0) node {$
				\xymatrix@1{
				-1
					\ar@{-}[r]^-{\widehat{a}}
				& 0
					\ar@{-}[r]^-{\widehat{b}_1}
				& \cdots
					\ar@{-}[r]^-{\widehat{b}_n}
				& n	
					\ar@{-}[r]^-{\widehat{c}}
				& n+1
			}$
			};
			\draw[<-] (3.1,0.05) arc (140.0037:-140:0.3);
			\draw (3.9,-0.15) node {\tiny$\widehat{h}_2$};
			\draw[->] (-3.1,0.05) arc (39.9963:320:0.3);
			\draw (-3.9,-0.15) node {\tiny$\widehat{h}_1$};
		\end{tikzpicture}
	\end{center}
	with the orientation of each non-loop arrow $\widehat{r} \in \widehat{Q}^w_1$ the same orientation as its corresponding arrow $r \in Q^w_1$.
	
	On the other hand, if we instead have $\es^3_w \neq \emptyset$, then note that there exists some $j<n$ such that 
	\begin{equation*}
		\alpha\beta_1\ldots \beta_j = \gamma\inv \beta_n\inv \ldots \beta_{n-j+1}\inv.
	\end{equation*}
	In this case, proceed with the following alternative construction. Let $Q^w$ be the quiver with underlying graph
	
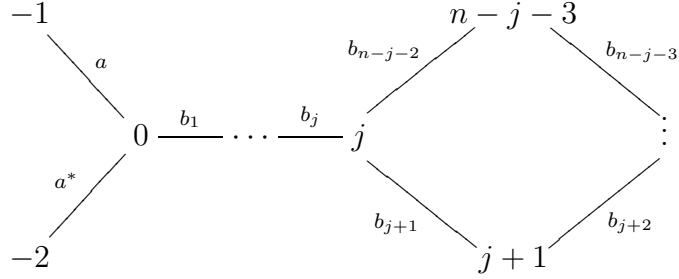
\begin{equation*}
		\xymatrix{
		-1
			\ar@{-}[dr]^-{a}
		&&&& n-j-1
			\ar@<1ex>@{-}[dr]^-{b_{n-j-1}}
		\\
		& 0
			\ar@{-}[r]^-{b_1}
		& \cdots
			\ar@{-}[r]^-{b_{j}}
		& j
			\ar@<-1ex>@{-}[dr]_-{b_{j+1}}
			\ar@<1ex>@{-}[ur]^-{b_{n-j}}
		&& \vdots
		\\
		-2
			\ar@{-}[ur]^-{a^\ast}
		&&&& j+1
			\ar@<-1ex>@{-}[ur]_-{b_{j+2}}
		}
	\end{equation*}
	The orientation of $Q^w$ is determined as follows. The orientation of both $a$ and $a^\ast$ point to the right if $\alpha \in Q_1$ and point to the left if $\alpha \in Q_1\inv$. Similarly for each $i \leq n-j$, $b_i$ points to towards $i$ (resp. towards $i-1$) if $\beta_i \in Q_1$ (resp. $\beta_i \in Q_1\inv$), where $n-j$ is identified with $j$. We know from the previous subsection that $\beta_{j+1}$ and $\beta_{n-j}$ are arrows, so define $I^w=\langle b_{n-j} b_{j+1} \rangle$.
	
	 The algebra $A_w$ in this alternative construction is isomorphic to the clannish (or skewed-gentle) algebra $K\widehat{Q}^w / \langle \widehat{h}^2 - \widehat{h}, \widehat{b}_{n-j} \widehat{b}_{j+1} \rangle$, where $\widehat{Q}^w$ is the quiver
	
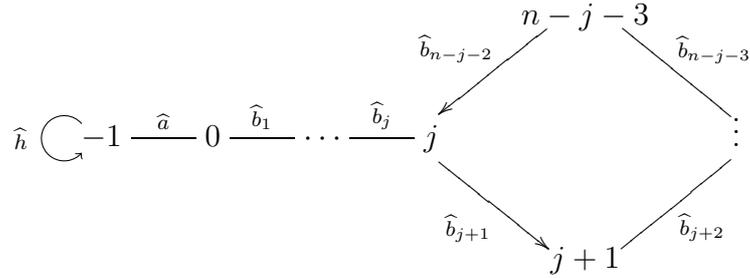
\begin{center}
		\begin{tikzpicture}
			\draw (0,0) node {$
				\xymatrix{
				&&&& n-j-1
					\ar@<1ex>@{-}[dr]^-{\widehat{b}_{n-j-1}}
				\\
				-1
					\ar@{-}[r]^-{\widehat{a}}
				& 0
					\ar@{-}[r]^-{\widehat{b}_1}
				& \cdots
					\ar@{-}[r]^-{\widehat{b}_{j}}
				& j
					\ar@<-1ex>@{->}[dr]_-{\widehat{b}_{j+1}}
					\ar@<1ex>@{<-}[ur]^-{\widehat{b}_{n-j}}
				&& \vdots
				\\
				&&&& j+1
					\ar@<-1ex>@{-}[ur]_-{\widehat{b}_{j+2}}
				}$
			};
			\draw[->] (-4.4,0.2) arc (39.9963:320:0.3);
			\draw (-5.2,0) node {\tiny$\widehat{h}$};
		\end{tikzpicture}
	\end{center}
	with the orientation of each non-loop arrow $\widehat{r} \in \widehat{Q}^w_1$ the same orientation as its corresponding arrow $r \in Q^w_1$. This is precisely the example given in Section~\ref{sec:ClanExample}.
	
	For technical purposes, we will establish some additional notation regarding $Q^w$ and its relationship with $Q$. Denote by $e_i$ the stationary path at vertex $i$ in $Q^w$ and by $\varepsilon_z$ the stationary path at vertex $z$ in $Q$ (where $z$ may also be considered as a polygon in $\chi$). In alignment with the notation of Section~\ref{sec:Involutions}, if $\alpha \in Q_1$ (resp. $\alpha \in Q^{-1}_1$), then denote by $\alpha^\ast$ the unique arrow in $Q_1$ (resp. formal inverse in $Q_1\inv$) of source $y_2$. Similarly, if $\gamma \in Q_1$ (resp. $\gamma \in Q^{-1}_1$), then denote by $\gamma^\ast$ the unique arrow in $Q_1$ (resp. formal inverse in $Q_1\inv$) of target $y'_2$. Define a map $\rho\colon Q_1 \cup Q_1\inv \rightarrow Q_1$ that maps each symbol to its corresponding arrow. Finally, define a morphism of quivers $\xi\colon Q^w \rightarrow Q$ by
	\begin{equation*}
		\xi(\xymatrix@1{j \ar[r]^-{r} & k}) =
		\begin{cases}
			\xymatrix@1{s(\rho(\alpha)) \ar[r]^-{\rho(\alpha)} & t(\rho(\alpha))}						&	\text{if } r=a, \\
			\xymatrix@1{s(\rho(\alpha^\ast)) \ar[r]^-{\rho(\alpha^\ast)} & t(\rho(\alpha^\ast))}						&	\text{if } r=a^\ast, \\
			\xymatrix@1{s(\rho(\beta_i)) \ar[r]^-{\rho(\beta_i)} & t(\rho(\beta_i))}	&	\text{if } r=b_i, \\
			\xymatrix@1{s(\rho(\gamma)) \ar[r]^-{\rho(\gamma)} & t(\rho(\gamma))}					&	\text{if } r=c, \\
			\xymatrix@1{s(\rho(\gamma^\ast)) \ar[r]^-{\rho(\gamma^\ast)} & t(\rho(\gamma^\ast))}					&	\text{if } r=c^\ast.
		\end{cases}
	\end{equation*}
	
	\begin{defn}\label{def:Tw}
		Define a functor $T_w\colon\Mod*A_w \rightarrow \Mod*A$ as follows. For a given module $V \in \Mod*A_w$, the underlying vector space of $M=T_wV$ is the same as $V$. That is, as a vector space, we define $M=V$. We then enrich $M$ with the structure of an $A$-module by defining the actions point-wise for each $m \in M$ as
		\begin{align*}
			m \varepsilon_{z} &= \sum_{i \in Q^w_0 : \xi(i)=z} m e_i	& m \delta &= \sum_{r \in Q^w_1 : \xi(r)=\delta} m r \\
		\end{align*}
		for each vertex $z \in Q_0$ and each arrow $\delta \in Q_1$. For any given morphism $q\colon V \rightarrow V'$ of $A_w$-modules, we simply define $f=T_w q$ as $f(m) = q(m)$ for all $m \in M$.
	\end{defn}
	    
	It is easy to see that $T_w$ is an exact functor. Also note that for any $V \in \Mod* A_w$ and any non-trivial idempotent $f \in \End_A(T_w V)$, there exists a non-trivial idempotent $q=f \in \End_{A_w}(V)$. Thus, $T_w$ preserves indecomposable objects.
	
	\begin{rem} \label{rem:AwString}
		Given $w \in \Wchi$. There exists a string $\widehat{w} \in \str_{A_w}$ such that $T_w M(\widehat{w})=M(w)$. This is precisely the string of maximal length such that $s(\widehat{w})=-1$ and $t(\widehat{w})=n+1$ (if $\es^3_w= \emptyset$) or $t(\widehat{w})=-1$ (if $\es^3_w\neq \emptyset$).
	\end{rem}
	
	\begin{lem} \label{lem:es4Image}
		Let $\widehat{w} \in \str_{A_w}$ be the string in the above remark. Then the almost split sequence
		\begin{equation*}
			\xymatrix@1{
				0
					\ar[r]
				& \tau M(\widehat w)
					\ar[r]^-q
				& V
					\ar[r]^-p
				& M(\widehat w)
					\ar[r]
				& 0
			}
		\end{equation*}
		ending in $M(\widehat{w})$ has an indecomposable middle term. Moreover, we have
		\begin{equation*}
			\xymatrix@1{
				0
					\ar[r]
				& T_w \tau M(\widehat w)
					\ar[r]^-{T_w q}
				& T_w V
					\ar[r]^-{T_w p}
				& T_w M(\widehat w)
					\ar[r]
				& 0
			} \in \es_w^4.
		\end{equation*}
	\end{lem}
	\begin{proof}
		If $\es^3_w=\emptyset$, then this result follows from known results of representations of hereditary algebras of type $\Dtilde$. See \cite[XIII.2.6. Table]{BlueBookII} for an account of the mouth modules in the tube containing $M(\widehat w)$ and \cite[X.2]{BlueBookII} for the fact that $T_w V \cong M(\bw_4)$.
		
		If $\es^3_w\neq\emptyset$, then we are required to use the theory of clannish algebras. The result follows directly from Section~\ref{sec:AppClans}, with $KQ^w / I^w$ specifically being the algebra in the example of \ref{sec:ClanExample}, albeit with a slightly different labelling/notation --- the module $V$ is denoted by $M(\bw_1)$ in Section~\ref{sec:ClanExample}.
	\end{proof}
	
	\begin{lem} \label{lem:es3Image}
		Let $\widehat{w} \in \str_{A_w}$ be as in Remark~\ref{rem:AwString} and let $\bw_3$ be the hyperstring from the exact sequence in $\es^3_w$, which we assume to be non-empty. Then there exists an exact sequence
		\begin{equation*}
			\xymatrix@1{
				0
					\ar[r]
				& \tau M(\widehat w)
					\ar[r]^-q
				& V
					\ar[r]^-p
				& M(\widehat w)
					\ar[r]
				& 0
			}
		\end{equation*}
		in $\Mod*A_w$ with an indecomposable middle term that is not almost split. Moreover,
		\begin{equation*}
			\xymatrix@1{
				0
					\ar[r]
				& T_w \tau M(\widehat w)
					\ar[r]^-{T_w q}
				& T_w V
					\ar[r]^-{T_w p}
				& T_w M(\widehat w)
					\ar[r]
				& 0
			} \in \es_w^3.
		\end{equation*}
	\end{lem}
	\begin{proof}
		By the example in Section~\ref{sec:ClanExample}, we know that there exists a module $V \in \Mod*A_w$ such that $T_w V = M(\bw_3)$, where $\bw_3$ is the hyperstring defined in Section~\ref{sec:WchiExt}. This is the module $M(\bw_2)$ defined in Section~\ref{sec:ClanExample}. The existence of the maps $p$ and $q$ is then obvious from the definitions.
	\end{proof}
	
	\begin{lem} \label{lem:es12NotAS}
		None of the exact sequences in $\es_w^0 \cup \es^1_w \cup \es^2_w$ are almost split.
	\end{lem}
	\begin{proof}
		Consider an exact sequence
		\begin{equation} \tag{$\ast$} \label{eq:E1E2seq}
			\xymatrix@1{
				0
					\ar[r] &
				M(\mu_1\mu_2(w))
					\ar[r]^-{
						\left(\begin{smallmatrix}
							g_1 \\ g_2
						\end{smallmatrix}\right)
					} &
				M(\mu_1(w'_1)) \oplus M(\mu_2(w'_2))
					\ar[r]^-{
						\left(\begin{smallmatrix}
							f_1 & f_2
						\end{smallmatrix}\right)
					} &
				M(w)
					\ar[r] &
				0
			}
			\in \es_w^0 \cup \es^1_w \cup \es^2_w.
		\end{equation}
		Let $f_1 = \theta_1 \psi_1$ and $f_2 = \theta_2 \psi_2$, where $\psi_1$ and $\psi_2$ are the canonical surjections onto $\im f_1$ and $\im f_2$ respectively, and $\theta_1$ and $\theta_2$ are the canonical inclusions of $\im f_1$ and $\im f_2$ into $M(w)$. Now $\im f_1$ and $\im f_2$ are submodules of $M(w)$ and thus are respectively isomorphic to string modules $M(w''_1)$ and $M(w''_2)$ for some substrings $w''_1$ and $w''_2$ of $w$. In particular, we have
		\begin{equation*}
			M(w''_1) = T_w M(\widehat{w}''_1)	\qquad \text{and} \qquad	M(w''_2) = T_w M(\widehat{w}''_2)
		\end{equation*}
		for some substrings $\widehat{w}''_1$ and $\widehat{w}''_2$ of $\widehat{w}$ in $A_w$, where $\widehat w$ is as in Remark~\ref{rem:AwString}. In addition, there exist inclusion morphisms
		\begin{equation*}		
			\widehat{\theta}_1 \in \Hom_{A_w}(M(\widehat{w}''_1),M(\widehat{w}))	\qquad \text{and} \qquad
			\widehat{\theta}_2 \in \Hom_{A_w}(M(\widehat{w}''_2),M(\widehat{w})).
		\end{equation*}
		such that
		\begin{equation*}
			\theta_1 = T_w \widehat{\theta}_1	\qquad \text{and} \qquad	\theta_2 = T_w \widehat{\theta}_2.
		\end{equation*}
		
		Suppose for a contradiction that (\ref{eq:E1E2seq}) is almost split. Then the morphism $f=\left(\begin{smallmatrix} f_1 & f_2 \end{smallmatrix}\right)$ is right almost split. So $f$ is not a retraction and for any object $X \in \Mod*A$ and any morphism $u \in \Hom_A(X,M(w))$ that is not a retraction, there exists a morphism $v \in \Hom_A(X, M(\mu_1(w'_1)) \oplus M(\mu_2(w'_2)))$ such that $u=fv$. Now consider the almost split sequence
		\begin{equation} \tag{$\ast\ast$} \label{eq:KQwass}
			\xymatrix@1{
				0
					\ar[r] &
				\tau M(\widehat{w})
					\ar[r]^-{q} &
				V
					\ar[r]^-{p} &
				M(\widehat{w})
					\ar[r] &
				0
			}
		\end{equation}
		from Lemma~\ref{lem:es4Image}. Since $\widehat{\theta}=\left(\begin{smallmatrix} \widehat{\theta}_1 & \widehat{\theta}_2 \end{smallmatrix}\right)$ and $\theta = \left(\begin{smallmatrix} \theta_1 & \theta_2 \end{smallmatrix} \right) = \left(\begin{smallmatrix} T_w \widehat{\theta}_1 & T_w \widehat{\theta}_2\end{smallmatrix} \right)$ are both maps between string modules, it is clear from the string combinatorics that neither of these maps are retractions. Since (\ref{eq:KQwass}) is almost split in $\Mod* A_w$, neither $p$ nor $T_w p$ is a retraction. Furthermore, there exists a morphism $\widehat{\varphi}$ such that $\widehat{\theta}=p \widehat{\varphi}$. Thus, we have the following commutative diagram.
		\begin{equation*}
			\xymatrix{
				X
					\ar[dr]^-u
					\ar[d]^-v
				\\
				M(\mu_1(w'_1)) \oplus M(\mu_2(w'_2))
					\ar[r]_-{
						\left(\begin{smallmatrix}
							f_1 & f_2
						\end{smallmatrix}\right)
					}
					\ar[d]^-{
						\left(\begin{smallmatrix}
							\psi_1 	& 0 \\
							0 		& \psi_2 
						\end{smallmatrix}\right)
					} &
				M(w)
					\ar@{=}[dd]
				\\
				T_w M(\widehat{w}''_1) \oplus T_w M(\widehat{w}''_2)
					\ar[dr]^-{\quad \left(\begin{smallmatrix}
							T_w \widehat{\theta_1} & T_w \widehat{\theta_2}
						\end{smallmatrix}\right)
					}
					\ar[d]^-{T_w \widehat{\varphi}}
				\\
				T_w V
					\ar[r]^-{T_w p} &
				T_w M(\widehat{w})
			}
		\end{equation*}	
		Hence, for any map $u \in \Hom_A(X,M(w))$ that is not a retraction, there exists a map
		\begin{equation*}
			h = T_w \widehat{\varphi}
			\begin{pmatrix}
				\psi_1	&	0	\\	0	&	\psi_2
			\end{pmatrix}
			v \in \Hom_A(X, T_w V)
		\end{equation*}
		such that $u=(T_w p) h$. So the morphism $T_w p \in \Mod*A$ is right almost split. It therefore follows that the exact sequence
		\begin{equation} \tag{$\ast\ast\ast$}\label{eq:Wchiass}
			\xymatrix@1{
				0
					\ar[r]
				& T_w \tau M(\widehat w)
					\ar[r]^-{T_w q}
				& T_w V
					\ar[r]^-{T_w p}
				& T_w M(\widehat w)
					\ar[r]
				& 0
			}
		\end{equation}
		is almost split in $\Mod*A$. By Lemma~\ref{lem:es4Image}, the exact sequences (\ref{eq:E1E2seq}) and (\ref{eq:Wchiass}) correspond to different basis elements of
		\begin{equation*}
			\Ext_A^1(M(w), \tau M(w)).
		\end{equation*}
		But this implies that there are two distinct almost split sequences ending in $M(w)$, which is not possible since almost split sequences are unique up to equivalence. So our original assumption that (\ref{eq:E1E2seq}) is almost split must be false.
	\end{proof}
	
	\begin{prop}
		The exact sequence in $\es^4_w$ is almost split.
	\end{prop}
	\begin{proof}
		By Lemma~\ref{lem:es12NotAS}, we know that the exact sequences in $\es_w^1$ and $\es_w^2$ are not almost split. By Lemma~\ref{lem:es3Image} and Lemma~\ref{lem:es4Image}, we know that the exact sequences of $\es_w^3$ and $\es_w^4$ are in the image of $T_w$. Since the exact sequence in the pre-image of $\es_w^4$ is almost split in $A_w$, we know that $\es_w^3$ cannot be almost split by a similar argument to that used in the proof of Lemma~\ref{lem:es12NotAS}. Since the exact sequences in $\es_w^1,\ldots,\es_w^4$ determine all extensions of $M(w)$ by $\tau M(w)$, we conclude that the exact sequence in $\es_w^4$ is almost split.
	\end{proof}
	
	The above proposition implies that for any $w \in \Wchi$, the almost split sequence ending in $M(w)$ has an idecomposable middle term. Since we also know that $M(w)$ is $\tau$-periodic of period 2, with $\tau M(w) = M(\mu_1\mu_2(w))$, we have the following result.
	
	\begin{thm}
		Let $A$ be a symmetric special multiserial algebra corresponding to a Brauer configuration $\chi$.
		\begin{enumerate}[label=(\alph*)]
			\item Let $w \in \Wchi$. Then $M(w)$ is a mouth module of a tube of rank 2 in $_s\Gamma_A$.
			\item For any $w,w' \in \Wchi$, define an equivalence relation $w \sim w'$ if and only if $w = (\mu_1\mu_2)^i(w')$ for some $i \in\{0,1\}$. Then there exists a distinct tube $\mathcal{T}_{[w]}$ of rank 2 for each equivalence class $[w] \in {\Wchi}/{\sim}$.
		\end{enumerate}
	\end{thm}
	
	\appendix
	\section{Symmetric Clannish Band Modules} \label{sec:AppClans}
	Clannish and skewed-gentle algebras are important classes of tame algebras. They are certain generalisations the hereditary algebras of type $\mathbb{D}_n$ and $\Dtilde_n$ in a similar sense to the way string and gentle algebras generalise the hereditary algebras of type $\mathbb{A}_n$ and $\Atilde_n$. Clannish algebras are defined using the path algebra of a quiver modulo an inadmissible ideal. In particular, clannish algebras are defined with quivers that contain \emph{special loops} $\eta$ and ideals generated with zero relations and idempotent relations $\eta^2-\eta$. The indecomposable modules of clannish algebras have been classified in terms of indecomposable representations of clans. These are given in terms of \emph{clannish strings and bands}. However despite their name, clannish strings and bands have key differences with the strings and bands one would see in a special multiserial algebra. In this section, we will provide a brief summary on the representation theory of clannish algebras, with a particular focus on indecomposable modules that arise from symmetric clannish bands. The definitions and results in this section are sourced from \cite{Clans} and \cite{ClanMaps}.
	
	\subsection{Clannish algebras, words and bands}
	We will begin with the definition of a clannish algebra.
	\begin{defn}
		Let $Q$ be a finite quiver with a subset $Q_2 \subseteq Q_1$ of special loops and define a set of \emph{special relations}
		\begin{equation*}
			S = \{\eta^2-\eta : \eta \in Q_2\}.
		\end{equation*}
		An algebra $A$ is called \emph{clannish} if it is Mortia equivalent to $KQ/\langle Z \cup S\rangle$, where the following properties hold.
		\begin{enumerate}[label=(C\arabic*)]
			\item $Z$ is a set of zero-relations that do not start or end with a special loop, or involve the square of a special loop. \label{ClanLoopRel}
			\item For any vertex $v \in Q_0$, there are at most two arrows in $Q_1$ of source $v$, and at most two arrows in $Q_1$ of target $v$. \label{ClanSB1}
			\item For any arrow $\beta \in Q_1 \setminus Q_2$, there exists at most one arrow $\alpha \in Q_1$ such that $\alpha \beta \not\in Z$, and at most one arrow $\gamma \in Q_1$ such that $\beta \gamma \not\in Z$. \label{ClanSB2}
		\end{enumerate}
	\end{defn}
	
	\begin{rem}
		\ref{ClanSB1} implies that if there exists a special loop $\eta$ incident to two ordinary arrows $\alpha$ and $\beta$, then $s(\eta)$ is the source of precisely one of $\alpha$ or $\beta$ and the target of the other arrow. Moreover, \ref{ClanLoopRel} and \ref{ClanSB2} imply that that either $\alpha\beta \in Z$ (if $t(\alpha)=s(\eta)$) or $\beta\alpha \in Z$ (if $s(\alpha)=s(\eta)$).
	\end{rem}
	
	Properties \ref{ClanSB1} and \ref{ClanSB2} are somewhat analogous to the properties of a special biserial algebra. In fact there are many similarities in the combinatorics between clannish algebras and special biserial algebras. For example, the classification of indecomposable modules of clannish algebras is given by \emph{clannish strings and bands} and their associated modules. Clannish strings are not used in this paper, and so we will omit these from our discussion and instead focus on (symmetric) clannish bands.
	
	In what follows, it is important to clarify that by $\str_A$ and $\band_A$, we mean the set of strings and bands associated to $A$ in the sense of Section~\ref{sec:StringsBands} and that we are not necessarily referring to clannish strings and clannish bands. In addition, we allow $\eta\in\str_A$ but not $\eta^2$ for any special loop $\eta \in Q_2$, and we define the the formal inverse $\eta\inv = \eta$. That is, $\eta \in Q_1 \cap Q_1\inv$ for each $\eta \in Q_2$.
	
	Given a clannish algebra $A$ associated to a quiver $Q$, a \emph{clannish word} of $A$ is a concatenation of symbols $w=\alpha_1\ldots \alpha_n$ such that each $\alpha_i \in Q_1 \cup Q_1\inv$ and $w \in \str_A$. A clannish word $b$ is called a \emph{clannish band} if we also have $b \in \band_A$. A clannish band is called \emph{symmetric} if $b\inv$ is a cyclic permutation/rotation of $b$, or equivalently, if $b=w_1 \eta_1 w_1\inv w_2\inv \eta_2 w_2$ for some clannish subwords $w_1$ and $w_2$ and some special loops $\eta_1,\eta_2 \in Q_2$. Otherwise, the clannish band is called \emph{asymmetric}.
	
	\begin{rem}
		Note that we have deliberately avoided calling clannish words by the name `strings'. This is to avoid confusing the reader between strings and clannish strings, the latter of which must satisfy an additional maximality property over a \emph{clan} associated to the algebra. We refer the reader to \cite{Clans} and \cite{ClanMaps} for formal definitions and constructions of these notions rather than presenting them here, as they are not needed for this paper.
	\end{rem}
	
	For technical reasons, we will need to define a partial ordering $>$ on the set $(Q_1 \cup Q_1\inv) \setminus Q_2$. For any arrows $\alpha,\beta \in Q_1 \setminus Q_2$ such that $t(\alpha)=s(\beta)$ and $\alpha \neq \beta\inv$, but also $\alpha\beta \not\in \str_A$ (so $\alpha\beta$ does not avoid a relation of the algebra), we define $\beta\inv > \alpha$. This partial ordering extends to the set of all clannish words on $A$ in the following way. We say $w > w'$ if and only if $w = w_1 \alpha w_0$ and $w' = w_1' \alpha' w_0'$ with $w_0 = w'_0$ and $\alpha > \alpha'$.
	
	\subsection{Symmetric clannish band modules}
	Each symmetric clannish band gives rise to a family of indecomposable representations indexed by the finite-dimensional indecomposable modules of the algebra
	\begin{equation*}
		K\langle x,y \rangle / \langle x^2-x, y^2-y \rangle.
	\end{equation*}
	
	\begin{rem} \label{rem:clash}
		The full subcategory of $\MOD K\langle x,y \rangle / \langle x^2-x, y^2-y\rangle$ consisting of finite-dimensional indecomposable modules is equivalent to the full abelian subcategory of $\Mod* K\Dtilde_m$ (with $m>4$) consisting of all except one of the tubes in the stable Auslander-Reiten quiver (with $\Dtilde_m$ being of any orientation). We refer the reader to \cite{BlueBookII} for details, however we will remark that there are infinitely many homogeneous (rank 1) tubes indexed by $K \setminus \{0\}$ and precisely two tubes of rank 2 in this category. There is actually a sense of a clash in terminology here between clannish strings, clannish bands and classical strings and bands (over a special biserial algebra). The modules at the mouth of the tubes of $\Mod* K\Dtilde_4$ of rank two are actually string modules (in the classical sense), and yet they arise from symmetric clannish bands. This clash is present in Section~\ref{sec:rank2} of this paper, where we have string modules $M(w)$ with $w \in \Wchi$ that actually correspond to modules arising from symmetric clannish bands.
	\end{rem}
	
	Suppose $b=w\eta_1w\inv\eta_2$ is a symmetric clannish band of a clannish algebra $A$. From this, we will define a family of symmetric clannish band modules $M(b, V, \phi_x, \phi_y)$, where $V$ is an $m$-dimensional $K$-vector space and $\phi_x,\phi_y \in \End_K(V)$ such that $(V, \phi_x, \phi_y)$ is a representation of $K\langle x,y \rangle / \langle x^2 - x, y^2 - y \rangle$. First write
	\begin{equation*}
		w=\alpha_1\ldots\alpha_n.
	\end{equation*}
	The underlying vector space of $M(b, V, \phi_x, \phi_y)$ is then
	\begin{equation*}
		M(\ub, V, \phi_x, \phi_y) = \bigoplus_{i=1}^m \langle c_{i,0}, \ldots, c_{i,n} \rangle,
	\end{equation*}
	where for each $j$, we have
	\begin{equation*}
		\bigoplus_{i=1}^m \langle c_{i,j} \rangle = V.
	\end{equation*}
	$M(\ub, V, \phi_x, \phi_y)$ is given the structure of an $A$-module by defining the action of an ordinary arrow or stationary path $\beta \in A$ on $M(\ub, V, \phi_x, \phi_y)$ by
	\begin{equation*}
		c_{i,j}\beta = 
		\begin{cases}
			c_{i,j-1}	&	\text{if } \ubeta=\ualpha_j\inv, \\
			c_{i,j}		&	\text{if } \beta = \varepsilon_{t(\ualpha_j)} \text{ or } \varepsilon_{s(\ualpha_{j+1})}, \\
			c_{i,j+1}	&	\text{if } \ubeta = \ualpha_{j+1}, \\
			0			&	\text{otherwise},
		\end{cases}
	\end{equation*}
	and the action of a special loop $\eta \in Q_2$ by
	\begin{equation*}
		c_{i,j}\eta = 
		\begin{cases}
			\phi_x(c_{i,0})	&	\text{if } j=0 \text{ and } \eta=\eta_1, \\
			c_{i,j-1}			&	\text{if } 0<j<n\text{, } \ueta=\ualpha_j \text{ and } \ualpha_n\inv\ldots\ualpha_{j+1}\inv < \ualpha_1\ldots\ualpha_{j-1}, \\
			c_{i,j}				&	\text{if } 0<j<n\text{, } \ueta = \ualpha_{j} \text{ and } \ualpha_n\inv\ldots\ualpha_{j+1}\inv > \ualpha_1\ldots\ualpha_{j-1}, \\
								&	\text{or } 0<j<n\text{, } \ueta=\ualpha_{j+1} \text{ and } \ualpha_n\inv\ldots\ualpha_{j+2}\inv < \ualpha_1\ldots\ualpha_{j}, \\
			c_{i,j+1}			&	\text{if } 0<j<n \text{, } \ueta = \ualpha_{j+1} \text{ and } \ualpha_n\inv\ldots\ualpha_{j+2}\inv > \ualpha_1\ldots\ualpha_{j}, \\
			\phi_y(c_{i,n})	&	\text{if } j=n \text{ and } \eta=\eta_2, \\
			0					&	\text{otherwise}.
		\end{cases}
	\end{equation*}
	
	\begin{rem}
		Suppose for the symmetric clannish band $\ub = \uw \ueta_1 \uw\inv \ueta_2$, we have
		\begin{equation*}
			\uw=\ualpha_1\ldots\ualpha_{j-1} \ueta \ualpha_{j+1} \ldots \ualpha_n, 
		\end{equation*}
		for some $\eta \in Q_2$ and $1<j<n$. Then one may notice that if
		\begin{equation} \tag{$\ast$} \label{eq:WordGeq}
			\ualpha_n\inv\ldots\ualpha_{j+1}\inv > \ualpha_1\ldots\ualpha_{j-1},
		\end{equation}
		then for each $i$, $\eta$ acts on $c_{i,{j-1}}$ as though $\ueta$ were an arrow. Conversely, if
		\begin{equation} \tag{$\ast\ast$} \label{eq:WordLeq}
			\ualpha_n\inv\ldots\ualpha_{j+1}\inv < \ualpha_1\ldots\ualpha_{j-1},
		\end{equation}
		then for each $i$, $\eta$ acts on $c_{i,j}$ as though $\ueta$ were a formal inverse. However, note that the action of $\eta$ on $M(\ub,V,\phi_x,\phi_y)$ is still very different to the action of an ordinary arrow or formal inverse in the sense that we additionally have an idempotent action on $c_{i,j}$  in the case of (\ref{eq:WordGeq}) or on $c_{i,j-1}$ in the case of (\ref{eq:WordLeq}).
	\end{rem}
	
	\begin{rem} \label{rem:RelCorrespondence}
		Recall that the partial ordering $>$ is broadly defined (whenever it makes sense to do so) such that $\ubeta\inv > \ualpha$ for arrows $\alpha,\beta \in Q_1 \setminus Q_2$. Thus for a clannish word
		\begin{equation*}
			\uw=\ualpha_1\ldots\ualpha_{j-1} \ueta \ualpha_{j+1} \ldots \ualpha_n
		\end{equation*}
		with $\eta \in Q_2$ and $1<j<n$, we have $\ualpha_n\inv \ldots \ualpha_{j+1}\inv  > \ualpha_1 \ldots \ualpha_{j-1}$ if and only if $\uw = \uw_1 \uw_0 \ueta \uw_0\inv \uw'_1$, where the last symbol of $\uw_1$ is distinct from the first symbol of $\uw'_1$, and where the last symbol of $\uw_1$ is an arrow and the first symbol of $\uw'_1$ is an arrow. This is precisely where the relation $\gg$ comes from in Section~\ref{sec:Hyperstrings}.
	\end{rem}
	
	\begin{rem}
		A key difference between the bands of special biserial algebras and clannish bands is as follows. In the classical special biserial setting, the dimension of a band module is a multiple of the number of symbols of the corresponding band (that is, $\dim_K(M(b,m,\phi)) = mn$ for some band $b$ consisting of $n$ symbols). The same is not necessarily true for clannish band modules. In particular, if $b$ is a symmetric clannish band consisting of $n$ symbols, then $\dim_K(M(b,V,\phi_x,\phi_y)) = \frac{n}{2}\dim_K V$.
	\end{rem}
	
	\begin{thm}[\cite{Clans}, \cite{Deng}, \cite{ClanMaps}]
		The modules $M(\ub, V, \phi_x, \phi_y)$ are indecomposable. Moreover, $M(\ub, V, \phi_x, \phi_y) \cong M(\ub', V', \phi'_x, \phi'_y)$ if and only if $\ub$ is a rotation of $\ub'$ and $(V, \phi_x, \phi_y) \cong (V', \phi'_x, \phi'_y)$ in the category of finite-dimensional representations of $K\langle x,y\rangle / \langle x^2-x, y^2-y \rangle$.
	\end{thm}
	
	\begin{prop}[\cite{ClanMaps}] \label{prop:ClanBandIrred}
		Let $\ub$ be a symmetric clannish band in a clannish algebra $A$ and
		\begin{equation*}
			f \colon (V, \phi_x, \phi_y) \rightarrow (V', \phi'_x, \phi'_y)
		\end{equation*}
		be an irreducible morphism in the category of finite-dimensional representations of \begin{equation*} K\langle x,y\rangle / \langle x^2-x, y^2-y \rangle. \end{equation*} Then there exists a corresponding irreducible morphism
		\begin{equation*}
			f_A \colon M(\ub, V, \phi_x, \phi_y) \rightarrow M(\ub, V', \phi'_x, \phi'_y)
		\end{equation*}
		in $\Mod*A$.
	\end{prop}
	
	\subsection{Rank 2 tubes arising from symmetric clannish bands} \label{sec:ClanBandTubes}
	As per Remark~\ref{rem:clash}, the category of finite-dimensional representations of $K\langle x,y\rangle / \langle x^2-x, y^2-y \rangle$ is equivalent to the full subcategory of $K\Dtilde_m$ ($m>4$) that consists of the two tubes of rank 2, and infinitely many homogeneous tubes indexed by $K \setminus \{0\}$ (see \cite[XIII.2.6.Table]{BlueBookII} for an account). Of the rank 2 tubes in $K\langle x,y\rangle / \langle x^2-x, y^2-y \rangle$, one of these tubes has mouth representations
	\begin{equation*}
		(K,0, 0) \qquad \text{and} \qquad (K,\id_K, \id_K)
	\end{equation*}
	and the other tube has mouth representations
	\begin{equation*}
		(K,\id_K, 0) \qquad \text{and} \qquad (K,0, \id_K).
	\end{equation*}
	The almost split sequences of modules in these tubes are as follows.
	\begin{equation*}
		\xymatrix@1{
			0
				\ar[r]
			& V^{(m)}_-
				\ar[r]^-g
			& V^{(m-1)}_- \oplus V^{(m+1)}_+
				\ar[r]^-f
			& V^{(m)}_+
				\ar[r]
			& 0
		},
	\end{equation*}
	where
	\begin{align*}
		V_+^{(r)}&= \left(K^r,
		\left(\begin{smallmatrix}
			\phi_{x,+}	&	0			&	\cdots	&	0	\\
			0				&	\phi_{x,-}	&	\ddots	&	\vdots		\\
			\vdots			&	\ddots		&	\ddots	&	0	\\
			0				&	\cdots		&	0		&	\phi_{x,\sgn(r)}
		\end{smallmatrix}\right),
		\left(\begin{smallmatrix}
			\phi_{y,+}&	0			&	\cdots	&	0	\\
			\id_K		&	\phi_{y,-}	&	\ddots	&	\vdots		\\
			\vdots		&	\ddots		&	\ddots	&	0	\\
			0			&	\cdots		&	\id_K	&	\phi_{y,\sgn(r)}
		\end{smallmatrix}\right) \right), \\
		V_{-}^{(r)}&= \left(K^r,
		\left(\begin{smallmatrix}
			\phi_{x,-}	&	0				&	\cdots	&	0	\\
			0			&	\phi_{x,+}	&	\ddots	&	\vdots		\\
			\vdots		&	\ddots			&	\ddots	&	0	\\
			0			&	\cdots			&	0		&	\phi_{x,\sgn(r-1)}
		\end{smallmatrix}\right),
		\left(\begin{smallmatrix}
			\phi_{y,-}	&	0				&	\cdots	&	0	\\
			\id_K		&	\phi_{y,+}	&	\ddots	&	\vdots		\\
			\vdots		&	\ddots			&	\ddots	&	0	\\
			0			&	\cdots			&	\id_K	&	\phi_{y,\sgn(r-1)}
		\end{smallmatrix}\right) \right),
	\end{align*}
	\begin{equation*}
		\sgn(r) =
		\begin{cases}
			+	& \text{if } r \text{ is odd}, \\
			-	& \text{if } r \text{ is even},
		\end{cases}
	\end{equation*}
	and $\phi_{x,\pm},\phi_{y,\pm} \in \{\id_K,0\}$ with $\phi_{x,+} \neq \phi_{x,-}$ and $\phi_{y,+} \neq \phi_{y,-}$. The maps $f$ and $g$ are defined such that
	\begin{align*}
		f(v_-^{(1)},\ldots,v_-^{(m-1)},v_+^{(1)},\ldots,v_+^{(m+1)})&=(v_-^{(2)}+v_+^{(1)}, \ldots, v_-^{(m-1)}+v_+^{(m)})	\\
		g(v_-^{(1)},\ldots,v_-^{(m)})&=(v_-^{(1)}, \ldots, v_-^{(m-1)}, 0, v_-^{(1)}, \ldots, v_-^{(m)}).
	\end{align*}
	By Proposition~\ref{prop:ClanBandIrred}, this determines the two tubes of rank 2 corresponding to a symmetric clannish band in a clannish algebra.
	
	\subsection{Direct hyperstrings and symmetric clannish bands} \label{sec:HyperstringBand}
	Recall that given a symmetric special multiserial algebra $A=KQ/I$ corresponding to a Brauer configuration $\chi$, we defined direct hyperstrings $\bw$ in Section~\ref{sec:Hyperstrings} as tuples of strings in $\Wchi$. In this section, we discuss the relationship between direct hyperstrings and symmetric clannish band modules.
	
	Let $\bw=(w_1,\ldots, w_m)$ be a direct hyperstring of $A$, where we write each
	\begin{equation*}
		w_i=\alpha_i\beta_{i,1}\ldots\beta_{i,n_i}\gamma_i.
	\end{equation*}
	For each $i$, let $\widehat{Q}^{w_i}$ be the associated clannish quiver given in Section~\ref{sec:WchiTubes}. Next, rewrite each $w_i$ in the quiver $Q$ as a clannish word
	\begin{equation*}
		\widehat{w}_i= \widehat{h}_{i,1}\widehat{a}_{i}\widehat{b}_{i,1}\ldots\widehat{b}_{i,n_i}\widehat{c}_{i}\widehat{h}_{i,2}
	\end{equation*}
	in the quiver $\widehat{Q}^{w_i}$, where $\widehat{h}_{i,1}$ and $\widehat{h}_{i,2}$ are the (possibly equal) special loops of $\widehat{Q}^{w_i}$.
	
	From this we can define a connected clannish quiver $\widehat{Q}^{\bw}$ by glueing each $\widehat{Q}^{w_i}$ along arrows that correspond to certain common (up to inverse) substrings of each $w_i$. Firstly, if we have $\beta_{i-1,1}\ldots\beta_{i-1,n_{i-1}}$ equal to either $\beta_{i,1}\ldots\beta_{i,n_i}$ or $\beta_{i,n_i}\inv\ldots\beta_{i,1}\inv$ for some $i$, then we identify the whole quivers $\widehat{Q}^{w_{i-1}}=\widehat{Q}^{w_i}$. Otherwise, we must have $w_{i-1} \gg w_i$ and we begin by identifying the special loops $\widehat{h}_{i-1,2} = \widehat{h}_{i,1}$ and the symbols $\widehat{c}_{i-1}=\widehat{a}_{i}\inv$. This corresponds to identifying arrows in the quivers $\widehat{Q}^{w_{i-1}}$ and $\widehat{Q}^{w_{i}}$ in the natural way. Then we recall from property (R2) of the definition of $\gg$ that there exist common (up to inverse) substrings
	\begin{equation*}
		\beta_{i-1,n_{i-1}-j_i}\ldots\beta_{i-1,n_{i-1}} = \beta_{i,j_i+1}\inv\ldots\beta_{i,1}\inv
	\end{equation*}
	of $w_{i-1}$ and $w_i$ respectively, where $j_i$ is taken to be maximal. We thus make the additional identification of the symbols $\widehat{b}_{i-1,n_{i-1}-k}=\widehat{b}_{i,k+1}\inv$ for each $0 \leq k \leq j_i$. Performing this identification process for each $0 < i \leq m$, we obtain a connected clannish quiver $\widehat{Q}^{\bw}$, as required.
	
	Recall from Section~\ref{sec:WchiTubes} that for each quiver $\widehat{Q}^{w_i}$, we obtain a clannish algebra $\widehat{A}_{w_i} = K\widehat{Q}^{w_i} / \widehat{I}^{w_i}$. From this, we obtain a clannish algebra $\widehat{A}_{\bw} = K\widehat{Q}^{\bw} / \widehat{I}^{\bw}$ in the following way. We have $\widehat{h}_{i,1}^2 - \widehat{h}_{i,1}, \widehat{h}_{i,2}^2 - \widehat{h}_{i,2} \in \widehat{I}^{\bw}$ for each $0 < i \leq m$. In addition, if there exists a zero relation $\widehat{b}_{i,k}\widehat{b}_{i,l} \in \widehat{I}^{w_i}$ for some $i$, $k$ and $l$, then we have a zero relation $\widehat{b}_{i,k}\widehat{b}_{i,l} \in \widehat{I}^{\bw}$.
	
	Finally, if there exists a clannish subword $\widehat{b}_{i-1,k-1}\widehat{b}_{i-1,k}$ of $\widehat{w}_{i-1}$ and a clannish subword $\widehat{b}_{i,l}\widehat{b}_{i,l+1}$ of $\widehat{w}_{i}$ such that we have an identification $\widehat{b}_{i-1,k}=\widehat{b}_{i,l}\inv$ but also $\widehat{b}_{i-1,k-1}\neq\widehat{b}_{i,l+1}\inv$, then we note that $\widehat{b}_{i-1,k}$ and $\widehat{b}_{i,l}$ are respectively the first/last common (up to inverse) symbols of $\widehat{w}_{i-1}$ and $\widehat{w}_{i}$ resulting from the glueing procedure. Since these arise from common (up to inverse) substrings in a special multiserial algebra, it follows from property (SM1) of special multiserial algebras that if $\widehat{b}_{i-1,k} \in \widehat{Q}^{\bw}_1$, then we cannot have $\widehat{b}_{i-1,k-1} \in \widehat{Q}^{\bw}_1$ and $\widehat{b}_{i,l+1} \in (\widehat{Q}^{\bw}_1)\inv$ simultaneously. On the other hand, it also follows from properties (W2) and (W3) of $\Wchi$ that we cannot have $\widehat{b}_{i-1,k-1} \in (\widehat{Q}^{\bw}_1)\inv$ and $\widehat{b}_{i,l+1} \in \widehat{Q}^{\bw}_1$ simultaneously. A similar argument holds if $\widehat{b}_{i-1,k} \in (\widehat{Q}^{\bw}_1)\inv$. Thus, we conclude that $\widehat{b}_{i-1,k-1}$ and $\widehat{b}_{i,l+1}$ are either both arrows or both formal inverses. If they are both arrows, then we say $\widehat{b}_{i-1,k-1}\widehat{b}_{i,l+1} \in \widehat{I}^{\bw}$. Otherwise, we say $\widehat{b}_{i,l+1}\inv\widehat{b}_{i-1,k+1}\inv \in \widehat{I}^{\bw}$. This ensures conditions (C2) and (C3) of clannish algebras are met, and hence, that we have a clannish algebra $\widehat{A}_{\bw} = K\widehat{Q}^{\bw} / \widehat{I}^{\bw}$, as required.
	
	Now we can translate direct hyperstrings into symmetric clannish bands. For each clannish word $\widehat{w}_i$, let $\widehat{w}^-_i$ denote the clannish subword with the first special loop removed and let $\widehat{w}^+_i$ denote the clannish subword with the last special loop removed. Define a clannish word $\widehat{\bw}$ of $\widehat{A}_{\bw}$ by
	\begin{equation*}
		\widehat{\bw}=\widehat{w}^+_1\ldots\widehat{w}^+_m(\widehat{w}^-_1\ldots\widehat{w}^-_m)\inv.
	\end{equation*}
	We then note that $\widehat{\bw}$ is either a symmetric clannish band, which we denote by $\widehat{\mathbf{b}}_{\bw}$, or it is a proper power of a clannish word $\widehat{\mathbf{b}}_{\bw}$ that is a symmetric clannish band. In either case, we write $\widehat{\bw} = \widehat{\mathbf{b}}_{\bw}^r$ for some $r \geq 1$. The hyperstring module $M(\bw)$ then corresponds to the symmetric clannish band module $M(\widehat{\mathbf{b}}_{\bw}, K^r, \phi_x, \phi_y)$, where $(K^r, \phi_x, \phi_y)$ is the representation $V_+^{(r)}$ in Section~\ref{sec:ClanBandTubes}.
	
	\subsection{An example} \label{sec:ClanExample}
	Let $Q$ be the quiver with underlying graph
	
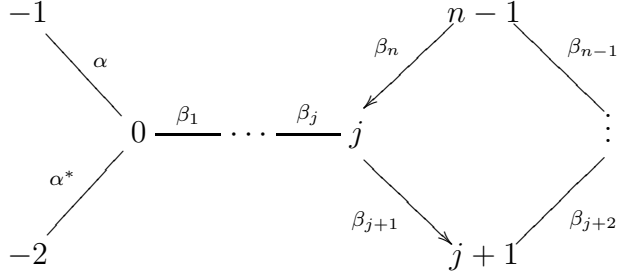
\begin{equation*}
		\xymatrix{
		-1
			\ar@{-}[dr]^-{\alpha}
		&&&& n-1
			\ar@<1ex>@{-}[dr]^-{\beta_{n-1}}
		\\
		& 0
			\ar@{-}[r]^-{\beta_1}
		& \cdots
			\ar@{-}[r]^-{\beta_{j}}
		& j
			\ar@<-1ex>@{->}[dr]_-{\beta_{j+1}}
			\ar@<1ex>@{<-}[ur]^-{\beta_{n}}
		&& \vdots
		\\
		-2
			\ar@{-}[ur]^-{\alpha^\ast}
		&&&& j+1
			\ar@<-1ex>@{-}[ur]_-{\beta_{j+2}}
		}
	\end{equation*}
	where the vertex $0 \in Q_0$ either a source or a sink. Define $I=\langle \beta_{n} \beta_{j+1} \rangle$. The algebra $A=KQ/I$ is a special multiserial algebra that is isomorphic to the clannish (in fact, skewed-gentle in the sense of \cite{Pena} and \cite{BekkertSkewGentle}) algebra $\widehat{A} = K\widehat{Q} / \langle \widehat{\eta}^2 - \widehat{\eta}, \widehat{\beta}_{n} \widehat{\beta}_{j+1} \rangle$, where $\widehat{Q}$ is of the form
	
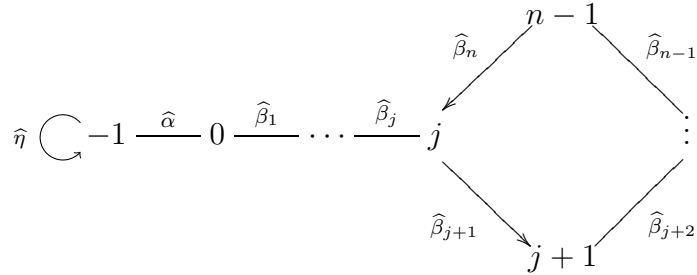
\begin{center}
		\begin{tikzpicture}
			\draw (0,0) node {$
				\xymatrix{
				&&&& n-1
					\ar@<1ex>@{-}[dr]^-{\widehat{\beta}_{n-1}}
				\\
				-1
					\ar@{-}[r]^-{\widehat{\alpha}}
				& 0
					\ar@{-}[r]^-{\widehat{\beta}_1}
				& \cdots
					\ar@{-}[r]^-{\widehat{\beta}_{j}}
				& j
					\ar@<-1ex>@{->}[dr]_-{\widehat{\beta}_{j+1}}
					\ar@<1ex>@{<-}[ur]^-{\widehat{\beta}_{n}}
				&& \vdots
				\\
				&&&& j+1
					\ar@<-1ex>@{-}[ur]_-{\widehat{\beta}_{j+2}}
				}$
			};
			\draw[->] (-4.1,0.2) arc (39.9963:320:0.3);
			\draw (-4.9,0) node {\tiny$\widehat{\eta}$};
		\end{tikzpicture}
	\end{center}
	with the orientation of each non-loop arrow $\widehat{\gamma} \in \widehat{Q}_1$ the same orientation as its corresponding arrow $\gamma \in Q_1$.
	
	$\widehat{A}$ has infinitely many symmetric clannish bands. For example, one can see that we have distinct symmetric clannish bands
	\begin{equation*}
		\ub_i=(\widehat{\uw} \widehat{\ueta})^i (\widehat{\uw}\inv \widehat{\ueta})^i
	\end{equation*}
	for any $i\geq 1$, where $\widehat{w}$ is the (unique up to inverse) string of maximal length without the symbol $\widehat{\eta}$ such that $s(\widehat{w})=t(\widehat{w})=-1$.
	
	To each clannish band $\ub_i$, one obtains two distinct tubes of rank 2 in $\Mod*\widehat{A}$, which are described by the almost split sequences in Section~\ref{sec:ClanBandTubes}. In particular, we have almost split sequences
	\begin{equation*}
		\xymatrix@1{
			0
				\ar[r]
			& M(\ub_i,K,0,0)
				\ar[r]^-{\left(\begin{smallmatrix} 0 \\ \id_K \end{smallmatrix}\right)}
			& M\left(\ub_i,K^2,{\left(\begin{smallmatrix} \id_K & 0 \\ 0 & 0 \end{smallmatrix}\right)},{\left(\begin{smallmatrix} \id_K & 0 \\ \id_K & 0 \end{smallmatrix}\right)}\right)
				\ar[r]^-{\left(\begin{smallmatrix} \id_K & 0 \end{smallmatrix}\right)}
			& M(\ub_i,K,\id_K,\id_K)
				\ar[r]
			& 0
		}
	\end{equation*}
	\begin{equation*}
		\xymatrix@1{
			0
				\ar[r]
			& M(\ub_i,K,\id_K,\id_K)
				\ar[r]^-{\left(\begin{smallmatrix} 0 \\ \id_K \end{smallmatrix}\right)}
			& M\left(\ub_i,K^2,{\left(\begin{smallmatrix} 0 & 0 \\ 0 & \id_K \end{smallmatrix}\right)},{\left(\begin{smallmatrix} 0 & 0 \\ \id_K & \id_K \end{smallmatrix}\right)}\right)
				\ar[r]^-{\left(\begin{smallmatrix} \id_K & 0 \end{smallmatrix}\right)}
			& M(\ub_i,K,0,0)
				\ar[r]
			& 0
		}
	\end{equation*}
	corresponding to the mouth of one tube, and almost split sequences
	\begin{equation*}
		\xymatrix@1{
			0
				\ar[r]
			& M(\ub_i,K,0,\id_K)
				\ar[r]^-{\left(\begin{smallmatrix} 0 \\ \id_K \end{smallmatrix}\right)}
			& M\left(\ub_i,K^2,{\left(\begin{smallmatrix} \id_K & 0 \\ 0 & 0 \end{smallmatrix}\right)},{\left(\begin{smallmatrix} 0 & 0 \\ \id_K & \id_K \end{smallmatrix}\right)}\right)
				\ar[r]^-{\left(\begin{smallmatrix} \id_K & 0 \end{smallmatrix}\right)}
			& M(\ub_i,K,\id_K,0)
				\ar[r]
			& 0
		}
	\end{equation*}
	\begin{equation*}
		\xymatrix@1{
			0
				\ar[r]
			& M(\ub_i,K,\id_K,0)
				\ar[r]^-{\left(\begin{smallmatrix} 0 \\ \id_K \end{smallmatrix}\right)}
			& M\left(\ub_i,K^2,{\left(\begin{smallmatrix} 0 & 0 \\ 0 & \id_K \end{smallmatrix}\right)},{\left(\begin{smallmatrix} \id_K & 0 \\ \id_K & 0 \end{smallmatrix}\right)}\right)
				\ar[r]^-{\left(\begin{smallmatrix} \id_K & 0 \end{smallmatrix}\right)}
			& M(\ub_i,K,0,\id_K)
				\ar[r]
			& 0
		}
	\end{equation*}
	corresponding to the mouth of the other tube.
	
	One can describe the almost split sequences of the tubes of rank 2 in $\Mod* A$ by using the equivalence of categories $F\colon \Mod*\widehat{A} \rightarrow \Mod*A$ defined such that for any $M \in \Mod*\widehat{A}$, we have $FM = M$ as vector spaces. The idempotent actions on $FM$ are defined such that
	\begin{align*}
		FM \varepsilon_{-1} &= \im \widehat{\eta} \\
		FM \varepsilon_{-2} &= \Ker \widehat{\eta} \cong \Coker \widehat{\eta} \\
		FM \varepsilon_{i} &= M\widehat{\varepsilon}_{i}
	\end{align*}
	for each vertex $i \in Q_0$. Let $\theta_{-1}$ and $\theta_{-2}$ be the canonical inclusion of $\im \widehat{\eta}$ and $\Ker \widehat{\eta}$ into $M\widehat{\varepsilon}_{i}$ respectively, and let $\psi_{-1}$ and $\psi_{-2}$ be the canonical surjection of $M\widehat{\varepsilon}_{i}$ onto $\im \widehat{\eta}$ and $\Coker \widehat{\eta}$ respectively. Then for any $m \in FM$
	\begin{align*}
		m\alpha &=
		\begin{cases}
			\theta_{-1}(m) \widehat{\alpha}	&	\text{if } 0 \text{ is a sink,} \\
			\psi_{-1}(m \widehat{\alpha})		&	\text{if } 0 \text{ is a source,}
		\end{cases} \\
		m\alpha^\ast &=
		\begin{cases}
			\theta_{-2}(m) \widehat{\alpha}	&	\text{if } 0 \text{ is a sink,} \\
			\psi_{-2}(m \widehat{\alpha})		&	\text{if } 0 \text{ is a source,}
		\end{cases} \\
		m \gamma &= m \widehat{\gamma} \text{ for any } \gamma\neq\alpha,\alpha^\ast,
	\end{align*}
	where $m \widehat{\delta}$ is given by the action $\widehat{\delta} \in \widehat{A}$ on $M$. Finally, we define $Ff=f$ as a linear map for any morphism $f \in \Mod* \widehat{A}$.
	
	As an example, consider the clannish band $\ub_1$. Then it is easy to see that $FM(\ub_1, K, \id_K, \id_K) \cong M(w)$ and $FM(\ub_1, K, 0, 0) \cong M(w')$ for the strings $w,w' \in \str_A$ of maximal length such that $s(w)=t(w)=-1$ and $s(w')=t(w')=-2$. It is also straightforward to show that
	\begin{equation*}
		FM\left(\ub_1,K^2,{\left(\begin{smallmatrix} \id_K & 0 \\ 0 & 0 \end{smallmatrix}\right)},{\left(\begin{smallmatrix} \id_K & 0 \\ \id_K & 0 \end{smallmatrix}\right)}\right) \cong M(\bw_1)
	\end{equation*}
	for the hyperstring $\bw_1=(w, (w')\inv)$. Thus, we have an almost split sequence
	\begin{equation*}
		0 \rightarrow M(w') \rightarrow M(\bw_1) \rightarrow M(w) \rightarrow 0
	\end{equation*}
	in $\Mod*A$. We remark that $\bw_1$ in this section corresponds to the direct hyperstring $\bw_4$ in Section~\ref{sec:ExactSequences}.
	
	On the other hand, considering the clannish band $\ub_2$, one can show via Remark~\ref{rem:RelCorrespondence} that $FM(\ub_2, K, \id_K, 0) \cong M(\bw_2)$, where $\bw_2 = (w, w')$. We remark that $\bw_2$ in this section corresponds to the direct hyperstring $\bw_3$ in Section~\ref{sec:ExactSequences}.
	
	More generally, the symmetric clannish band $b_i$ corresponds to the direct hyperstring $(w,w',w,w',\ldots)$ consisting of $i$ terms.
	
	\section{Glossary of Notation} \label{ap:Glossary}
	\begin{longtable}{l | l}
		$K$	&	algebraically closed field	\\
		$Q$	&	finite quiver	\\
		$Q_0$	&	set of vertices of $Q$ \\
		$Q_1$	&	set of arrows of $Q$ \\
		$Q^{-1}_1$	&	set of formal inverses of $Q$ \\
		$Q_2$	&	set of special loops of a clannish algebra \\
		$I$		&	admissible set of relations for a quiver	\\
		$\varepsilon_x$	& stationary path at a vertex $x$ in a quiver \\
		$_s \Gamma_A$	& stable Auslander-Reiten quiver of the module category of an algebra $A$ \\
		$\tau$	&	Auslander-Reiten translate \\
		$\Omega(M)$	&	syzygy of a module $M$ \\
		$\chi$	&	Brauer configuration	\\
		$\Vchi$	& set of all vertices in $\chi$ \\
		$\Gchi$	& set of all germs of polygons in $\chi$	\\
		$\Pchi$	&	set of all polygons in $\chi$ \\
		$\kappa(g)$	&	vertex associated to a germ $g$	\\
		$\pi(g)$	&	polygon associated to a germ $g$	\\
		$\mult(v)$	&	multiplicity of a vertex $v$	\\
		$\sigma(g)$	&	the germ following $g$ in the cyclic ordering of a Brauer configuration \\
		$\cycle_v$		&	the cycle of arrows in $Q$ induced by a vertex $v \in \Vchi$ \\
		$\cycle_{v,\alpha}$	&	the rotation of $\cycle_v$ such that the first arrow is $\alpha$ \\
		$s(\alpha)$	&	polygon in $\Pchi$ (or vertex in $Q_0$) at the source of symbol $\alpha \in Q_1 \cup Q^{-1}_1$ \\
		$t(\alpha)$	&	polygon in $\Pchi$ (or vertex in $Q_0$) at the target of symbol $\alpha \in Q_1 \cup Q^{-1}_1$ \\
		$\widehat{s}(\alpha)$	&	germ in $\Gchi$ at the source of symbol $\alpha \in Q_1 \cup Q^{-1}_1$ \\
		$\widehat{t}(\alpha)$	&	germ in $\Gchi$ at the target of symbol $\alpha \in Q_1 \cup Q^{-1}_1$ \\
		$\ap(w,w')$	&	set of admissible pairs between strings $w$ and $w'$ \\
		$\sap(w,w')$	&	subset of $\ap(w,w')$ corresponding to stable homomorphisms \\
		$S(x)$	&	simple module associated to a polygon $x$ \\
		$P(x)$	&	indecomposable projective-injective associated to a polygon $x$ \\
		$\str_A$	&	set of all strings associated to an algebra $A=KQ/I$ \\
		$\band_A$	&	set of all bands associated to an algebra $A=KQ/I$ \\
		$\hstep$	&	set of all Green hyperwalk steps	 in $\chi$\\
		$\Mchi$	&	set of modules corresponding to Green hyperwalk steps in $\chi$ \\
		$\omega$	&	the bijection between $\Mchi$ and $\hstep$ \\
		$\Dchi$	&	set of all $\mathbb{D}$-triples in $\chi$	\\
		$\Wchi$	&	set of strings between $\mathbb{D}$-triples in $\chi$ corresponding to mouth modules \\ & of additional tubes of rank 2
	\end{longtable}
	
	\bibliography{Bibliography}{}
	\bibliographystyle{habbrv}
\end{document}